\documentclass{amsart}
\usepackage[margin=1.2in]{geometry}

\usepackage{xcolor}
\usepackage{url}
\usepackage{amssymb}
\usepackage{amsmath}
\usepackage{amsthm}
\usepackage{amsfonts}
\usepackage{graphicx}
\usepackage{wrapfig}
\usepackage{enumerate}
\usepackage{listings}
\usepackage{dsfont}
\usepackage{mathtools}
\usepackage{tikz}
\usepackage{tcolorbox}

\newtheorem{theorem}{Theorem}[section]
\newtheorem{corollary}[theorem]{Corollary}
\newtheorem{lemma}[theorem]{Lemma}
\newtheorem{proposition}[theorem]{Proposition}
\newtheorem{definition}[theorem]{Definition}
\newtheorem{example}[theorem]{Example}

\newtheorem{remark}[theorem]{Remark}

\newcommand{\R}{\mathbb{R}}
\newcommand{\N}{\mathbb{N}}
\newcommand{\AW}{\mathcal{AW}}
\renewcommand{\epsilon}{\varepsilon}

\title[A DDP for multiperiod bicausal optimization problems]{A dynamic programming principle for multiperiod control problems with bicausal constraints}

\author{Ruslan Mirmominov}
\address{Ruslan Mirmominov\newline
\hbox{}\hspace{0.33cm} Carnegie Mellon University\newline
\hbox{}\hspace{0.33cm} Department of Mathematics}
\email{rmirmomi@andrew.cmu.edu}

\author{Johannes Wiesel}
\address{Johannes Wiesel\newline
\hbox{}\hspace{0.33cm} Carnegie Mellon University\newline
\hbox{}\hspace{0.33cm} Department of Mathematics}
\email{wiesel@cmu.edu}


\thanks{JW acknowledges support by NSF Grant DMS-2345556 and would like to thank the Sydney Mathematical Research Institute for its hospitality. JW would also like to thank Martin Larsson for helpful discussions.}

\begin{document}
\maketitle

\begin{abstract}
We consider multiperiod stochastic control problems with non-parametric uncertainty on the underlying probabilistic model. We derive a new metric on the space of probability measures, called the adapted $(p, \infty)$--Wasserstein distance $\AW_p^\infty$ with the following properties: (1) the adapted $(p, \infty)$--Wasserstein distance generates a topology that guarantees continuity of stochastic control problems and (2) the corresponding $\AW_p^\infty$-distributionally robust optimization (DRO) problem can be computed via a dynamic programming principle involving one-step Wasserstein-DRO problems.
If the cost function is semi-separable, then we further show that a minimax theorem holds, even though balls with respect to $\AW_p^\infty$ are neither convex nor compact in general. We also derive first-order sensitivity results. 
\end{abstract}

\section{Introduction}

Distributionally robust optimization (DRO) problems of the type 
\begin{align} \label{eq:dro_intro}
V^{\mathcal{W}_p}(\delta):=\inf_{\alpha\in \mathcal{A}} \sup_{\nu \in B^{\mathcal{W}_p}_\delta(\mu)} \int f(x, \alpha)\,\nu(dx)
\end{align}
have recently gained popularity in operations research, mathematical finance, statistics and many other fields. In \eqref{eq:dro_intro}, the probability measure $\mu$ on $\R^d$ is considered a benchmark, which could either be derived from an idealized model as common in mathematical finance, or from historical observations as often assumed in machine learning or statistics, and we aim to minimize the expected value of a cost function $f$ over a set of controls $\mathcal{A}$. To account for uncertainty in the choice of $\mu$, \eqref{eq:dro_intro} formalises a worst-case approach: instead of merely considering the probabilistic dynamics under $\mu$, nature is permitted to choose perturbations $\nu$ in a $p$-Wasserstein ball $B_\delta ^{\mathcal{W}_p}(\nu)$ centered at $\mu$ (see \eqref{eq:wass} for a formal definition of the Wasserstein distance $\mathcal{W}_p$). For specific choices of $f$ and $\mathcal{A}$, $V^{\mathcal{W}_p}(\delta)$ captures robust versions of (one-period) option pricing models, optimal investment problems and risk measures classically studied in mathematical finance, as well as linear regression or  training of neural networks in machine learning and statistics; we refer to \cite{bartl2021sensitivity} for a more detailed analysis of these exemplary applications. In the last couple of years, many important contributions in the study of $V^{\mathcal{W}_p}(\delta)$ have been made: we refer to \cite{blanchet2019quantifying, gao2023distributionally, bartl2020computational, mohajerin2018data} for dual representations, to \cite{bartl2021sensitivity, blanchet2023statistical} for first-order approximations and to \cite{kuhn2019wasserstein, blanchet2019robust, olea2022out, gao2023finite} and the references therein for applications to machine learning. 

While the DRO problem \eqref{eq:dro_intro} is thus well understood as way to formalize model uncertainty on $\R^d$, the situation changes a lot if $\mu$ is considered to be the law of an $N$-step stochastic process $X=(X_t)_{t=1}^N$ in its natural filtration. Contrary to the static case \eqref{eq:dro_intro} discussed above, it turns out that Wasserstein balls $B^{\mathcal{W}_p}_\delta(\mu)$ around $\mu$ are \emph{not} a suitable model for model uncertainty in this time-dynamic setting anymore. In fact, if $\mathcal{A}$ is the set of predictable processes (as would be the case for utility maximization problems), then the functional 
\begin{align}\label{eq:nu}
\nu \mapsto \inf_{\alpha\in \mathcal{A}} \int f(x, \alpha)\,\nu(dx)
\end{align}
is not continuous with respect to $\mathcal{W}_p$ so that in general $V^{\mathcal{W}_p}(\delta)\downarrow V^{\mathcal{W}_p}(0)$ for $\delta\downarrow 0$ does not hold. In essence this is due to the fact, that any distance compatible with weak convergence does not control the natural filtration of the process $X$ and consequently its set of admissible controls $\mathcal{A}$; we refer to \cite{backhoff2020adapted} for a well-written explanation of this discontinuity effect, illustrated with a simple two-step stochastic process. In consequence, the interpretation of $V^{\mathcal{W}_p}(\delta)$ as a small perturbation of $V^{\mathcal{W}_p}(0)$ is not justified any more, and it is necessary to consider stronger topologies to define feasible versions of \eqref{eq:dro_intro} for laws of stochastic processes.

Recently the \emph{adapted Wasserstein topology} --- as one canonical choice of such a topology --- has received a lot of attention. In fact it is the coarsest topology, which makes  optimal stopping problems continuous \cite{backhoff2020all}. The adapted Wasserstein topology has been rediscovered many times throughout different disciplines; we refer e.g., to \cite{aldous1979extended, hoover1984adapted, ruschendorf1985wasserstein, pflug2012distance, pflug2014multistage, lassalle2018causal, bartl2021wasserstein, bonnier2023adapted} and the references therein.
A suitable metric for this topology is the so-called \emph{adapted Wassersten distance} $\mathcal{AW}_p$ (see Definition \ref{def:aw} below). Similarly to $\mathcal{W}_p$, $\mathcal{AW}_p$ is defined as an optimal transport (OT) problem and under certain regularity assumptions, the functional \eqref{eq:nu} is in fact Lipschitz-continuous with respect to $\AW_p$ \cite{backhoff2020adapted}. It thus seems reasonable to formulate the $\AW_p$-DRO problem 
\begin{align}\label{eq:dro2_intro}
    V^{\mathcal{AW}_p}(\delta) := \inf_{\alpha\in \mathcal{A}} \sup_{\nu \in B^{\mathcal{AW}_p}_\delta(\mu)} \int f(x, \alpha)\,\nu(dx)
\end{align}
as a natural multiperiod counterpart to \eqref{eq:dro_intro}. While theoretically appealing, the structure of $V^{\mathcal{AW}_p}(\delta)$ is significantly more involved than $V^{\mathcal{W}_p}$. In fact, $\mathcal{AW}_p$ is a \emph{nested} optimization problem and contrary to the Wasserstein balls $B^{\mathcal{W}_p}_\delta(\mu)$, $\AW_p$-balls are neither convex nor closed in general. These facts make a numerical exploration of $B^{\mathcal{AW}_p}_\delta(\mu)$ difficult; in consequence, $V^{\mathcal{\AW}}(\delta)$ is hard to compute. While this issue seems to be well-known, to the best of our knowledge it has not been addressed at this level of generality in the literature so far: \cite{jiang2024duality, han2022distributionally, aroradata} derive a Langragian representation of $V^{\mathcal{AW}_p}(\delta)$ under specific assumptions on the function $f$, while \cite{bartl2023sensitivity} give a first-order approximation of the corresponding $\sup$-$\inf$-problem; however due to non-convexity it is not clear if this problem equals $V^\mathcal{AW}(\delta)$ in general. We also refer to \cite{jiang2024sensitivity} for a recent study of sensitivities of an uncontrolled DRO problem with causal constraints in discrete and continuous time, as well as for the specific case of a causal DRO problem with martingale constraints to \cite{sauldubois2024first}.

On the other hand, computation of DRO problems is a well studied problem in the optimization literature (see e.g., \cite{rahimian2019distributionally} for an overview), and is classically achieved through a \emph{dynamic programming principle} (DPP), which breaks up the multiperiod problem $V^{\mathcal{AW}_p}(\delta)$ into its one-step counterparts. Existence of a DPP is strongly connected to a \emph{rectangularity} property of the underlying sets of probabilistic models considered; see e.g., \cite{epstein2003recursive, iyengar2005robust, shapiro2016rectangular, wiesemann2013robust, shapiro2021distributionally, wang2023foundation} and the references therein. In this context, the seminal works \cite{epstein2003recursive, iyengar2005robust} give specific constructions for sets of probability measures which allow for a DPP reformulation. A similar approach is taken for utility maximization under non-parametric uncertainty in discrete time \cite{nutz2016utility}, which also gives a construction of the set of probability measures considered; see also \cite{carassus2019robust, carassus2023discrete, bayraktar2017arbitrage, blanchard2018multiple, neufeld2018robust} for extensions of this framework. On the other hand, \cite{shapiro2016rectangular} goes one step further, and simply defines a set of measures to be rectangular, if a DPP holds. Inspired by this, we aim to answer the following question in this article:

\begin{tcolorbox}\label{color}
Can we identify a (close) variant of $\AW_p$, that makes balls around $\mu$ rectangular, i.e., that allows for a dynamic programming principle with one-step DRO problems of the type $V^{\mathcal{W}_p}(\delta)$?
\end{tcolorbox}

While the recursive structure of the adapted Wasserstein distance $\AW_p$ (see e.g., \cite[Chapter 2]{pflug2014multistage}) might initially suggest that it is in fact already possible to formulate a DPP for $V^{\mathcal{AW}_p}(\delta)$, a closer analysis of the balls $B^{\mathcal{AW}_p}_\delta(\mu)$ and their nested $L^p$-structure quickly reveals that this is not achievable (see Section \ref{sec:aw_inf} for a more detailed discussion). In this note we resolve this issue by deriving a new distance that satisfies the above abstract rectangularity condition. We call this new distance the \emph{adapted $(p,\infty)$--Wasserstein metric} $\mathcal{AW}_p^\infty$ (see Definition \ref{def:aw_inf} below) and show that $\AW_p \lesssim \AW_p^\infty \lesssim \AW_\infty$. In this sense, $\AW_p^\infty$ can be seen as a natural interpolation between adapted Wasserstein metrics. Denoting the corresponding $\AW_p^\infty$--DRO problem by
\begin{align*}
    V(\delta) := \inf_{\alpha\in \mathcal{A}} \sup_{\nu \in B_\delta(\mu)} \int f(x, \alpha)\,\nu(dx),
\end{align*}
where $B_\delta(\mu):= B_\delta^{\AW_p^\infty}(\mu)$ is  a ball of radius $\delta$ around $\mu$ in $\AW_p^\infty$-distance and $\mathcal{A}$ are predictable controls taking values in the compact set $K^N$, our main contributions can be informally summarized as follows: define $V^\delta_N: =f $ and
\begin{align}
V_t^\delta(x_{1:t}, y_{1:t}, \alpha_{1:t})= \inf_{\alpha_{t+1}\in K} \sup_{\gamma^{t+1}\in \Pi_\delta(\mu_{x_{1:t}}, \cdot)} \int V_{t+1}^\delta(x_{1:t+1}, y_{1:t+1}, \alpha_{1:t+1})\,\gamma^{t+1}(dx_{t+1},dy_{t+1})
\end{align}
for $t=N-1, \dots, 0$, where $\Pi_\delta(\mu_{x_{1:t}}, \cdot)$ is the set of couplings $\pi$ on $\R\times \R$ with first marginal $\mu_{x_{1:t}}(\cdot)=\mu(\cdot|X_{1:t}=x_{1:t})$ and $L^p$-cost at most $\delta.$ Theorem \ref{equiv:control_open} below states that the DPP $V(\delta)=V^\delta_0$ holds under mild regularity assumptions; in particular there is no need to assume that $\alpha\mapsto f(x,\alpha)$ is convex. In conclusion, the $\AW_p^\infty$-DRO problem $V(\delta)$ captures the best of two worlds: on the one hand the topology induced by $\AW_p^\infty$ makes \eqref{eq:nu} continuous, while $V(\delta)$ allows for an elegant reformulation of well-studied one-step Wasserstein distributionally robust optimization problems on the other hand.

Next to the derivation of a DPP, efficient computation of DRO problems classically relies on so-called \emph{minimax theorems} \cite{von1928theorie, sion1958general, fan1953minimax}, asserting that the supremum and the infimum in $V(\delta)$ can be interchanged without altering the value of the optimization problem. These minimax problems usually rely on convexity and compactness of the sets $\mathcal{A}$ and $B_\delta(\mu)$. Even though such a convexity and compacity property is not satisfied for the $\AW^\infty_p$-balls $B_\delta(\mu)$ in general, a minimax theorem can still be derived: we show in Corollary \ref{cor:minimax} that the representation
\begin{align*}
V(\delta)=  \sup_{\nu \in B_\delta(\mu)} \inf_{\alpha\in \mathcal{A}} \int f(x, \alpha)\,\nu(dx)
\end{align*}
holds, as soon as $f$ is semi-separable and convex in the control variable. To the best of our knowledge, this is the first minimax theorem for bicausal DRO problems. 

Let us emphasize that $V(\delta)$ allows for an easily interpretable DPP without sacrificing structural results already established for $V^{\mathcal{AW}_p}(\delta)$. In fact, following the approach in \cite{bartl2023sensitivity} we can still compute first order-approximations 
$$V(\delta)\approx V(0)+\delta\cdot \Upsilon +o(\delta),$$
where $\Upsilon$ depends only on $f$ and $\mu$; 
see Theorem \ref{dro:control_sensitivity}. Finally, in order to pave the way for potential applications of our results in robust option pricing, we also extend our study to the case where the plausible models $\nu\in B_\delta(\mu)$ are restricted to be martingale measures.\\

\textbf{Notation.}
Throughout this note, we take $p,q\in [1,\infty)$, such that $1/p + 1/q = 1$, and fix $N\in \mathbb N$. For $k\in \N$ we equip $\R^k$ with the Euclidean norm $|\cdot|$ and often consider the vectors $x_{s:t}:= (x_s, \dots, x_t)$, where $x=(x_1, x_2, \dots, x_k)\in \R^k$ and $1\le s\le t\le k.$ We also define $\|x\|:=(\sum_{t=1}^k |x_t|^p)^{1/p}.$ We use the notation $x\twoheadrightarrow F(x)$ for correspondences (i.e., set-valued functions).
We write $\mathcal{P}(\R^k)$  for the set of Borel probability measures on $\R^k$ and similarly $\mathcal{P}_p(\R^k):= \{\nu\in \mathcal{P}(\R^k): \int \|x\|^p\,\nu(dx)<\infty\}$. The Lebesgue measure on $\R^k$ will be denoted by $\text{Leb}$, while we use $\text{spt}(\mu)$ for the support of a measure $\mu\in \mathcal{P}(\R^k)$.  In order to shorten notation, we write $\nu(A):=\nu(\{x \in \R^k: x\in A\})$ if there is no confusion, and write spt$(\nu)$ for the support of $\nu$. We define the push-forward measures $f_{\#}\nu$ via the relation $f_{\#}\nu(A):=\nu(\{x\in \R^k: f(x)\in A\})$ for a Borel measurable function $f:\R^k\to \R^\ell$ and all $A\in \mathcal{B}(\R^\ell)$, $\ell \ge 1$; in particular we consider the projections $\operatorname{proj}^{s:t}:\R^k\to \R^{t-s+1}$ given by $\operatorname{proj}^{s:t}(x)=x_{s:t}$ for $1\le s\le t\le k$ and write $\nu^{s:t}:= (\operatorname{proj}^{s:t})_{\#}\nu.$ For measures $\mu,\nu \in \mathcal{P}(\R^k)$, the set $\Pi(\mu,\nu)=\{\pi\in \mathcal{P}(\R^k\times \R^k):\, \pi(\cdot \times \R^k)=\mu(\cdot), \pi(\R^k\times\cdot)=\nu(\cdot)\}$ is called the set of transport plans between $\mu$ and $\nu$. We define $\operatorname{proj}^{i, j}(x, y) := (x_i, y_j)$ for $1 \leq i, j \leq k$ and often consider $(\operatorname{proj}^{i, j})_{\#}\pi$ for $\pi \in \mathcal{P}(\mathbb{R}^k \times \mathbb{R}^k)$. We disintegrate measures $\pi\in \mathcal{P}(\R^{k+\ell})$ and write $\pi(A\times B)= \int_A \pi_{x_{1:k}}(B)\, \pi^{1:k}(dx_{1:k})$ for all $(A,B)\in \mathcal{B}(\R^k) \times \mathcal{B}(\R^\ell)$, where $\pi_{x_{1:k}}(\cdot)=\pi(x_{k+1:k+\ell}\in \cdot|X_{1:k}=x_{1:k})$ is a conditional probability distribution. Additionally we set $\pi_{x_{1:0}} := \pi^1$. For two measures $\pi\in \Pi(\mu,\nu),\gamma\in \Pi(\nu, \eta)$ we denote the gluing of $\pi$ and $\gamma$ according to \cite[Gluing lemma, p.12]{villani2009optimal} by $\pi\dot{\oplus} \gamma$. Lastly, we often abbreviate the inequality $a\le C \cdot b$ for some constant $C>0$ by $a\lesssim b$.\\

\textbf{Organization of the paper.} We give basic definitions related to optimal transport and bicausality in Section \ref{sec:ot}. We then define the adapted $(p,\infty)$--Wasserstein distance in Section \ref{sec:aw_inf} and discuss its basic properties. In Section \ref{sec:dpp} we derive the dynamic programming principle for $V(\delta)$ as well as the minimax theorem. Lastly, we compute first-order sensitivies of $V(\delta)$ in Section \ref{sec:sens}. Section \ref{sec:proofs} collects all remaining proofs. 

\section{Optimal transport, bicausality and the adapted Wasserstein distance} \label{sec:ot}

Throughout this article we fix numbers $p\ge 1$ and $N\in \N$. We think of a Borel probability measure $\mu\in \mathcal{P}_p(\R^N)$ as the law of discrete-time stochastic process $X=\{X_t\}_{t=1}^N$ with finite $p$th moment on the canonical space $(\R^N, \mathcal{B}(\R^N))$ with its natural filtration. Given $\mu,\nu\in \mathcal{P}(\R^N)$, we denote the set of its couplings by $\Pi(\mu,\nu)$; in other words $\Pi(\mu,\nu)$ is the set of joint distributions of $\mu$ and $\nu$ on the product space $\R^N\times\R^N$. A well-known metric on the space of probability measures is the so-called \emph{$p$--Wasserstein distance} given by
\begin{equation}\label{eq:wass}
    \mathcal{W}_p(\mu, \nu)^p := \inf_{\gamma \in \Pi(\mu, \nu)} \int \|x - y\|^p \,\gamma(dx, dy),
\end{equation}
where $\|x\|:=(\sum_{t=1}^N |x_t|^p)^{1/p}.$ The optimization problem \eqref{eq:wass} is called an optimal transport problem with cost function $\|\cdot\|^p$; we refer e.g., to \cite{villani2009optimal, santambrogio2015optimal} for a historical overview of such problems, as well as theoretical background.

Throughout this article we are interested in a specific subset of $\Pi(\mu,\nu)$ as stated in the following definition.

\begin{definition}\label{def:bicausal}
Let \(\gamma \in \Pi(\mu, \nu)\) be a transport plan for \(\mu,\nu \in \mathcal{P}_p(\mathbb{R}^N)\). Then $\gamma$ is called \emph{causal} if 
\begin{align}\label{eq:causal}
\gamma_ {x_{1:N}}(y_{1:t}\in A) = \gamma_{x_{1:t}}(y_{1:t}\in A )    
\end{align}
for $A\in \mathcal{B}(\R^t)$ and all $t=1, \dots, N$. It is called \emph{bicausal} if \eqref{eq:causal} additionally holds with the roles of $x$ and $y$ reversed. We write $\Pi_{\operatorname{bc}}(\mu,\nu)$ for the set of bicausal transport plans.
\end{definition}

Equivalently to the above, $\gamma$ is bicausal if 
\begin{align}\label{eq:disintegration}
\gamma = \gamma^1 \otimes \gamma_{x_1, y_1} \otimes \ldots \otimes \gamma_{x_{1:N-1}, y_{1:N-1}}, \quad \mu = \mu^1 \otimes \mu_{x_1} \otimes \ldots \otimes \mu_{x_{1:N-1}}, \quad   \nu = \nu^1 \otimes \nu_{y_1} \otimes \ldots \otimes \nu_{y_{1:N-1}}
\end{align}
satisfy $\gamma^1\in \Pi(\mu^1, \nu^1)$ and \(\gamma_{x_{1:t}, y_{1:t}} \in \Pi(\mu_{x_{1:t}}, \nu_{y_{1:t}})\) for $t= 1, \ldots, N- 1$. Here \eqref{eq:disintegration} is short for the disintegration rule 
\[
\gamma(B_1 \times B_2 \times \ldots \times B_N) = \int_{B_1} \int_{B_2} \ldots \int_{B_N} \gamma_{x_{1:N-1}, y_{1:N-1}}(dx_N, dy_N) \ldots \gamma_{x_1, y_1}(dx_2,dy_2) \, \gamma^1(dx_1,dy_1)
\]
for any \(B_1, B_2, \ldots, B_N \in \mathcal{B}(\R^2)\),
and similarly for $\mu,\nu$. The key property of a bicausal plan $\gamma\in \Pi_{\operatorname{bc}}(\mu,\nu)$ is thus its non-anticipativity: at a given time $t$, it only ``sees" the conditional laws $\mu_{x_{1:t}}, \nu_{y_{1:t}}$ instead of the unconditional distributions.

Definition \ref{def:bicausal} gives rise to the so-called adapted Wasserstein distance.

\begin{definition}\label{def:aw}
For \(\mu, \nu \in \mathcal{P}_p(\mathbb{R}^N)\) the adapted Wasserstein distance is defined as
\[
\mathcal{AW}_{p}(\mu, \nu)^p := \inf_{\gamma \in \Pi_{\operatorname{bc}}(\mu, \nu)} \int \|x - y\|^p \gamma(dx, dy).
\]
\end{definition}

The adapted Wasserstein distance and the concept of (bi-)causality have been rediscovered many times throughout the disciplines. We refer to \cite{backhoff2020all} for a well-written overview and comparison of different concepts related to $\mathcal{AW}_p$. Most importantly, $\mathcal{AW}_{p}$ generates the coarsest topology which makes filtration-dependent optimization problems like optimal stopping continuous; see \cite[Theorem 1.3]{backhoff2020all}. To showcase the difference between $\mathcal{W}_p$ and $\mathcal{AW}_p$ we give the following example, which goes back at least to \cite{backhoff2020adapted}.

\begin{example}
Let $$\mu_n=\frac{1}{2}\delta_{(\epsilon_n, 1)}+\frac{1}{2}\delta_{(-\epsilon_n, -1)}\quad \text{and} \quad \mu=\frac{1}{2}\delta_{(0,1)}+\frac{1}{2}\delta_{(0,-1)}.$$ If $\epsilon_n\to 0$, then $\mathcal{W}_p(\mu, \mu_n)\to 0$ and $\mathcal{AW}_p(\mu, \mu_n)=(\epsilon^p_n+2^{p-1})^{1/p}\to 2^{(p-1)/p}>0$. 
\end{example}

\section{The adapted $(p,\infty)$--Wasserstein distance $\mathcal{AW}_p^\infty$} \label{sec:aw_inf}

As the adapted Wasserstein distance itself can be computed via a dynamic programming formulation (see e.g., \cite[Chapter 2]{pflug2014multistage}), it is natural to expect a similar result for $\mathcal{AW}_p$-distributionally robust optimization (DRO) problems of the form 
\begin{align*}
    V^{\AW_p}(\delta)=\inf_{\alpha \in \mathcal{A}} \sup_{\nu\in B^{\AW_p}_\delta(\mu)} \int f(y, \alpha)\,\nu(dy)
\end{align*}
defined in the Introduction. More precisely, given that $\mathcal{AW}_p(\mu,\nu)$ is essentially a Wasserstein distance between the conditional one-step kernels $\mu_{x_{1:t}}$ and $\nu_{x_{1:t}}$,
one would hope for $V^\delta_0=V^{\AW_p}(\delta)$, where one formally defines the iteration $V^\delta_N:=f$ and
\begin{align}\label{eq:easy_dpp}
    V_t^\delta(x_{1:t}, y_{1:t}):= \sup_{\gamma^{t+1}\in \Pi( \mu_{x_{1:t}}, \cdot): \int |x-y|^p\,d\gamma^{t+1}\le \delta^p } \int V^\delta_{t+1}(x_{1:t+1}, y_{1:t+1})\,\gamma^{t+1}(dx_{t+1}, dy_{t+1})
\end{align}
for $t=N-1, \dots, 0.$ Somewhat anticlimactically, a closer inspection of $V^{\AW_p}(\delta)$ and in particular of the balls $B^\AW_\delta(\mu)$ shows that the hope for a representation of type $V^\delta_0=V^{\AW_p}(\delta)$ is unfounded. This has been observed in previous literature, and \cite{jiang2024duality, han2022distributionally, aroradata} offer a remedy to this problem via a Lagrangian approach. While mathematically concise, their formulation does not reduce $V^{\AW_p}(\delta)$ to one-step DRO problems and the interpretation of the intermediate DPP steps is less obvious.

In this article we thus turn the problem of finding a ``nice" DPP for the adapted weak topology on its head: we define a metric, which is stronger than $\AW_p$ and whose balls are rectangular. We call it the adapted $(p,\infty)$--Wasserstein distance.

\begin{definition}\label{def:aw_inf}
For $t=0,1, \dots, N$ we recursively define the functional \(\mathcal{F}_{t, p}: \mathcal{P}_p(\mathbb{R}^{N-t} \times \mathbb{R}^{N-t}) \to \mathbb{R}\) via \(\mathcal{F}_{N, p} := 0\) and
\begin{align*}
\mathcal{F}_{t, p}(\gamma) := \mathcal{C}_p(\gamma^1) \vee \|\mathcal{F}_{t+1,p}( \bar{\gamma}_{x_{1}, y_{1}})\|_{L^\infty(\gamma^1)} \;\; \text{for} \;\; \gamma \in \mathcal{P}_p(\mathbb{R}^{N-t} \times \mathbb{R}^{N-t}),
\end{align*}
where 
\begin{align*}
\mathcal{C}_p(\pi):= \Big( \int |x_1-y_1|^p\,\pi(dx_1,dy_1) \Big)^{1/p} \qquad \forall \pi\in \mathcal{P}(\R^2),
\end{align*}
and
\begin{align*}
 \bar{\gamma}_{x_1, y_1}(dx_{2:s}, dy_{2:s}):= \gamma_{x_{1:s-1}, y_{1:s-1}}(dx_s, dy_s)\dots\gamma_{x_{1:2}, y_{1:2}}(dx_3, dy_3)\gamma_{x_1, y_1}(dx_2, dy_2) 
\end{align*}
for all $s= 2\dots, N.$
We set 
\[
\mathcal{AW}_{p}^\infty(\mu, \nu) := \inf_{\gamma \in \Pi_{\operatorname{bc}}(\mu, \nu)} \mathcal{F}_{0, p}(\gamma)
\]
and call $\mathcal{AW}^\infty_p$ the  adapted $(p,\infty)$--Wasserstein distance.
\end{definition}

In Section \ref{sec:dpp} we show that this definition indeed gives rise to a DPP as explained above. In the remainder of this section, we discuss basic properties of $\mathcal{AW}_p^\infty$: we show that $\mathcal{AW}_p^\infty$ is a metric, that dominates $\mathcal{AW}_p.$ We also identify assumptions under which $\mathcal{AW}_p\approx\mathcal{AW}_p^\infty$. 

Let us start with an easy reformulation of Definition \ref{def:aw_inf}, which is similar to \cite[(3.2)-(3.3)]{veraguas2020fundamental}.

\begin{lemma}[DPP formulation for $\AW_p^\infty$] \label{lem:adap_new}
For $t=1, \dots, N$ and $\mu, \nu \in \mathcal{P}_p(\mathbb{R}^{N})$ define $A_{N,p}:=0$ and
\begin{align}\label{eq:adap_inf_new}
A_{t,p}(x_{1:t}, y_{1:t}) = \inf_{\gamma^{t+1} \in \Pi(\mu_{x_{1:t}}, \nu_{y_{1:t}})} \mathcal{C}_p(\gamma^{t+1}) \vee \|A_{t + 1,p}(x_{1:t+1} ,y_{1:t+1})\|_{L^\infty(\gamma^{t+1})}. 
\end{align}
Then we have
\begin{align}\label{eq:a_infty_dpp}
\mathcal{AW}^\infty_{p}(\mu,\nu)=\inf_{\gamma^{1} \in \Pi(\mu^1, \nu^1 )} \mathcal{C}_p(\gamma^{1}) \vee \|A_{1,p}(x_{1} ,y_{1})\|_{L^\infty(\gamma^{1})}.
\end{align}
In particular
\begin{align}\label{eq:a_infty_dpp1}
\mathcal{AW}^\infty_{p}(\mu,\nu)=\inf_{\gamma^{1} \in \Pi(\mu^1, \nu^1 )} \mathcal{C}_p(\gamma^{1}) \vee \|\AW^\infty_p(\bar\mu_{x_1}, \bar\nu_{y_1})\|_{L^\infty(\gamma^{1})}.
\end{align}
\end{lemma}

\begin{remark}
As a comparison, by \cite[(3.2)-(3.3)]{veraguas2020fundamental}  one has
\begin{align*}
\mathcal{AW}_{p}(\mu,\nu)^p=\inf_{\gamma^{1} \in \Pi(\mu^1, \nu^1 )} \mathcal{C}_p(\gamma^{1})^p + \|\AW_p(\bar\mu_{x_1}, \bar\nu_{y_1})^p\|_{L^1(\gamma^{1})}. 
\end{align*}
\end{remark}

\begin{lemma}\label{lem:easy}
The adapted $(p,\infty)$--Wasserstein distance is a metric satisfying $\mathcal{AW}_p\le N^{1/p} \mathcal{AW}^\infty_p.$
\end{lemma}

In general, $\mathcal{AW}_p$ and $\mathcal{AW}_p^\infty$ are not equivalent, as the following example shows:

\begin{example}
Define 
\begin{align*}
\mu=\delta_{(0,0)}\quad \text{and}\quad
\mu_n = \epsilon_n\delta_{(\epsilon_n(\epsilon_n-1), -\frac{1}{\sqrt{\epsilon_n}})} + (1-\epsilon_n) \delta_{(\epsilon_n^2, \frac{\sqrt{\epsilon_n}}{1 - \epsilon_n})}.
\end{align*}
Then 
\begin{align*}
\mathcal{AW}_1(\mu,\mu_n) &= 
\epsilon_n \Big(|\epsilon_n(\epsilon_n-1)| + \frac{1}{\sqrt{\epsilon_n}}\Big)  + (1-\epsilon_n) \Big(\epsilon_n^2 + \frac{\sqrt{\epsilon_n}}{1-\epsilon_n}\Big)\\
&=2 \epsilon_n^2 (1 - \epsilon_n) + 2 \sqrt{\epsilon_n} \to 0,
\end{align*}
if \(\epsilon_n \to 0\). However,
\[
\AW_{1}^\infty (\mu,\mu_n) = \max\Big(2 \epsilon_n^2 (1 - \epsilon_n), \mathcal{W}_1\Big(\delta_{-\frac{1}{\sqrt{\epsilon_n}}}, \delta_0\Big), \mathcal{W}_1\Big(\delta_{\frac{\sqrt{\epsilon_n}}{1 - \epsilon_n}}, \delta_0\Big)\Big) = \frac{1}{\sqrt{\epsilon_n}}
\]
for \(\varepsilon_n > 0\) small enough. Thus, $\AW_{1}^\infty (\mu,\mu_n)\to \infty$.
\end{example}

However, under strong regularity assumptions one can show that $\mathcal{AW}_p\approx \mathcal{AW}_p^\infty$:

\begin{proposition}[Equivalence of $\AW_p$ and $\AW_p^\infty$] \label{prop:equivalence}
Assume that 
\begin{itemize}
    \item $\operatorname{spt}(\mu), \operatorname{spt}(\nu) \subseteq B_R(0) \subset \mathbb{R}^N$ for some \(R > 0\),
    \item  \(\mu,\nu\) have \(L\)-Lipschitz disintegrations, i.e., \(x_{1:t} \mapsto \mu_{x_{1:t}}\) is \(L\)-Lipschitz with respect to \(\mathcal{W}_p\) and similarly for $\nu$,
    \item there exists an optimal coupling $\gamma\in \Pi_{\operatorname{bc}}(\mu, \nu)$ for $\AW_p(\mu,\nu)$, which satisfies $\gamma \ll \text{Leb}|_{B_R(0)}$ and has a density bounded from below by some constant $K>0$.
\end{itemize}
Then for any $\delta>0$ there exists a constant $C=C(\delta,K)$ such that $\mathcal{AW}_p^\infty(\mu,\nu)\le C\mathcal{AW}_p(\mu,\nu) +N\delta$.
\end{proposition}

\begin{proof}
Fix an arbitrary $\delta > 0$. Since $\mu$ and $\nu$ have $L$-Lipschitz disintegrations, applying Lemma \ref{lem:aw_growth} for the function
$$
a_t(x_{1:t}, y_{1:t}) = \mathcal{AW}_p(\bar{\mu}_{x_{1:t}}, \bar{\nu}_{y_{1:t}})
$$
yields
$$
\mathcal{AW}_p(\bar{\mu}_{x'_{1:t}}, \bar{\nu}_{y'_{1:t}})^p \leq C_{t, L} (\|\Delta x_{1:t}\|^p + \|\Delta y_{1:t}\|^p) + D_{t, L} \mathcal{AW}_p(\bar{\mu}_{x_{1:t}}, \bar{\nu}_{y_{1:t}})^p
$$
for all $\Delta x_{1:t} := x'_{1:t} - x_{1:t}, \Delta y_{1:t} := y'_{1:t} - y_{1:t}$ and constants $C_{t, L}, D_{t, L} > 0$. Moreover, the density of $\gamma^{1:t} := (\operatorname{proj}^{1:t})_\# \gamma$ is bounded from below by $K > 0$ as a projection of the measure $\gamma$ with the same property. Therefore, by Lemma \ref{prop:linfty_bound} applied to $g(x_{1:t}, y_{1:t}) =\mathcal{AW}_p(\bar{\mu}_{x_{1:t}}, \bar{\nu}_{y_{1:t}})^p $ we have 
\begin{equation}\label{eq:bdd_bel1}
\|\mathcal{AW}_p(\bar{\mu}_{x_{1:t}}, \bar{\nu}_{y_{1:t}})^p\|_{L^\infty(\gamma^{1:t})} \leq C_{t, \delta, K} \|\mathcal{AW}_p(\bar{\mu}_{x_{1:t}}, \bar{\nu}_{y_{1:t}})^p\|_{L^1(\gamma^{1:t})} + \delta
\end{equation}
for a constant $C_{t, \delta, K}>0$. Combining this inequality with the bound
\begin{align}\label{eq:bdd_bel2}
\| \mathcal{AW}_p(\bar{\mu}_{x_{1:t}}, \bar{\nu}_{y_{1:t}})^p\|_{L^1(\gamma^{1:t})} \leq \mathcal{AW}_p(\mu, \nu)^p,
\end{align}
which follows by optimality of $\gamma$, we obtain
\begin{align*}
\mathcal{AW}^\infty_p(\mu, \nu)^p &\quad \le \mathcal{C}_p(\gamma^1)^p + \sum_{t = 1}^{N-1} \|\mathcal{C}_p(\gamma_{x_{1:t}, y_{1:t}})^p\|_{L^\infty(\gamma^{1:t})}\\
&\stackrel{\text{Def. }\AW_p}{\leq} \mathcal{AW}_p(\mu, \nu)^p + \sum_{t = 1}^{N-1} \|\mathcal{AW}_p(\bar{\mu}_{x_{1:t}}, \bar{\nu}_{y_{1:t}})^p\|_{L^\infty(\gamma^{1:t})}\\
&\quad \stackrel{\eqref{eq:bdd_bel1}}{\leq} \mathcal{AW}_p(\mu, \nu)^p + \sum_{t = 1}^{N-1} C_{t, \delta, K} \|\mathcal{AW}_p(\bar{\mu}_{x_{1:t}}, \bar{\nu}_{y_{1:t}})^p\|_{L^1(\gamma^{1:t})} + (N - 1) \delta\\
&\quad \stackrel{\eqref{eq:bdd_bel2}}{\leq} \left(1 + \sum_{t = 1}^{N-1} C_{t, \delta, K}\right) \mathcal{AW}_p(\mu, \nu)^p + (N - 1) \delta,
\end{align*}
which completes the proof.
\end{proof}


\section{A dynamic programming principle for $\mathcal{AW}_p^\infty$-DRO problems} \label{sec:dpp}

We set $B_\delta(\mu):=\{\nu\in \mathcal{P}(\R^N): \mathcal{AW}_p^\infty(\mu,\nu)< \delta\}$ 
 for the remainder of this note.

\subsection{Properties of $B_\delta(\mu)$} \label{sec:balls}

Let us first remark, that deriving a DPP for $V(\delta)$ is nontrivial, because $B_\delta(\mu)$ is neither convex nor precompact, as the following example shows:

\begin{example}
\begin{enumerate}
\item \(B_\delta(\mu)\) is not convex: consider \(\mu = \frac{1}{2} \delta_{1, 1} + \frac{1}{2} \delta_{0, 100}\), and \(\nu = \frac{1}{2} \delta_{1, 100} + \frac{1}{2} \delta_{0, 1}\). Then we have $\AW_p^\infty(\mu, \nu)\le 1 \vee \max(\mathcal{W}_p(\delta_1, \delta_1), \mathcal{W}_p(\delta_{100}, \delta_{100}))=1$, so that \(\nu \in B_{2}(\mu)\). However,
\[
\widehat{\nu} := \frac{1}{2} \mu + \frac{1}{2} \nu = \frac{1}{4} \delta_{1, 1} + \frac{1}{4} \delta_{1, 100} + \frac{1}{4} \delta_{0, 1} + \frac{1}{4} \delta_{0, 100}\notin B_2(\mu),
\]
as the conditional probabilities are \(\widehat{\nu}_{1} = \widehat{\nu}_{0} = \frac{1}{2} \delta_1 + \frac{1}{2} \delta_{100}\), and \(\mathcal{W}_p(\widehat{\nu}_0, \delta_1) = \mathcal{W}_p(\widehat{\nu}_0, \delta_{100}) = 99 \cdot (\frac{1}{2})^{\frac{1}{p}} > 2\).
\item \(B_\delta(\mu)\) is not precompact: consider
\[
\mu = \delta_{0, 0} \quad \mu_n = \left(1 - \frac{1}{n}\right)\delta_{0, 0} + \frac{1}{n} \delta_{1, 1}.
\]
Then \(\mu_n \in B_2(\mu)\), however \((\mu_n)_{n \in \mathbb{N}}\) does not have a convergent subsequence. Indeed, it should match the weak limit, which is equal to \(\mu\), but \(\mathcal{AW}^\infty_p(\mu_n, \mu) = 1\) for any \(n \in \mathbb{N}\), making  convergence in \((\mathcal{P}_p(\mathbb{R}^2), \mathcal{AW}^\infty_p) \) impossible.
\end{enumerate}
\end{example}

In conclusion, direct methods from calculus of variations cannot be applied to derive a DPP.
However, as hinted at in the Introduction, optimization over $B_\delta(\mu)$ can be achieved through a recursive construction, which is reminiscent of a DPP for robust utility maximization, see e.g., \cite{nutz2016utility}. To see this, we consider the following sets:
\[
\Pi_{\delta}(\mu_{x_{1:t}}, \cdot) := \{\pi \in \Pi(\mu_{x_{1:t}}, \cdot):\, \mathcal{C}_p(\pi) < \delta\} \;\; \text{for} \;\; x_{1:t} \in \mathbb{R}^t\;\; \text{and} \;\; t = 0, 1, \ldots, N.
\]
To make a connection to \cite{nutz2016utility} we set $\Omega_t:\R^t\times \R^t$, $\Omega:=\Omega_N= \R^N\times \R^N$ and $\mathcal{P}_t(\omega):=\Pi_{\delta}(\mu_{x_{1:t}}, \cdot) \in \mathcal{P}(\Omega_1)$ for $\omega=(x_{1:t}, y_{1:t})\in \Omega_t$. With these definitions $\mathcal{P}_t(\omega)$ is non-empty, convex, weakly pre-compact and one can show that the graph of $\omega \twoheadrightarrow\mathcal{P}_t(\omega)$ is analytic. In other words, $\Pi_{\delta}(\mu_{x_{1:t}}, \cdot)$ is simply a special instance of the robust single-step models of \cite{nutz2016utility} on the enlarged space $\Omega.$ 
Following again the convention in \cite{nutz2016utility} we define the set of models up to time \(t =1, \dots, N\) via
\begin{align*}
\Pi_\delta(\mu^{1:t}, \cdot) := \{  \gamma^1 \otimes \gamma_{x_1, y_1} \otimes \ldots \otimes \gamma_{x_{1:t-1}, y_{1:t-1}}:\, \gamma_{x_{1:s-1},y_{1:s-1}} \in \Pi_\delta(\mu_{x_{1:s-1}}, \cdot),\, s =1, 2, \ldots, t\},
\end{align*}
where the kernels $(x_{1:t}, y_{1:t})\mapsto \gamma_{x_{1:t},y_{1:t}}$ are Borel measurable, and we define $\gamma_{x_{1:0}, y_{1:0}}:=\gamma^1$.

Take $\gamma\in \Pi_\delta(\mu, \cdot)=\Pi_\delta(\mu^{1:N}, \cdot)$ and denote its second marginal by $\nu$, i.e., $\gamma\in \Pi(\mu,\nu).$ By \cite[Proposition 2.4, 2]{backhoff2017causal} $\gamma$ is causal, but \emph{not} necessarily bicausal. Similar issues have been observed e.g., in \cite{bartl2023sensitivity} and can be overcome by an additional approximation argument. For this we define 
$$\Pi^{\operatorname{T}}_\delta(\mu_{x_{1:t}}, \cdot) := \{\pi \in \Pi_{\delta}(\mu_{x_{1:t}}, \cdot):\, \pi = (T, \operatorname{Id})_\# \nu \; \text{for some}\; \nu \in \mathcal{P}_p(\mathbb{R}) \text{ and measurable }T:\R\to \R\},$$
and
\begin{equation*}
\Pi_\delta^{\operatorname{T}}(\mu^{1:t}, \cdot) := \{  \gamma^1 \otimes \gamma_{x_1, y_1} \otimes \ldots \otimes \gamma_{x_{1:t-1}, y_{1:t-1}}:\, \gamma_{x_{1:s-1},y_{1:s-1}} \in \Pi_\delta^{\operatorname{T}}(\mu_{x_{1:s-1}}, \cdot), s =1, 2, \ldots, t\},
\end{equation*}
where $(x_{1:t}, y_{1:t})\mapsto \gamma_{x_{1:t},y_{1:t}}$ is Borel measurable.
By definition it then follows for $(X,Y)\sim \gamma\in \Pi_\delta^{\operatorname{T}}(\mu, \cdot)$, that $X_t$ is $\sigma(Y_{1:t})$-measurable, and thus $\gamma$ is in fact bicausal. On the other hand $\Pi_\delta^{\operatorname{T}}(\mu_{x_{1:t}},\cdot)$ is weakly dense in $\Pi_\delta(\mu_{x_{1:t}},\cdot)$, see Lemma \ref{density_of_maps}, so that the difference between $\Pi_\delta(\mu,\cdot)$ and $\Pi^{\operatorname{T}}_\delta(\mu,\cdot)$ is often negligible. In fact it will turn out that
\[
\sup_{\nu \in B_\delta(\mu)} \int f(y)\, \nu(dy) = \sup_{\gamma \in \Pi_{\delta}(\mu, \cdot)} \int f(y)\, \gamma(dx, dy) = \sup_{\gamma \in \Pi_{\delta}^{\operatorname{T}}(\mu, \cdot)} \int f(y) \,\gamma(dx, dy)
\]
under mild regularity assumptions, see Lemma \ref{lem:supremum_closed}. 

\subsection{Main results}
We are now ready to state the main results of this section. We start with the DPP for $V(\delta)$ and defer proofs to Section \ref{sec:proofs}. Let us first consider the case with not controls. For this we make the following definition:
\begin{definition}\label{def:dpp}
For a Borel measurable function $f: \mathbb{R}^N \to \mathbb{R}$ we define
\begin{align}
\begin{split}\label{eq:dpp_def}
V^\delta_N(x, y) &:= f(y),\\
V_t^\delta(x_{1:t}, y_{1:t}) &:= \sup_{\gamma^{t+1} \in \Pi_\delta(\mu_{x_{1:t}}, \cdot)} \int V^\delta_{t+1}(x_{1:t+1}, y_{1:t+1})\,\gamma^{t+1}(dx_{t+1}, dy_{t+1}),
\end{split}
\end{align}
where we recall that \(\Pi_\delta(\mu_{x_{1:t}}, \cdot) = \{\pi \in \Pi(\mu_{x_{1:t}}, \cdot):\, \mathcal{C}_p(\pi) < \delta\}\), as well as the DRO problem
\[
V(\delta) := \sup_{\nu \in B_\delta(\mu)} \int f(y) 
\,\nu(dy).
\]
\end{definition}

\begin{theorem}[DPP for $V(\delta)$, uncontrolled case]\label{equiv:basic}
Let \(f: \mathbb{R}^N \to \mathbb{R}\) be a lower semicontinuous function, such that $f(x) \geq -C(1+\|x\|^{p - \varepsilon})$ for some $\varepsilon > 0$ and constant $C > 0$.
Then the dynamic programming principle
\[
V(\delta) = V_0^\delta
\]
holds.
\end{theorem}

Next we introduce the counterpart of Definition \ref{def:dpp}, that allows for controls.

\begin{definition}
For a Borel measurable function $f:\R^N\times \R^N\to \R$ and a compact set $K\subset \R$ we define
\begin{align}\label{dpp:control}
\begin{split}
V_N^\delta(x, y, \alpha) &:= f(y, \alpha),\\
V_t^\delta(x_{1:t}, y_{1:t}, \alpha_{1:t}) &:= \inf_{\alpha_{t+1} \in K} \sup_{\gamma^{t+1} \in \Pi_\delta(\mu_{x_{1:t}}, \cdot)} \int V_{t+1}^{\delta}(x_{1:t+1}, y_{1:t+1}, \alpha_{1:t+1}) \,\gamma^{t+1}(dx_{t+1}, dy_{t+1}),
\end{split}
\end{align}
as well as the DRO problem
\begin{equation}\label{dro:control}
V(\delta) := \inf_{\alpha \in \mathcal{A}} \sup_{\gamma \in \Pi_{\operatorname{bc}, \delta}(\mu, \cdot)} \int f(y, \alpha(x, y)) \,\gamma(dx, dy),
\end{equation}
where \(\mathcal{A}\) is the set of \textit{predictable} controls \((\alpha_1, \alpha_2, \ldots, \alpha_N) = \alpha: \mathbb{R}^N \times \mathbb{R}^N \to K^N\), meaning that \(\alpha_t(x,y)\) only depends on \((x_{1:t-1}, y_{1:t-1})\) and 
\begin{align*}
\Pi_{\operatorname{bc}, \delta}(\mu, \cdot):= \{\gamma \in \Pi_{ \delta}(\mu, \cdot):\,\gamma \text{ is bicausal}\}.
\end{align*}
\end{definition}

Before we can state the corresponding DPP, we also need the following definition:
\begin{definition}
The measure \(\mu \in \mathcal{P}_p(\mathbb{R}^N)\) is \emph{successively $\mathcal{W}_p$--continuous}, if
\[
x_{1:t} \mapsto \mu_{x_{1:t}} \;\; \text{is continuous with respect to \(\mathcal{W}_p\) for each } t=1,\dots, N.
\]
\end{definition}

This property of $\mu$ is needed for a DPP in the controlled case --- see Theorem \ref{equiv:control_open} below---, whereas the uncontrolled case in Theorem \ref{equiv:basic} did not require any kind of regularity of $\mu$.
In fact it turns out that, as soon as $\mu$ is successively $\mathcal{W}_p$--continuous, the cost-to-go functions $V_t$ inherit the regularity of $f$, and the Monge and Kantorovich formulations of \eqref{dpp:control} agree in the following sense:

\begin{lemma}[Regularity of $V_t^\delta$] \label{lem:cost_to_go_regularity}
Let $\mu \in \mathcal{P}_p(\mathbb{R}^N)$ be a successively $\mathcal{W}_p$--continuous probability measure. Let $f: \mathbb{R}^N \times \mathbb{R}^N \to \mathbb{R}$ be a Borel measurable function, and let $K \subset \mathbb{R}$ be a compact set. Then the following holds:
\begin{enumerate}
    \item If $f$ is lower semicontinuous and bounded from below, then $V^\delta_t$ is lower semicontinuous and bounded from below.
    \item Assume that $\mu$ satisfies $\int \|x_{t+1:N}\|^p \, \bar{\mu}_{x_{1:t}}(dx_{t+1:N}) \lesssim 1 + \|x_{1:t}\|^p$ and that the function $f$ is continuous and satisfies $|f(x, \alpha)| \lesssim 1 + \|x\|^{p - \varepsilon}$ for some $\varepsilon > 0$. Then $V^\delta_t$ is continuous and satisfies the same growth condition, i.e., $|V^\delta_t(x_{1:t}, y_{1:t}, \alpha_{1:t})| \lesssim 1 + \|x_{1:t}\|^{p - \varepsilon} + \|y_{1:t}\|^{p - \varepsilon}$.
\end{enumerate}
In both cases we have
$$
V_t^\delta(x_{1:t}, y_{1:t}, \alpha_{1:t}) = \inf_{\alpha_{t+1} \in K} \sup_{\gamma^{t+1} \in \Pi^{\operatorname{T}}_\delta(\mu_{x_{1:t}}, \cdot)} \int V^\delta_{t+1}(x_{1:t+1}, y_{1:t+1}, \alpha_{1:t+1}) \,\gamma^{t+1}(dx_{t+1}, dy_{t+1}).
$$
\end{lemma}

As a consequence of Lemma \ref{lem:cost_to_go_regularity} and an additional approximation argument, \eqref{dpp:control} is equal to \eqref{dro:control}:

\begin{theorem}[DPP for $V(\delta)$, controlled case] \label{equiv:control_open}
Let $\mu \in \mathcal{P}_p(\mathbb{R}^N)$ be successively $\mathcal{W}_p$--continuous and let $K \subset \mathbb{R}$ be a compact set. Assume that $f: \mathbb{R}^N \times \mathbb{R}^N \to \mathbb{R}$ is lower semicontinuous and  satisfies one of the following:
\begin{enumerate}[(a)]
    \item $f$ is bounded from below,
    \item $|f(x, \alpha)| \lesssim 1 + \|x\|^{p - \varepsilon}$ for some $\varepsilon > 0$. 
\end{enumerate} 
Then
\[
V(\delta) = V_0^\delta.
\]
Consequently \(V^{\delta}_t\) can be written as the cost-to-go function
\[
V^{\delta}_t(x_{1:t}, y_{1:t}, \alpha_{1:t}) = \inf_{\alpha \in \mathcal{A}(\alpha_{1:t})} \sup_{\gamma \in \Pi_{\operatorname{bc}, \delta}(\bar{\mu}_{x_{1:t}}, \cdot)} \int f(y, \alpha)\, \gamma(dx_{t+1:N}, dy_{t+1:N}),
\]
where \(\mathcal{A}(\alpha_{1:t})\) is the set of predictable controls with the first \(t\) values equal to \(\alpha_{1:t}\).
\end{theorem}

As \eqref{dpp:control} naturally solves a causal DRO problem, an immediate consequence of Theorem \ref{equiv:control_open} is that bicausal and causal optimization problems have the same value:

\begin{corollary}[Causal problem] \label{cor:causal_bicausal}
In the setting of Theorem \ref{equiv:control_open} we have
\[
V(\delta) = \inf_{\alpha \in \mathcal{A}} \sup_{\gamma \in \Pi_{\delta}(\mu, \cdot)} \int f(y, \alpha(x, y)) \,\gamma(dx, dy).
\]
\end{corollary}

This corollary is useful when determining the sensitivity of the map $\delta \mapsto V(\delta)$ in Section \ref{sec:sens}.

Instead of considering the ball $B_\delta(\mu)$, our results can also be stated for the subset of $B_\delta(\mu)$, that satisfies an additional martingale constraint:

\begin{corollary}[DPP for $V(\delta)$ with martingale constraint]\label{cor:linear}
Let \(\mu \in \mathcal{P}_p(\mathbb{R}^N)\) be a successively $\mathcal{W}_p$--continuous martingale measure, i.e.,
\[
\int x_{t+1} \,\mu_{x_{1:t}}(dx_{t+1}) = x_t \quad \mu\text{-a.s.}
\] Let \(f: \mathbb{R}^N \times \mathbb{R}^N \to \mathbb{R}\) be a lower semicontinuous function, that is bounded from below. 
Define
\[
\Pi^\mathcal{M}_{\delta}(\mu_{x_{1:t}}, \cdot) := \Big\{\pi \in \Pi(\mu_{x_{1:t}}, \cdot)\;:\; \mathcal{C}_p(\pi) < \delta, \int (x-y)\,\pi(dx,dy) =0  \Big\} 
\]
for $x_{1:t} \in \mathbb{R}^t$, $t = 0, 1, \ldots, N$, and set
\begin{align*}
\Pi^\mathcal{M}_{ \delta}(\mu, \cdot):= \big\{  \gamma^1 \otimes \gamma_{x_1, y_1} \otimes \ldots \otimes \gamma_{x_{1:t-1}, y_{1:t-1}}:\ \gamma_{x_{1:s-1},y_{1:s-1}} \in \Pi^{\mathcal{M}}_\delta(\mu_{x_{1:s-1}}, \cdot),\, s =1, 2, \ldots, t\big\},
\end{align*}
where $(x_{1:t}, y_{1:t})\mapsto \gamma_{x_{1:t},y_{1:t}}$ are Borel measurable functions,
as well as 
\begin{align*}
\Pi^\mathcal{M}_{\operatorname{bc}, \delta}(\mu, \cdot):= \{\gamma \in \Pi^\mathcal{M}_{ \delta}(\mu, \cdot):\,\gamma \text{ is bicausal}\}.  
\end{align*}
Next let 
\begin{align}\label{dro:control_martingale}
V^{\mathcal{M}}(\delta):= \inf_{\alpha \in \mathcal{A}} \sup_{\gamma \in \Pi^\mathcal{M}_{\operatorname{bc}, \delta}(\mu, \cdot)} \int f(y, \alpha(x, y)) \,\gamma(dx, dy),
\end{align}
and define 
\begin{align}\label{eq:martingale}
\begin{split}
V_N^{\mathcal{M},\delta}(x, y, \alpha) &:= f(y, \alpha),\\
V_t^{\mathcal{M},\delta}(x_{1:t}, y_{1:t}, \alpha_{1:t}) &:= \inf_{\alpha_{t+1} \in K} \sup_{\gamma^{t+1} \in \Pi^\mathcal{M}_\delta(\mu_{x_{1:t}}, \cdot)} \int V_{t+1}^{\mathcal{M},\delta}(x_{1:t+1}, y_{1:t+1}, \alpha_{1:t+1}) \,\gamma^{t+1}(dx_{t+1}, dy_{t+1}).
\end{split}
\end{align}
Then
\[
V^{\mathcal{M}}(\delta) = V^{\mathcal{M},\delta}_0.
\]
\end{corollary}

At the level of generality of Theorem \ref{equiv:control_open}, it is not clear to us if the infimum and supremum in \eqref{dro:control} can be interchanged. We will return to this question in Section \ref{sec:sens}, when we discuss first-order approximations. If we assume more structure on $f$ however, this is indeed true, even though the balls $B_\delta(\mu)$ are neither convex nor compact as seen in Section \ref{sec:balls}. Before stating this result, we first need the following definition:

\begin{definition}
We say that a function \(f: \mathbb{R}^N \times \mathbb{R}^N \to \mathbb{R}\) is \emph{semi-separable}, if the decomposition 
\[
f(y, \alpha) = \sum_{t = 1}^N f_t(y_{1:t}, \alpha_t)
\]
holds for Borel functions \(f_t: \mathbb{R}^t \times \mathbb{R} \to \mathbb{R}\).
\end{definition}

\begin{theorem}[Minimax theorem for $V(\delta)$] \label{minimax}
Let $p > 1$ and let $\mu \in \mathcal{P}_p(\mathbb{R}^N)$ be a successively $\mathcal{W}_p$--continuous probability measure. Let $f: \mathbb{R}^N \times \mathbb{R}^N \to \mathbb{R}$ be a semi-separable Borel function, such that $\alpha \mapsto f(x, \alpha)$ is convex for any $x \in \mathbb{R}^N$, and let $K \subset \mathbb{R}$ be a compact set. Moreover, assume that one of the following holds:
\begin{enumerate}[(a)]
    \item $f$ is continuous and satisfies $|f(x, \alpha)| \lesssim 1 + \|x\|^{p - \varepsilon}$ for some $\varepsilon > 0$, and $$\int \|x_{t+1:N}\| \, \bar{\mu}_{x_{1:t}}(dx_{t+1:N}) \lesssim 1 + \|x_{1:t}\|^p,$$
    \item $f$ is lower semicontinuous and bounded from below.
\end{enumerate}
Then
\[
V(\delta) = \inf_{\alpha \in \mathcal{A}} \sup_{\gamma \in \Pi_{\operatorname{bc}, \delta}(\mu, \cdot)} \int f(y, \alpha(x, y)) \,\gamma(dx, dy) = \sup_{\gamma \in \Pi_{\operatorname{bc}, \delta}(\mu, \cdot)} \inf_{\alpha \in \mathcal{A}} \int f(y, \alpha(x, y)) \,\gamma(dx, dy).
\]
\end{theorem}

\begin{corollary}\label{cor:minimax}
In the setting of Theorem \ref{minimax} we have
\[
V(\delta) = \sup_{\nu \in B_\delta(\mu)} \inf_{\alpha \in \widetilde{\mathcal{A}}} \int f(y, \alpha(y)) \,\nu(dy),
\]
where $\widetilde{\mathcal{A}}$ is the set of predictable controls $\alpha:\R^N\to \R^N.$
\end{corollary}

\begin{proof}
The ``$\leq$"--inequality is trivial. In order to establish the ``$\geq$"--inequality, it suffices to notice that
\[
\int f(y, \alpha(x, y)) \, \gamma(dx, dy) \geq \int f\left(y, \int \alpha(x, y) \, \gamma_y(dx)\right) \gamma(dx, dy)
\]
for any \(\gamma \in \Pi_{\operatorname{bc}, \delta}(\mu, \cdot)\) and any \(\alpha \in \mathcal{A}\) by convexity.
\end{proof}

\section{First-order sensitivity for $\mathcal{AW}_p^\infty$-DRO problems} \label{sec:sens}

Theorems \ref{equiv:basic} and \ref{equiv:control_open} allow to reduce the $\AW_p^\infty$-DRO problem $V(\delta)$ to a  sequence of simpler DRO problems. In fact, using convex duality, \eqref{eq:dpp_def} can be restated as a finite-dimensional optimization problem, see \cite{blanchet2019quantifying, gao2023distributionally}:
\begin{equation}\label{weak_transport:duality}
\sup_{\gamma \in \Pi_\delta(\mu, \cdot)} \int f(y)\, \gamma(dx, dy) = \inf_{\lambda \geq 0} \left[\lambda \delta^p + \int f_\lambda(x)\, \mu(dx) \right],
 \end{equation}
where \(f_\lambda(x) = \sup_{y \in \mathbb{R}} [f(y) - \lambda |x - y|^p]\) is the \(\lambda|\cdot|\)-transform of \(f \in C(\mathbb{R})\). Albeit being much more tractable than its multiperiod counterpart, \eqref{weak_transport:duality} is still computationally heavy: one has to compute the \(\lambda|\cdot|\)-transform and numerically integrate it for each optimization step. An alternative to this procedure, at least for small $\delta>0$, is the following approximation: as in \cite{bartl2021sensitivity, bartl2023sensitivity} we write
\[
V(\delta) = V(0) + \delta\cdot  \Upsilon + o(\delta), \quad \delta \to 0,
\]
where the sensitivity \(\Upsilon\) is defined as 
\[
\Upsilon: = \lim_{\delta \to 0+} \frac{V(\delta) - V(0)}{\delta}.
\]

As in the case for $\AW_p$ one can derive an explicit formula for \(\Upsilon\) under regularity assumptions on \(f\).

\begin{theorem}[Sensitivity of $V(\delta)$]\label{dro:control_sensitivity}
Let \(p > 1\) and let \(\mu \in \mathcal{P}_p(\mathbb{R}^N)\) be a successively $\mathcal{W}_p$--continuous probability measure, which satisfies \(\int \|x_{t+1:N}\|^p \,\bar{\mu}_{x_{1:t}}(dx_{t+1:N}) \lesssim 1 + \|x_{1:t}\|^p\). Let \(f: \mathbb{R}^N \times K^N \to \mathbb{R}\) be a lower semicontinuous function, which satisfies the following:
\begin{enumerate}
    \item The mapping \(x \mapsto f(x, \alpha)\) is differentiable, $(x,\alpha)\mapsto \nabla_x f(x,\alpha)$ is continuous, and
    \[
    \|\nabla_x f(x, \alpha)\| \leq C(1 + \|x\|^{p - 1 - \varepsilon})
    \]
    for some \(\varepsilon > 0\) and a constant \(C > 0\).
    \item The mapping \(\alpha \mapsto f(x, \alpha)\) is \(\varepsilon(x)\)-strongly convex for all \(x \in \operatorname{spt}(\mu)\), where \(\mu(\varepsilon(x) > 0) = 1\), i.e.,
    \[
    f(x, \widetilde{\alpha}) - f(x, \alpha) \geq \langle \nabla_\alpha f(x, \alpha), \widetilde{\alpha} - \alpha \rangle + \frac{1}{2} \varepsilon(x) \|\widetilde{\alpha} - \alpha\|^2
    \]
    for any \(x \in \mathbb{R}^N\) and \(\alpha, \widetilde{\alpha} \in K^N\).
\end{enumerate}
Then the sensitivity $\Upsilon$ of the robust optimal control problem
$$
V(\delta) := \inf_{\alpha \in \mathcal{A}} \sup_{\gamma \in \Pi_{\operatorname{bc}, \delta}(\mu, \cdot)} \int f(y, \alpha) \, \gamma(dx, dy)
$$
is given by
\begin{align*}
\Upsilon = \inf_{\alpha \in \mathcal{A}^*} &\left[ \left\|\int\partial_{1} f(x, \alpha) \,\bar{\mu}_{x_{1}}(dx_{2:N})\right\|_{L^q(\mu^1(dx_{1}))} \right.\\
&\left.+ \sum_{t = 2}^{N} \int \left\|\int\partial_{t} f(x, \alpha) \,\bar{\mu}_{x_{1:t}}(dx_{t+1:N})\right\|_{L^q(\mu_{x_{1:t-1}}(dx_{t}))} \,\mu(dx_{1:t-1}) \right],
\end{align*}
where $\mathcal{A}^* = \operatorname{argmin}(V^\mathcal{M}(0))$.
\end{theorem}

In order to find the first-order approximation for \(V(\delta)\) it thus suffices to solve a sequence of  convex optimization problems for $V(0)$ to determine $\alpha^*$ and then integrate wrt. $\mu$.

\begin{remark}[Comparison to sensitivity of $\AW_p$-DRO problems]
In general the sensitivity $\Upsilon$ computed in Theorem \ref{dro:control_sensitivity} does not align with the sensitivity for $V^{\AW_p}(\delta)$ computed in \cite{bartl2023sensitivity}. In fact, for $N=2$,
\begin{align*}
\Upsilon= \Big\|\int \partial_1f(x, \alpha^*)\,\mu_{x_1}(dx_2)\Big\|_{L^q(\mu)} + \| \|\partial_{2} f(x,\alpha^*)\|_{L^q(\mu_{x_{1}})} \|_{L^1(\mu)},
\end{align*}
while the one computed in \cite{bartl2023sensitivity} yields
\begin{align*}
\tilde \Upsilon = \Big(\Big\| \int \partial_1 f(x,\alpha^*) \,\mu_{x_1}(dx_2) \Big\|_{L^q(\mu)}^q  +\| \partial_2 f(x,\alpha^*)\|_{L^q(\mu)}^q\Big)^{1/q}
\end{align*}
and thus $\Upsilon \le \tilde \Upsilon,$ as expected from the inequality $\mathcal{AW}_p\le 2^{1/p} \mathcal{AW}^\infty_p$. Furthermore $|\Upsilon-\tilde\Upsilon|\to 0$ for $p\to \infty$, which is reflected by $\mathcal{AW}_p^\infty \asymp \mathcal{AW}_\infty $ for $p \to \infty.$
\end{remark}

\subsection{Sensitivity for $\mathcal{AW}_p^\infty$-DRO problems with martingale constraint}

As an extension of our results above, we consider the sensitivity of $V^\mathcal{M}(\delta).$ As in Section \ref{sec:dpp} we assume that $\mu$ is a martingale measure and set up the DPP as in \eqref{eq:martingale}. Moreover we recall that the corresponding distributionally robust problem is 
\begin{equation*}
V^{\mathcal{M}} (\delta) = \inf_{\alpha \in \mathcal{A}} \sup_{\gamma \in \Pi^{\mathcal{M}}_{\operatorname{bc}, \delta}(\mu, \cdot)} \int f(y, \alpha(x, y)) \,\gamma(dx, dy)
\end{equation*}
and set
\[
\Upsilon^{\mathcal{M}} := \lim_{\delta \to 0+} \frac{V^{\mathcal{M}}(\delta) - V^{ \mathcal{M}}(0)}{\delta}.
\]

Then Theorem \ref{dro:control_sensitivity} can be extended as follows.

\begin{corollary}[Sensitivity for $V(\delta)$ with martingale constraint] \label{dro:control_martingale_sensitivity}
Let \(p > 1\) and let \(\mu \in \mathcal{P}_p(\mathbb{R}^N)\) be a successively $\mathcal{W}_p$--continuous probability measure, which satisfies \(\int \|x_{t+1:N}\|^p \bar{\mu}_{x_{1:t}}(dx_{t+1:N}) \lesssim 1 + \|x_{1:t}\|^p\). Let \(f: \mathbb{R}^N \times K^N \to \mathbb{R}\) be a lower semicontinuous function, which satisfies the following:
\begin{enumerate}
    \item The map \(x \mapsto f(x, \alpha)\) is differentiable, $(x,\alpha)\mapsto \nabla_x f(x, \alpha)$ is continuous and
    \[
    \|\nabla_x f(x, \alpha)\| \leq C(1 + \|x\|^{p - 1 - \varepsilon})
    \]
    for some \(\varepsilon > 0\) and a constant \(C > 0\).
    \item The mapping \(\alpha \mapsto f(x, \alpha)\) is \(\varepsilon(x)\)-strongly convex for all \(x \in \operatorname{spt}(\mu)\), where \(\mu(\varepsilon(x) > 0) = 1\), meaning that
    \[
    f(x, \widetilde{\alpha}) - f(x, \alpha) \geq \langle \nabla_\alpha f(x, \alpha), \widetilde{\alpha} - \alpha \rangle + \frac{1}{2} \varepsilon(x) \|\widetilde{\alpha} - \alpha\|^2
    \]
    for any \(x \in \operatorname{spt}(\mu)\) and \(\alpha, \widetilde{\alpha} \in K^N\).
\end{enumerate}
Then the sensitivity $\Upsilon^{\mathcal{M}}$ of
$$
V^\mathcal{M}(\delta) := \inf_{\alpha \in \mathcal{A}} \sup_{\gamma \in \Pi^{\mathcal{M}}_{\operatorname{bc}, \delta}(\mu, \cdot)} \int f(y, \alpha) \, \gamma(dx, dy)
$$
is given by
\begin{align*}
\Upsilon^\mathcal{M} = \inf_{\alpha \in \mathcal{A}^*} &\left[ \inf_{\lambda_1 \in \mathbb{R}} \left\|\int\partial_{1} f(x, \alpha) \,\bar{\mu}_{x_{1}}(dx_{2:N}) + \lambda_1\right\|_{L^q(\mu^1(dx_{1}))} \right.\\
&\left.+ \sum_{t = 2}^{N} \int \inf_{\lambda_t \in \mathbb{R}} \left\|\int\partial_{t} f(x, \alpha) \,\bar{\mu}_{x_{1:t}}(dx_{t+1:N}) + \lambda_t \right\|_{L^q(\mu_{x_{1:t-1}}(dx_{t}))} \,\mu(dx_{1:t-1})\right],
\end{align*}
where $\mathcal{A}^* = \operatorname{argmin}(V^\mathcal{M}(0))$.
\end{corollary}

\begin{corollary}
In the case \(p = 2\) we obtain
\begin{align*}
\Upsilon^\mathcal{M} = \inf_{\alpha \in \mathcal{A}^*} &\left[ \sqrt{\operatorname{Var}_{\mu^1(dx_1)} \left(\int\partial_{1} f(x, \alpha) \,\bar{\mu}_{x_{1}}(dx_{2:N})\right)}\right.\\
&\left.+\sum_{t = 2}^{N} \int \sqrt{\operatorname{Var}_{\mu_{x_{1:t-1}}(dx_t)} \left(\int\partial_{t} f(x, \alpha) \,\bar{\mu}_{x_{1:t}}(dx_{t+1:N})\right)} \,\mu(dx_{1:t-1}) \right],
\end{align*}
where $\operatorname{Var}_{\mu_{x_{1:t}}}(\cdot)$ denotes the variance wrt. $\mu_{x_{1:t}}.$
\end{corollary}

\section{Remaining proofs} \label{sec:proofs}

\subsection{Proofs of Lemma \ref{lem:adap_new} and Lemma \ref{lem:easy}}

\begin{proof}[Proof of Lemma \ref{lem:adap_new}]

We start with the proof of the ``$\ge$"--inequality of \eqref{eq:a_infty_dpp}. Fix an arbitrary \(\varepsilon > 0\) and let \(\eta \in \Pi_{\operatorname{bc}}(\mu, \nu)\) be a near-optimal transport plan for \(\mathcal{AW}^\infty_p(\mu, \nu)\), i.e.,
\begin{align}\label{eq:near_optimal}
\mathcal{F}_{0, p}(\eta) \leq \mathcal{AW}^\infty_p(\mu, \nu) + \varepsilon.
\end{align}
Note that  \((x_{1:t}, y_{1:t}) \mapsto \eta_{x_{1:t}, y_{1:t}}\) is Borel and $\eta_{x_{1:t}, y_{1:t}} \in \Pi(\mu_{x_{1:t}}, \nu_{y_{1:t}})$ for all $t = 1, \ldots, N-1$ by bicausality. Furthermore, $A_{N,p}=\mathcal{F}_{N,p}=0$. Hence, for any $t = 1, \ldots, N-1$ we obtain 
\begin{align}\label{eq:aw_dynamic}
\begin{split}
A_{t, p}(x_{1:t}, y_{1:t}) &\leq \mathcal{C}_p(\eta_{x_{1:t}, y_{1:t}}) \vee \|A_{t+1, p}(x_{1:t+1}, y_{1:t+1})\|_{L^\infty(\eta_{x_{1:t}, y_{1:t}})} \\
&\le  \mathcal{C}_p(\eta_{x_{1:t}, y_{1:t}}) \vee \|\mathcal{F}_{t+1, p}(\bar{\eta}_{x_{1:t+1}, y_{1:t+1}})\|_{L^\infty(\eta_{x_{1:t}, y_{1:t}})}\\
&= \mathcal{F}_{t, p}(\bar{\eta}_{x_{1:t}, y_{1:t}})
\end{split}
\end{align}
by backward induction and by definition of $\mathcal{F}_{t, p}$. Thus,
\begin{align*}
\inf_{\gamma^{1} \in \Pi(\mu^1, \nu^1 )} \mathcal{C}_p(\gamma^{1}) \vee \|A_{1,p}(x_{1} ,y_{1})\|_{L^\infty(\gamma^{1})} &\leq 
\mathcal{C}_p(\eta^{1}) \vee \|A_{1,p}(x_{1} ,y_{1})\|_{L^\infty(\eta^{1})}\\
&\stackrel{\eqref{eq:aw_dynamic}}{\le} \mathcal{C}_p(\eta^1) \vee \|\mathcal{F}_{1, p}(\bar{\eta}_{x_1, y_1})\|_{L^\infty(\eta^1)} \\
&= \mathcal{F}_{0, p}(\eta)\\
&\stackrel{\eqref{eq:near_optimal}}{\le}
\mathcal{AW}^\infty_p(\mu, \nu) + \varepsilon.
\end{align*}
Taking \(\varepsilon \to 0\) we obtain $$\inf_{\gamma^{1} \in \Pi(\mu^1, \nu^1 )} \mathcal{C}_p(\gamma^{1}) \vee \|A_{1,p}(x_{1} ,y_{1})\|_{L^\infty(\gamma^{1})} \leq \mathcal{AW}^\infty_p(\mu, \nu).$$

For the ``$\le$"--inequality we follow  \cite[proof of Theorem 4.2]{backhoff2017causal} and first show that $(x_{1:t}, y_{1:t}) \mapsto A_{t, p}(x_{1:t}, y_{1:t})$ is lower semianalytic by backward induction. The case $t = N$ is straightforward, so suppose that $t < N$. As in \cite[proof of Theorem 4.2, Step 1]{backhoff2017causal}, the set
$$
\{(x_{1:t}, y_{1:t}, \pi):\, \pi \in \Pi(\mu_{x_{1:t}}, \nu_{y_{1:t}})\} \subset (\mathbb{R}^t \times \mathbb{R}^t) \times \mathcal{P}_p(\mathbb{R}^2)
$$
is analytic. Moreover, the map $\pi \mapsto \mathcal{C}_p(\pi)$ is Borel. Recall that $(x_{1:t+1}, y_{1:t+1}) \mapsto A_{t+1, p}(x_{1:t+1}, y_{1:t+1})$ is lower semianalytic by the induction hypothesis.

Note that for fixed $\pi \in \mathcal{P}(\R^2)$ 
we have
\begin{align}\label{eq:essential_sup}
\sup_{q\in \N} \|g\|_{L^q(\pi)} =\|g\|_{L^\infty(\pi)};
\end{align}
see e.g., \cite[Lemma 13.1]{guide2006infinite}. Following \cite[proof of Theorem 4.2, Step 2]{backhoff2017causal}, $$(x_{1:t}, y_{1:t}, \pi) \mapsto \|A_{t+1, p}(x_{1:t+1}, y_{1:t+1})\|_{L^q(\pi)}$$ is lower semianalytic for any $q\in \N$ by \cite[Proposition 7.48]{bertsekas1996stochastic} as the integration of lower semianalytic functions against Borel kernels, and so is 
\begin{align*}
(x_{1:t}, y_{1:t}, \pi) \mapsto \|A_{t+1, p}(x_{1:t+1}, y_{1:t+1})\|_{L^\infty(\pi)}   
\end{align*}
using \eqref{eq:essential_sup} and \cite[Lemma 7.30.(2)]{bertsekas1996stochastic}. In conclusion, 
\begin{align*}
(x_{1:t}, y_{1:t}, \pi) \mapsto \mathcal{C}_p(\pi)\vee \|A_{t+1, p}(x_{1:t+1}, y_{1:t+1})\|_{L^\infty(\pi)}
\end{align*}
is lower semianalytic. Lastly we fix an arbitrary $\varepsilon > 0$ and apply \cite[Proposition 7.50.(b)]{bertsekas1996stochastic} to obtain a universally measurable selection of near-optimizers $(x_{1:t}, y_{1:t}) \mapsto \gamma_{x_{1:t}, y_{1:t}} \in \Pi(\mu_{x_{1:t}}, \nu_{y_{1:t}})$, i.e.,
\[
\mathcal{C}_p(\gamma_{x_{1:t}, y_{1:t}}) \vee \|A_{t+1, p}(x_{1:t+1}, y_{1:t+1})\|_{L^\infty(\gamma_{x_{1:t}, y_{1:t}})} \leq A_{t, p}(x_{1:t}, y_{1:t}) + \varepsilon.
\]
Using \cite[Lemma 7.28.(c)]{bertsekas1996stochastic} iteratively for $t=1, \dots, N-1$, one can actually choose Borel measurable versions of $(x_{1:t}, y_{1:t})\mapsto \gamma_{x_{1:t}, y_{1:t}}$. This allows to construct a bicausal transport plan $\Pi_{\operatorname{bc}}(\mu, \nu) \ni \gamma := \gamma^1 \otimes \gamma_{x_1, y_1} \otimes \ldots \otimes \gamma_{x_{1:N-1}, y_{1:N-1}}$ as a concatenation of Borel measurable kernels by \cite[Proposition 7.28]{bertsekas1996stochastic}. By backward induction
\begin{align}\label{eq:details}
\begin{split}
\mathcal{F}_{t,p}(\bar{\gamma}_{x_{1:t}, y_{1:t}}) &= \mathcal{C}_p(\gamma_{x_{1:t}, y_{1:t}}) \vee  \|\mathcal{F}_{t+1,p}(\bar\gamma_{x_{1:t+1}, y_{1:t+1}}) \|_{L^\infty(\gamma_{x_{1:t}, y_{1:t}})} \\
&\le \mathcal{C}_p(\gamma_{x_{1:t}, y_{1:t}}) \vee  \|A_{t+1,p}(x_{1:t+1}, y_{1:t+1}) \|_{L^\infty(\gamma_{x_{1:t}, y_{1:t}})} +(N-t-1)\epsilon\\
&\le A_{t, p}(x_{1:t}, y_{1:t}) +(N-t)\epsilon.
\end{split}
\end{align}
Thus 
\begin{align*}
\mathcal{AW}^\infty_p(\mu, \nu) = \inf_{\gamma \in \Pi_{\text{bc}} (\mu, \nu )}\mathcal{F}_{0, p}(\gamma) &\le\inf_{\gamma^{1} \in \Pi(\mu^1, \nu^1 )} \mathcal{C}_p(\gamma^1) \vee \|\mathcal{F}_{1,p}( \bar{\gamma}_{x_{1}, y_{1}})\|_{L^\infty(\gamma^1)}\\ &\stackrel{\eqref{eq:details}}{\le} 
\inf_{\gamma^{1} \in \Pi(\mu^1, \nu^1 )} \mathcal{C}_p(\gamma^{1}) \vee \|A_{1,p}(x_{1} ,y_{1})\|_{L^\infty(\gamma^{1})} +(N-1)\varepsilon,
\end{align*}
and hence the result follows by taking $\varepsilon \to 0$.
\end{proof}

\begin{proof}[Proof of Lemma \ref{lem:easy}]
We first show the inequality $\mathcal{AW}_p\le N^{1/p} \mathcal{AW}^\infty_p$ by induction. The case \(N = 1\) is trivial, so we assume \(N > 1\) and set $m:=\mathcal{AW}^\infty_p(\mu, \nu)$. Fix $\epsilon > 0$. By definition there exists a transport plan \(\gamma^1 \in \Pi(\mu^1, \nu^1)\), which satisfies
\begin{equation*}
\mathcal{C}_p(\gamma^1) \vee \|\mathcal{AW}^\infty_p(\bar{\mu}_{x_1}, \bar{\nu}_{y_1})\|_{L^\infty(\gamma^1)} \le m + \varepsilon/2.
\end{equation*}
The induction hypothesis implies that for $\gamma^1$-every $(x_1, y_1)$ we have
\[
\mathcal{AW}_p(\bar{\mu}_{x_1}, \bar{\nu}_{y_1}) \leq (N - 1)^{1/p} \cdot \mathcal{AW}^\infty_p(\bar{\mu}_{x_1}, \bar{\nu}_{y_1}).
\]
Combining these two inequalities and using  an argument similar to the proof of Lemma \ref{lem:adap_new}, there exists a universally measurable kernel \((x_1, y_1) \mapsto \bar{\gamma}_{x_1, y_1} \in \Pi_{\operatorname{bc}}(\bar{\mu}_{x_1}, \bar{\nu}_{y_1})\) satisfying 
\begin{equation}\label{inequality_residual}
\Big(\int \|x_{2:N} - y_{2:N}\|^p \, \bar{\gamma}_{x_1, y_1}(dx_{2:N}, dy_{2:N})\Big)^{1/p} \le (N - 1 )^{1/p} \cdot \left(\mathcal{AW}^\infty_p(\bar{\mu}_{x_1}, \bar{\nu}_{y_1}) +\epsilon \right).
\end{equation}
To conclude the proof of the inequality, note that \(\gamma(dx_{1:N}, dy_{1:N}) := \gamma^1(dx_1, dy_1) \otimes \bar{\gamma}_{x_1, y_1}(dx_{2:N}, dy_{2:N}) \in \Pi_{\operatorname{bc}}(\mu, \nu)\), and 
\[ \int \|x-y\|^p\,\gamma(dx,dy)\le 
\mathcal{C}_p(\gamma^1)^p + \int \|x_{2:N} - y_{2:N}\|^p \gamma(dx_{1:N}, dy_{1:N}) < N (m + \varepsilon)^p,
\]
which implies \(\mathcal{AW}_p(\mu, \nu) \leq N^{1/p} (m + \varepsilon)\) for any \(\varepsilon > 0\).

By the inequality $\mathcal{AW}_p\le N^{1/p} \mathcal{AW}^\infty_p$ and the fact that $\AW_p$ is a metric, positive definiteness of $\AW_p^\infty$ follows immediately. As symmetry of $\AW_p^\infty$ is obvious from Definition \ref{def:aw_inf}, we only need to prove the triangle inequality for $\AW_p^\infty.$ We proceed by induction: the case \(N = 1\) is simply the triangle inequality for $\mathcal{W}_p$, so we take \(N > 1\). Fixing measures $\mu,\nu,\eta\in \mathcal{P}(\R^N)$, the induction hypothesis yields
\begin{equation}\label{triangle}
|\mathcal{AW}_{p}^\infty (\bar{\mu}_{x_1}, \bar{\nu}_{y_1})| \leq |\mathcal{AW}_{p}^\infty (\bar{\mu}_{x_1}, \bar{\eta}_{z_1})| + |\mathcal{AW}_{p}^\infty (\bar{\eta}_{z_1}, \bar{\nu}_{y_1})|
\end{equation}
for any \(x_1 \in \text{spt}(\mu^1)\), \(y_1 \in \text{spt}(\nu^1)\), \(z_1 \in \text{spt}(\eta^1)\). Take any transport plans \(\pi_1 \in \Pi(\mu^1, \eta^1)\) and \(\pi_2 \in \Pi(\eta^1, \nu^1)\) and define
\begin{equation}\label{proj}
\widehat{\pi} := \pi_1 \dot{\oplus} \pi_2 \in \Pi(\mu^1, \eta^1, \nu^1), \;\; \pi_3 := (\operatorname{proj}^{1, 3})_\# \widehat \pi \in \Pi(\mu^1, \nu^1).
\end{equation}
From the pointwise inequality \eqref{triangle} together with \eqref{proj} we obtain
\begin{align}
\begin{split}\label{triangular_first}
\|\mathcal{AW}_{p}^\infty (\bar{\mu}_{x_1}, \bar{\nu}_{y_1})\|_{L^\infty(\pi_3)} &= \|\mathcal{AW}^\infty_{p}(\bar{\mu}_{x_1}, \bar{\nu}_{z_1})\|_{L^\infty(\widehat{\pi})}\ \\
&\leq \|\mathcal{AW}_{p}^\infty (\bar{\mu}_{x_1}, \bar{\eta}_{y_1})\|_{L^\infty(\widehat{\pi})} + \|\mathcal{AW}_{p}^\infty (\bar{\eta}_{y_1}, \bar{\nu}_{z_1})\|_{L^\infty(\widehat{\pi})} \\
&= \|\mathcal{AW}_{p}^\infty (\bar{\mu}_{x_1}, \bar{\eta}_{y_1})\|_{L^\infty(\pi_1)} + \|\mathcal{AW}_{p}^\infty (\bar{\eta}_{y_1}, \bar{\nu}_{z_1})\|_{L^\infty(\pi_2)}.
\end{split}
\end{align}
Furthermore, by Minkowski's inequality for \(L^p(\widehat{\pi})\) and \eqref{proj} we obtain
\begin{align}
\begin{split} 
\mathcal{C}_{p}(\pi_3) 
&= \left(\int |x_1 - y_1|^p \,\pi_3(dx_1, dy_1)\right)^\frac{1}{p} 
= \left(\int |x_1 - z_1|^p \,\widehat{\pi}(dx_1, dy_1, dz_1)\right)^\frac{1}{p} \\ 
&\leq \left(\int |x_1 - y_1|^p \,\widehat{\pi}(dx_1, dy_1, dz_1)\right)^\frac{1}{p} + \left(\int |y_1 - z_1|^p \,\widehat{\pi}(dx_1, dy_1, dz_1)\right)^\frac{1}{p} \\
&= \left(\int |x_1 - y_1|^p \,\pi_1(dx_1, dy_1)\right)^\frac{1}{p} + \left(\int |x_1 - y_1|^p \,\pi_2(dx_1, dy_1)\right)^\frac{1}{p} \\ &= \mathcal{C}_{p}(\pi_1) + \mathcal{C}_{p}(\pi_2).
\end{split}\label{triangular_second}
\end{align}
Combining \eqref{eq:a_infty_dpp1} in Lemma \ref{lem:adap_new} with \eqref{triangular_first}, \eqref{triangular_second} we obtain
\begin{align*}
\mathcal{AW}^\infty_{p}(\mu, \nu) &\leq 
 \mathcal{C}_{p}(\pi_3) \vee \|\mathcal{AW}_{p}^\infty (\bar{\mu}_{x_1}, \bar{\nu}_{y_1})\|_{L^\infty(\pi_3)}\\
&\leq \mathcal{C}_{p}(\pi_1) \vee \|\mathcal{AW}_{p}^\infty(\bar{\nu}_{y_1}, \bar{\eta}_{z_1})\|_{L^\infty(\pi_1)} + \mathcal{C}_{p}(\pi_2) \vee \|\mathcal{AW}_{p}^\infty (\bar{\eta}_{z_1}, \bar{\nu}_{y_1})\|_{L^\infty(\pi_2)}.
\end{align*}
Lastly, taking the infimum over transport plans \(\pi_1\in \Pi(\mu^1,\eta^1)\) and \(\pi_2\in \Pi(\eta^1, \nu^1)\) we conclude that
\[
\mathcal{AW}_{p}^\infty (\mu, \nu) \leq \mathcal{AW}_{p}^\infty (\mu, \eta) + \mathcal{AW}^\infty_{p}(\eta, \nu)
\]
as claimed. This concludes the proof.
\end{proof}

\subsection{Auxiliary results for the proof of Proposition \ref{prop:equivalence}} 
We need the following two technical results:

\begin{lemma}\label{lem:aw_growth}
Assume that $\mu, \nu \in \mathcal{P}_p(\mathbb{R}^N)$ have $L$-Lipschitz disintegrations and let $1\le t\le N-1$. Then there exist constants $C_{t, L}, D_{t, L} > 0$ depending only on \(t, L\), such that the function $$a_t(x_{1:t}, y_{1:t}) := \mathcal{AW}_p(\bar\mu_{x_{1:t}}, \bar\nu_{y_{1:t}})$$ satisfies
$$
a_t(x'_{1:t}, y'_{1:t})^p \leq C_{t, L} (\|\Delta x_{1:t}\|^p + \|\Delta y_{1:t}\|^p) + D_{t, L} a_t(x_{1:t}, y_{1:t})^p$$
for all $\Delta x_{1:t} := x'_{1:t} - x_{1:t}$ and $\Delta y_{1:t} := y'_{1:t} - y_{1:t}$.
\end{lemma}

\begin{proof}
We prove the claim via backward induction.
The case $t = N - 1$ is straightforward: indeed, $a_{N-1}(x_{1:N-1}, y_{1:N-1}) = \mathcal{W}_p(\mu_{x_{1:N-1}}, \nu_{y_{1:N-1}})$, and by $L$-Lipschitz continuity of disintegrations and the triangle inequality for $\mathcal{W}_p$ we have
\begin{align*}
a_{N-1}(x'_{1:N-1}, y'_{N-1}) &\leq a_{N-1}(x_{1:N-1}, y_{1:N-1}) + |a_{N-1}(x_{1:N-1}, y'_{1:N-1}) - a_{N-1}(x_{1:N-1}, y_{1:N-1})|\\
&+ |a_{N-1}(x_{1:N-1}, y'_{1:N-1}) - a_{N-1}(x_{1:N-1}', y_{1:N-1}')|\\
&\leq a_{N-1}(x_{1:N-1}, y_{1:N-1}) + L(\|\Delta y_{1:N-1}\| + \|\Delta x_{1:N-1}\|).
\end{align*}
Suppose now that $t < N - 1$ and fix an arbitrary \((x_{1:t}, y_{1:t}, x'_{1:t}, y'_{1:t}) \in (\R^{t})^4\). Let $\bar\pi \in \Pi(\bar\mu_{x_{1:t}},\bar\nu_{y_{1:t}})$ be optimal for $\mathcal{AW}_p(\bar{\mu}_{x_{1:t}}, \bar{\nu}_{y_{1:t}}).$ Define $\pi: = (\operatorname{proj}^{1,1})_{\#}\bar\pi \in \Pi(\mu_{x_{1:t}}, \nu_{y_{1:t}})$ and let $\pi_1 \in \Pi(\mu_{x'_{1:t}}, \mu_{x_{1:t}})$ and $\pi_2 \in \Pi(\nu_{y_{1:t}}, \nu_{y'_{1:t}})$ be $\mathcal{W}_p$-optimal transport plans. Take $\widehat \pi := \pi_1 \dot{\oplus} \pi \dot{\oplus} \pi_2 \in \Pi(\mu_{x'_{1:t}}, \mu_{x_{1:t}}, \nu_{y_{1:t}}, \nu_{y'_{1:t}})$, and define $\tilde \pi := (\operatorname{proj}^{1, 4})_\# \widehat{\pi} \in \Pi(\mu_{x'_{1:t}}, \nu_{y'_{1:t}})$. Then by Minkowski's inequality for $L^p(\widehat{\pi})$ we have
\begin{align}\label{eqn:aw_lipschitz:one_step_cost}
\begin{split}
\mathcal{C}_p(\tilde \pi) &\leq \mathcal{C}_p(\pi_1) + \mathcal{C}_p(\pi) + \mathcal{C}_p(\pi_2)\\
&= \mathcal{W}_p(\mu_{x'_{1:t}}, \mu_{x_{1:t}}) + \mathcal{C}_p(\pi) + \mathcal{W}_p(\nu_{y_{1:t}}, \nu_{y'_{1:t}})\\
&\leq L(\|\Delta x_{1:t}\| + \|\Delta y_{1:t}\|) + \mathcal{C}_p(\pi).
\end{split}
\end{align}
Moreover, applying the induction hypothesis to $x'_{1:t+1}=(x'_{1:t}, \tilde x_{t+1})$, $y'_{1:t+1}=(y'_{1:t}, \tilde y_{t+1})$, $x_{1:t+1} = (x_{1:t} , \tilde x_{t+1})$ and $y_{1:t+1}=(y_{1:t}, \tilde y_{t+1})$
we obtain
\begin{align*}
&\int a_{t+1}(x'_{1:t}, \tilde x_{t+1}, y'_{1:t}, \tilde y_{t+1})^p \,\tilde \pi(d\tilde x_{t+1}, d\tilde y_{t+1})\\
&\leq C_{t+1, L} (\|\Delta x_{1:t}\|^p + \|\Delta y_{1:t}\|^p) + D_{t+1, L} \int a_{t+1}(x_{1:t}, \tilde x_{t+1}, y_{1:t}, \tilde y_{t+1})^p \,\tilde \pi(d\tilde x_{t+1}, d\tilde y_{t+1}),
\end{align*}
where the integral on the right-hand side is estimated as
\begin{align}\label{eqn:aw_lipschitz:induction_hypothesis}
\begin{split}
&\int a_{t+1}(x_{1:t}, \tilde x_{t+1}, y_{1:t}, \tilde y_{t+1})^p \,\tilde \pi(d\tilde x_{t+1}, d\tilde y_{t+1})\\
&= \int a_{t+1}(x_{1:t}, \tilde x_{t+1}, y_{1:t}, \tilde y_{t+1})^p \,\hat \pi(d\tilde x_{t+1}, dx_{t+1}, dy_{t+1}, d\tilde y_{t+1})\\
&\stackrel{\text{(IH)}}{\le} \int [C_{t+1,L} (|x_{t+1}-\tilde x_{t+1}|^p +|y_{t+1}-\tilde y_{t+1}|^p) + D_{t+1,L} a_{t+1}(x_{1:t+1}, y_{1:t+1})^p]\,\hat \pi(d\tilde x_{t+1}, dx_{t+1}, dy_{t+1}, d\tilde y_{t+1})\\
&= C_{t+1, L} (\mathcal{W}_p(\mu_{x'_{1:t}}, \mu_{x_{1:t}})^p + \mathcal{W}_p(\nu_{y_{1:t}}, \nu_{y'_{1:t}})^p) + D_{t+1, L} \int a_{t+1}(x_{1:t+1}, y_{1:t+1})^p\, \pi(dx_{t+1}, dy_{t+1})\\
&\leq L^p \cdot C_{t+1, L} (\|\Delta x_{1:t}\|^p + \|\Delta y_{1:t}\|^p) + D_{t+1, L} \int a_{t+1}(x_{1:t+1}, y_{1:t+1})^p \,\pi(dx_{t+1}, dy_{t+1}).
\end{split}
\end{align}
Combining \eqref{eqn:aw_lipschitz:one_step_cost} and \eqref{eqn:aw_lipschitz:induction_hypothesis}, applying the inequality $(|a|+|b|+|c|)^p \leq 3^{p-1} (|a|^p + |b|^p +|c|^p)$ and using optimality of $\bar\pi$ we conclude
\begin{align*}
a_t(x'_{1:t}, y'_{1:t})^p &\leq \mathcal{C}_p(\tilde \pi)^p + \int a_{t+1}(x'_{1:t}, \tilde x_{t+1}, y'_{1:t}, \tilde x_{t+1})^p\, \tilde \pi(d\tilde x_{t+1}, d\tilde y_{t+1})\\
&\leq 3^{p-1} \mathcal{C}_p(\pi)^p+ (3^{p-1} L^p + C_{t+1, L} + D_{t+1, L} L^p C_{t+1, L}) (\|\Delta x_{1:t}\|^p + \|\Delta y_{1:t}\|^p) \\
&+ D_{t+1, L}^2 \int a_{t+1}(x_{1:t+1}, y_{1:t+1})^p\, \pi(dx_{t+1}, dy_{t+1})\\
&\leq C_{t, L} (\|\Delta x_{1:t}\|^p + \|\Delta y_{1:t}\|^p) + D_{t, L} a_t(x_{1:t}, y_{1:t})^p,
\end{align*}
where \(C_{t, L} := 3^{p-1} L^p + C_{t+1, L} + D_{t+1, L} L^p C_{t+1, L}\) and \(D_{t, L} := 3^{p-1} \vee D_{t+1, L}^2\). The proof is complete.
\end{proof}

\begin{lemma}\label{prop:linfty_bound}
Let $R>0$ and $g: B_R(0) \to \mathbb{R}$ be a function satisfying
\begin{equation}\label{eqn:growth_condition_for_linfty_bound}
g(x') \leq C_1 g(x) + C_2 \|x - x'\|^p
\end{equation}
for all $x,x'\in B_R(0)$ and some constants $C_1, C_2, p > 0$. Let $\pi \in \mathcal{P}(B_R(0))$ satisfy $\pi\ll \text{Leb}|_{B_R(0)}$ with a density bounded from below by a constant $K > 0$.
Then for any $\delta > 0$ there exists a constant $C = C(\delta, K)$, such that
$$
\|g\|_{L^\infty(\pi)} \leq C \|g\|_{L^1(\pi)} + \delta.
$$

\end{lemma}

\begin{proof}
For any $\delta > 0$ we can find $x_\delta \in B_R(0)$, such that
$$
\|g\|_{L^\infty(\pi)} = \|g\|_{L^\infty(\pi_\delta)},
$$
where $\pi_\delta:=\pi|_{B_{\delta/2}(x_\delta)}$. Taking the $L^\infty(\pi_\delta)$-norm in \eqref{eqn:growth_condition_for_linfty_bound} we obtain
$$
\|g\|_{L^\infty(\pi)} = \|g\|_{L^\infty(\pi_\delta)} \leq C_1 g(x) + C_2 \delta^p \quad \text{for any } x \in B_{\delta/2}(x_\delta).
$$
We rewrite this inequality as $C_1 g(x) \geq \|g\|_{L^\infty(\pi)} - C_2 \delta^p$ and by boundedness  of the density of $\pi$ we obtain
\begin{align*}
C_1 \|g\|_{L^1(\pi)} &\geq C_1 \int_{B_{\delta/2}(x_\delta)} |g(x)|\, \pi(dx) \geq (\|g\|_{L^\infty(\pi)} -C_2 \delta^p) \cdot \pi(B_{\delta/2}(x_\delta))\\
&\geq (\|g\|_{L^\infty(\pi)} - C_2 \delta^p) \cdot K \cdot \text{Leb}(B_{\delta/2}(0)),
\end{align*}
where $\text{Leb}(B_{\delta/2}(0))$ is the Lebesgue measure of $B_{\delta/2}(0).$
In conclusion we have
\begin{align*}
\|g\|_{L^\infty(\pi)} \leq \frac{C_1}{K \cdot \text{Leb}(B_{\delta/2}(0))} \|g\|_{L^1(\pi)} + C_2 \delta^p.
\end{align*}
The claim now follows from a re-normalization of constants.
\end{proof}

\subsection{Auxiliary results for the proofs in Sections \ref{sec:dpp} and \ref{sec:sens}}

We start with a number of density results, which help to conclude the equivalence of the Kantorovich and Monge formulations of the transportation problems we consider.
\begin{lemma}\label{lem:density_two_ac}
Let $\mu, \nu \in \mathcal{P}_p(\mathbb{R})$, where $\mu, \nu$ are atomless. Then the set of one-to-one transport maps between $\nu$ and $\mu$ is dense in $\Pi(\mu, \nu)$ with respect to $\mathcal{W}_p$.
\end{lemma}
\begin{proof}
Take any transport plan $\gamma \in \Pi(\mu, \nu)$. According to \cite[Proposition A.3]{gangbo1999monge} there exists a sequence of one-to-one maps \(T_n: \mathbb{R} \to \mathbb{R}\), such that
\begin{equation}\label{weak_property}
    T_{n \#} \nu = \mu, \;\; \text{and} \;\; \gamma_n := (x \mapsto (T_n(x), x))_{\#} \nu \to \gamma \quad \text{weakly}.
\end{equation}
Moreover, since \(\mu\) and \(\nu\) are \(p\)-integrable and \(\gamma_n \in \Pi(\mu, \nu)\),
\begin{align}
\int [\|x\|^p + \|y\|^p]\, \gamma_n(dx, dy) &= \int \|x\|^p\, \mu(dx) + \int \|y\|^p\, \nu(dy) < +\infty.\label{moments}
\end{align}
Combining \eqref{weak_property} and \eqref{moments}, we refer to \cite[Theorem 6.9]{villani2009optimal} to conclude that
\[
\gamma_n \to \gamma \quad \text{in} \;\; (\mathcal{P}_p(\mathbb{R} \times \mathbb{R}), \mathcal{W}_p),
\]
as \(n \to \infty\), which completes the proof.
\end{proof}

\begin{lemma}\label{lem:gangbo}
Let \(\mu_{x_{1:t}}, \nu \in \mathcal{P}_p(\mathbb{R})\), where \(\nu\) is atomless. Then the set of transport maps between \(\nu\) and \(\mu_{x_{1:t}}\) is dense in \(\Pi(\mu_{x_{1:t}}, \nu)\) with respect to \(\mathcal{W}_p\).
\end{lemma}

\begin{proof}
Take any transport plan $\pi \in \Pi(\mu_{x_{1:t}}, \nu)$ and fix an  arbitrary $\varepsilon > 0$. Let $\mu_\varepsilon \ll \operatorname{Leb}$ be a measure, which satisfies $\mathcal{W}_p(\mu_{x_{1:t}}, \mu_\varepsilon) < \varepsilon$, and denote by $\pi_\varepsilon^T := (\operatorname{Id}, T_{\mu_{x_{1:t}}})_\# \mu_\varepsilon \in \Pi(\mu_\varepsilon, \mu_{x_{1:t}})$ the optimal transport plan. Define
$$
\widehat \pi := \pi_\varepsilon^T \dot{\oplus} \pi \in \Pi(\mu_\epsilon,\mu_{x_{1:t}}, \nu), \;\; \pi_\varepsilon := (\operatorname{proj}^{1, 3})_\# \widehat \pi \in \Pi(\mu_\epsilon,\nu),
$$
and note that
\begin{align}\label{eqn:lem_gangbo_conv}
\mathcal{W}_p(\pi_\epsilon, \pi) \le \Big(\int |x-y|^p+|z-z|^p \,\widehat{\pi}(dx,dy,dz) \Big)^{1/p} =  \mathcal{C}_p(\pi^T_\varepsilon) < \varepsilon.
\end{align}
Noting that both $\nu, \mu_\epsilon$ are atomless, we can use Lemma \ref{lem:density_two_ac} to find a transport plan $\pi^{T}_\nu: = (T_\nu, \operatorname{Id})_\# \nu \in \Pi(\mu_\varepsilon, \nu)$, which satisfies $\mathcal{W}_p(\pi^T_\nu, \pi_\epsilon) < \varepsilon$. Define
$$
\pi^T := (T_{\mu_{x_{1:t}}}(x), y)_\# \pi^T_\nu(dx, dy) = (T_{\mu_{x_{1:t}}} \circ T_\nu, \operatorname{Id})_\# \nu \in \Pi(\mu_{x_{1:t}}, \nu).
$$
Similarly to \eqref{eqn:lem_gangbo_conv} we have $\mathcal{W}_p(\pi^T, \pi^T_\nu) < \varepsilon$, where the upper bound is achieved by the transport plan $(T_{\mu_{x_{1:t}}} \circ T_\nu, \operatorname{Id}, T_\nu, \operatorname{Id})_\#\nu\in \Pi(\mu_{x_{1:t}}, \nu,\mu_\epsilon,\nu)$. Lastly we estimate 
\begin{align*}
\mathcal{W}_p(\pi^T, \pi) \leq \mathcal{W}_p(\pi^T, \pi^T_\nu) + \mathcal{W}_p(\pi^T_\nu, \pi_\varepsilon) + \mathcal{W}_p(\pi_\varepsilon, \pi) < 3 \varepsilon
\end{align*}
using the  triangle inequality, which completes the proof since $\varepsilon > 0$ was arbitrary.
\end{proof}


\begin{lemma}\label{density_of_maps}
The set \(\Pi_\delta^{\operatorname{T}}(\mu_{x_{1:t}}, \cdot)\) is dense in \(\Pi_\delta(\mu_{x_{1:t}}, \cdot)\) with respect to \(\mathcal{W}_p\).
\end{lemma}

\begin{proof}
Take any coupling \(\pi \in \Pi_\delta(\mu_{x_{1:t}}, \cdot)\). For $\sigma>0$ define
\begin{align}\label{eq:smooth}
\pi_\sigma(A \times B) := \int_{A\times \R \times B} \varphi_\sigma(z - y) \,\pi(dx,dy)dz \quad \text{for all } A, B \in \mathcal{B}(\mathbb{R}),
\end{align}
where \(\varphi_\sigma(x) := \frac{1}{\sigma \sqrt{2\pi}} \exp(-\frac{x^2}{2 \sigma^2})\) is the pdf of a normal distribution with mean zero and variance $\sigma^2$, and call its second marginal $\nu_\sigma$. In probabilistic terms, $\pi_\sigma$ corresponds to \((X, Y + \sigma \eta)\), where \((X, Y) \sim \pi\) and \(\eta \sim \mathcal{N}(0, 1)\) is independent of \((X, Y)\). Note that $\nu_\sigma\ll \text{Leb}$ by definition. We also have
\[
\lim_{\sigma\to 0} \mathcal{W}_p(\pi, \pi_\sigma) \le \limsup_{\sigma\to 0} \Big(\int |y-z|^p \, \varphi_\sigma(z - y) \,\pi(dx,dy)dz\Big)^{1/p} =0.
\]
The claim thus follows from Lemma \ref{lem:gangbo}.
\end{proof}

\begin{proposition}\label{prop:open_dense_in_closed}
The set $\Pi_\delta(\mu_{x_{1:t}}, \cdot)$ is dense in $\overline \Pi_\delta(\mu_{x_{1:t}}, \cdot) := \{\pi \in \Pi(\mu_{x_{1:t}}, \cdot):\, \mathcal{C}_p(\pi)\le \delta\}$ with respect to $\mathcal{W}_p$.
\end{proposition}
\begin{proof}
Take any transport plan $\pi \in \overline \Pi_\delta(\mu_{x_{1:t}}, \cdot)$. Define the transport plans
$$
\pi_n := \frac{1}{n} (\operatorname{Id}, \operatorname{Id})_\# \mu_{x_{1:t}} + \left(1 - \frac{1}{n}\right) \pi.
$$
By construction we have $\mathcal{C}_p(\pi_n) = (1 - \frac{1}{n}) \mathcal{C}_p(\pi) < \delta.$ Thus, $\pi_n \in \Pi_\delta(\mu_{x_{1:t}}, \cdot)$. Moreover, $\mathcal{W}_p(\pi_n, \pi) \leq \frac{1}{n} \delta$ with the bound being achieved by the coupling $\frac{1}{n}(x, x, x, y)_\# \pi(dx, dy) + (1 - \frac{1}{n}) (\operatorname{Id}, \operatorname{Id})_\# \pi$. Therefore, $\pi_n \to \pi$ in $(\mathcal{P}_p(\mathbb{R}^2),\mathcal{W}_p)$, and the density follows.
\end{proof}

\begin{corollary}\label{cor:open_martingale_dense_in_closed}
The set $\Pi_\delta^{\mathcal{M}}(\mu_{x_{1:t}}, \cdot)$ is dense in $\overline \Pi_\delta^{\mathcal{M}}(\mu_{x_{1:t}}, \cdot) := \{\pi \in \Pi(\mu_{x_{1:t}}, \cdot)\;:\; \mathcal{C}_p(\pi) \le \delta, \int (x-y)\,\pi(dx,dy) = 0 \}$ with respect to $\mathcal{W}_p$.
\end{corollary}
\begin{proof}
Follows from Proposition \ref{prop:open_dense_in_closed}, as
$$
\pi \in \overline \Pi_\delta^{\mathcal{M}}(\mu_{x_{1:t}}, \cdot) \Rightarrow \pi_n := \frac{1}{n} (\operatorname{Id}, \operatorname{Id})_\# \mu_{x_{1:t}} + \left(1 - \frac{1}{n}\right) \pi \in \Pi_\delta^{\mathcal{M}}(\mu_{x_{1:t}}, \cdot).
$$
\end{proof}

\begin{lemma}\label{lem:supremum_closed}
Let \(g: \mathbb{R} \times \mathbb{R} \to \mathbb{R}\) be a lower semicontinuous function, which satisfies
$$
g(x_{t+1}, y_{t+1}) \geq -C (1 + |x_{t+1}|^p + |y_{t+1}|^p)
$$
for some constant $C > 0$. Then
\begin{align*}
\sup_{\pi \in \Pi_\delta^{\operatorname{T}}(\mu_{x_{1:t}}, \cdot)} \int g(x_{t+1}, y_{t+1})\, \pi(dx_{t+1}, dy_{t+1}) &= \sup_{\pi \in \Pi_\delta(\mu_{x_{1:t}}, \cdot)} \int g(x_{t+1}, y_{t+1})\, \pi(dx_{t+1}, dy_{t+1})\\
\sup_{\pi \in \Pi_\delta(\mu_{x_{1:t}}, \cdot)} \int g(x_{t+1}, y_{t+1})\, \pi(dx_{t+1}, dy_{t+1}) &= \sup_{\pi \in \overline \Pi_\delta(\mu_{x_{1:t}}, \cdot)} \int g(x_{t+1}, y_{t+1})\, \pi(dx_{t+1}, dy_{t+1})\\
\sup_{\pi \in \Pi_\delta^\mathcal{M}(\mu_{x_{1:t}}, \cdot)} \int g(x_{t+1}, y_{t+1})\, \pi(dx_{t+1}, dy_{t+1}) &= \sup_{\pi \in \overline \Pi_\delta^\mathcal{M}(\mu_{x_{1:t}}, \cdot)} \int g(x_{t+1}, y_{t+1})\, \pi(dx_{t+1}, dy_{t+1}).
\end{align*}
\end{lemma}
\begin{proof}
Recall that the mapping \(\pi \mapsto \int fd\pi\) is weakly lower semicontinuous by \cite[Lemma 4.3]{villani2009optimal} applied with $h(x, y) = -C (1 + |x|^p + |y|^p)$. The result then follows from density established in Lemma \ref{density_of_maps}, Proposition \ref{prop:open_dense_in_closed} and Corollary \ref{cor:open_martingale_dense_in_closed}.    
\end{proof}

\begin{lemma}\label{lem:cost_to_go:lipschitz}
Let $f: \mathbb{R}^N \to \mathbb{R}$ be a Lipschitz function with constant $L > 0$. Then $(\delta, y_{1:t}) \mapsto V^\delta_t(x_{1:t}, y_{1:t})$ is Lipschitz with constant $2^{N - t} L$.
\end{lemma}
\begin{proof}
The statement holds for $t = N$, as $V^\delta_N(x, y) = f(y)$ is Lipschitz with constant $L > 0$. Suppose now that it holds for $t \leq N$, and fix $x_{1:t-1} \in \operatorname{spt}(\mu^{1:t-1})$. Take any $(y_{1:t-1}, y'_{1:t-1}, \delta, \delta')\in \R^{t-1}\times \R^{t-1}\times [0,\infty)\times [0,\infty)$ and define $\Delta \delta := \delta' - \delta,\; \Delta y_{1:t-1} := y'_{1:t-1} - y_{1:t-1}$. Let $\gamma^t\in \Pi_\delta(\mu_{x_{1:t-1}}, \cdot)$ be $\varepsilon$-optimizer for $V^\delta_{t-1}(x_{1:t-1}, y_{1:t-1})$, i.e.,
\begin{equation}\label{eqn:lem:cost_to_go:lipschitz:near_opt}
V^\delta_{t-1}(x_{1:t-1}, y_{1:t-1}) \leq \int V^\delta_t(x_{1:t}, y_{1:t}) \, \gamma^t(dx_t, dy_t) + \varepsilon.
\end{equation}
We observe that $(x_t, y_t + \frac{\Delta \delta}{\delta} (y_t - x_t))_\# \gamma^t \in \Pi_{\delta'}(\mu_{x_{1:t-1}}, \cdot)$, as 
\begin{align*}
\Big(\int \Big|y_t+ \frac{\Delta \delta}{\delta} (y_t - x_t) -x_t\Big|^p\,\gamma_t(dx_{t}, dy_t)\Big)^{1/p} &=  \Big|\frac{\Delta \delta} {\delta}+1\Big| \Big(\int |y_t-x_t|^p\,\gamma_t(dx_{t}, dy_t)\Big)^{1/p}\\
&\le |\Delta \delta +\delta| = \delta'.
\end{align*}
Hence,
\begin{align*}
V^{\delta'}_{t-1}(x_{1:t-1}, y_{1:t-1}) &\geq \int V^{\delta'}_t\left(x_{1:t}, y'_{1:t-1}, y_t + \frac{\Delta \delta}{\delta} (y_t - x_t)\right) \,\gamma^t(dx_t, dy_t)\\
&\geq \int V^\delta_t(x_{1:t}, y_{1:t}) \,\gamma^t(dx_t, dy_t) - L \left(\|\Delta y_{1:t-1}\| + \frac{\Delta \delta}{\delta} \int |y_t - x_t| \, \gamma^t(dx_t, dy_t)+|\Delta\delta|\right)\\
&\geq V^\delta_{t-1}(x_{1:t-1}, y_{1:t-1}) - \varepsilon - L (\|\Delta y_{1:t-1}\| + 2 |\Delta \delta|),
\end{align*}
where the second inequality follows from induction hypothesis and third inequality is true because of \eqref{eqn:lem:cost_to_go:lipschitz:near_opt} and $\int |y_t - x_t| \, \gamma^t(dx_t, dy_t) \leq \mathcal{C}_p(\gamma^t) \leq \delta$ by Jensen's inequality. As $\varepsilon > 0$ was arbitrary, the proof is complete.
\end{proof}

We now show that the cost-to-go functions \(V^{\delta}_t\) are lower semicontinuous. For this we need a number of technical lemmas. These will also be used in the proof of Lemma \ref{lem:cost_to_go_regularity} and Theorem \ref{dro:control_sensitivity}.

\begin{lemma}\label{diff_inf_control_continuity}
Let \(K \subset \mathbb{R}\) be  a compact and convex set and let \(F: \mathbb{R} \times K \to \mathbb{R}\) be a continuous function. Assume furthermore that 
\begin{itemize}
\item $x\mapsto F(x,\alpha)$ is differentiable with derivative $\nabla_x F(x,\alpha)$ for all $\alpha\in K$,
\item $(x,\alpha)\mapsto \nabla_x F(x,\alpha)$ is continuous,
\item \(\alpha \mapsto F(x, \alpha)\) is strictly convex for all $x\in \R.$ 
\end{itemize}
Define \(\alpha(x) := \operatorname{argmin}_{\alpha \in K} F(x, \alpha)\) and $F(x):= F(x, \alpha(x)).$ Then $x\mapsto \alpha(x)$ is continuous and $x\mapsto F(x)$ is differentiable with derivative  
\[
F'(x) = \nabla_x F(x, \alpha(x)).
\]
\end{lemma}

\begin{proof}
As \(\alpha \mapsto F(x, \alpha)\) is strictly convex, the correspondence $x \twoheadrightarrow \operatorname{argmin}_{\alpha \in K} F(x, \alpha)$ is single-valued. It follows from Berge's maximum theorem \cite[Theorem 17.31]{guide2006infinite} that 
\(x\mapsto \alpha(x) = \operatorname{argmin}_{\alpha \in K} F(x, \alpha)\) is continuous (as any single-valued upper hemicontinuous correspondence is continuous).

We now prove that
\begin{align}
    \lim_{h\to 0 } \frac{F(x+h)-F(x)}{h} = \nabla_x F(x, \alpha(x)).
\end{align}
We start with the upper bound: by definition of $\alpha(x+h)$ we have
\[
\limsup_{h \to 0} \frac{F(x + h, \alpha(x + h)) - F(x, \alpha(x))}{h} \leq \limsup_{h \to 0} \frac{F(x + h, \alpha(x)) - F(x, \alpha(x))}{h} = \nabla_x F(x, \alpha(x)).
\]
For the lower bound we again use the definition of $\alpha(x)$ to conclude
\begin{align*}
\liminf_{h \to 0} \frac{F(x + h, \alpha(x + h)) - F(x, \alpha(x))}{h} &\ge \liminf_{h \to 0} \frac{F(x + h, \alpha(x + h)) - F(x, \alpha(x + h))}{h}\\
&= \liminf_{h \to 0} \int_0^1 \nabla_x F(x + ht, \alpha(x + h)) dt = \nabla_x F(x, \alpha(x)),
\end{align*}
where the last equality follows from continuity of $x\mapsto \alpha(x)$, continuity of $(x,\alpha)\mapsto \nabla_x F(x, \alpha)$ and the dominated convergence theorem, noting that $(x, \alpha)\mapsto \nabla_x F(x,\alpha)$ is bounded on the compact set $\{(x+ht, \alpha): t,h\in [0,1], \alpha\in K\}$.
\end{proof}

\begin{lemma}\label{sliced_continuity}
Let \((\mathcal{X},\|\cdot\|_\mathcal{X}), (\mathcal{Y}, \|\cdot\|_{\mathcal{Y}})\) be two normed spaces, and define \(F: \mathcal{X} \times \mathcal{P}(\mathcal{Y}) \to \mathbb{R}\) via
\[
F(x, \gamma) := \int g(x, y)\, \gamma(dy)
\]
for a Borel function \(g: \mathcal{X} \times \mathcal{Y} \to \mathbb{R}\). Then the following hold:
\begin{enumerate}
\item If \(g\) is continuous and satisfies \(|g(x, y)| \leq C(x) (1 + \|y\|_{\mathcal{Y}}^p)\) for a locally bounded function $x\mapsto C(x)$, then \(F\) is continuous wrt.\,\(\|\cdot\|_{\mathcal{X}}+ \mathcal{W}_p(\cdot)\).
\item If \(g\) is lower semicontinuous and bounded from below, then \(F\) is lower semicontinuous wrt. $\|\cdot\|_{\mathcal{X}}+\mathcal{W}_p(\cdot) $.
\end{enumerate}
\end{lemma}

\begin{proof}
For \textit{(1)} take a sequence $(x_n, \gamma_n)_{n\in \N}$ converging to $(x_0, \gamma_0)$ with respect to in $\|\cdot\|_{\mathcal{X}} + \mathcal{W}_p(\cdot)$ and fix $B_r(x_0)\subseteq \mathcal{X}$ for some $r>0$. Then
\[
F(x_n, \gamma_n) = \int g(x, y) \,(\delta_{x_n} \otimes \gamma_n)(dx, dy) \to \int g(x, y) \,(\delta_{x_0} \otimes \gamma_0)(dx, dy) = F(x_0, \gamma_0)
\]
as \(\delta_{x_n} \otimes \gamma_n \to \delta_{x_0} \otimes \gamma_0\) in \((\mathcal{P}_p(\mathcal{X} \times \mathcal{Y}), \mathcal{W}_p)\), $g$ is continuous and \(|g(x, y)| \lesssim 1 + \|y\|_{\mathcal{Y}}^p\) on \(B_r(x_0) \times \mathcal{Y}\). For \textit{(2)} we use the same arguments together with lower semi-continuity of 
 \(\tilde \gamma \mapsto \int g \,d\tilde\gamma\) wrt.\,\(\mathcal{W}_p\), see \cite[Lemma 4.3]{villani2009optimal}.
\end{proof}

We recall the standard definitions of continuity for set-valued mappings, see \cite[Definition 17.2, Theorem 17.16, Theorem 17.19]{guide2006infinite}.

\begin{definition}\label{def:lower_hemicontinuous}
Let $A,B$ be two sets. A correspondence \(g: A \twoheadrightarrow B\) is \emph{lower hemi-continuous} at \(a \in A\), if \(g(a)\neq \emptyset\) and for every \(b \in g(a)\) and every sequence \(a_n \to a\) there exists a subsequence $(a_{n_k})$ and a sequence \(b_{n_k} \to b\), where \(b_{n_k} \in g(a_{n_k})\) for all \(k \in \mathbb{N}\).
\end{definition}
\begin{definition}\label{def:uhc}
Let $A,B$ be two sets. A compact-valued correspondence \(g: A \twoheadrightarrow B\) is \emph{upper hemi-continuous} at \(a \in A\), if \(g(a)\neq\emptyset\) and for every sequence \(a_n \to a\) and \(b_n \in g(a_n)\) there exists a convergent subsequence \(b_{n_k} \to b \in g(a)\).
\end{definition}

\begin{definition}
Let $A,B$ be two sets.  A compact-valued correspondence \(g: A \twoheadrightarrow B\) is \textit{continuous} at a point \(a \in A\) if it is both upper and lower hemi-continuous at $a.$
\end{definition}

\begin{corollary}\label{sup_continuity_maximum_th}
Let \((\mathcal{X},\|\cdot\|_\mathcal{X}), (\mathcal{Y}, \|\cdot\|_{\mathcal{Y}})\) be two normed spaces, and define $F:\mathcal{X}\to \R$ via
$$F(x) := \sup_{\gamma \in K(x)} \int g(x, y) \,\gamma(dy),$$ 
where
\begin{itemize}
\item \(g: \mathcal{X} \times \mathcal{Y} \to \mathbb{R}\) is continuous and satisfies \(|g(x, y)| \leq C(x) (1 + \|y\|_{\mathcal{Y}}^p)\) for a locally bounded function \(x\mapsto C(x)\). 
\item the correspondence $K:\mathcal{X} \twoheadrightarrow \mathcal{P}(\mathcal{Y})$ is continuous, where $K(x)$ is non-empty and compact in \((\mathcal{P}_p(\mathcal{Y}), \mathcal{W}_p)\) for each $x \in \mathcal{X}.$
\end{itemize}
Then \(F\) is continuous.
\end{corollary}

\begin{proof}[Proof.]
The function \(F(x, \gamma) := \int g(x, y) \,\gamma(dy)\) is continuous by Lemma \ref{sliced_continuity}. By Berge's maximum theorem \cite[Theorem 17.31]{guide2006infinite} we conclude that \(F(x)\) is continuous as well.
\end{proof}

\begin{lemma}\label{uniform_convergence_dini}
Let \((\mathcal{X},\|\cdot\|_\mathcal{X}), (\mathcal{Y}, \|\cdot\|_{\mathcal{Y}})\) be two normed spaces. For each $\delta \ge 0$ we define $F_\delta:\mathcal{X}\to \R$ via
$$F_\delta(x) := \sup_{\gamma \in K_\delta(x)} \int g(x, y)\, \gamma(dy),$$ where
\begin{itemize}
\item \(g: \mathcal{X} \times \mathcal{Y} \to \mathbb{R}\) is continuous and \(|g(x, y)| \leq C(x) (1 + \|y\|_{\mathcal{Y}}^p)\) for a locally bounded function  \(x\mapsto C(x)\),
\item for any \(\delta \ge 0\) the correspondence $K_\delta:\mathcal{X} \twoheadrightarrow \mathcal{P}(\mathcal{Y})$ is continuous, where $K_\delta(x)$ is non-empty and compact in \((\mathcal{P}_p(\mathcal{Y}), \mathcal{W}_p)\) for each $x \in \mathcal{X},$
\item for any $x\in \mathcal{X}$ the sequence $(K_\delta(x))_\delta$ is decreasing and $K_0$ is single-valued.
\end{itemize}
Then \(x \mapsto F_\delta(x)\) is continuous for each \(\delta \geq 0\), and \(F_\delta(x) \downarrow F_0(x)\) uniformly on any compact subset of $\mathcal{X}$ as \(\delta \to 0\).
\end{lemma}
\begin{proof}[Proof.]
Continuity of \(x \mapsto F_\delta(x)\) for $\delta \geq 0$ follows from Corollary \ref{sup_continuity_maximum_th}. Monotonicity of \(\delta \mapsto F_\delta(x)\) follows from monotonicity of \(K_\delta(x)\). Therefore, we conclude the second part by Dini's Theorem \cite[3.2.18]{engelking1989general}.
\end{proof}

Recall $$\overline\Pi_\delta(\mu_{x_{1:t}}, \cdot):= \{\pi \in \Pi(\mu_{x_{1:t}}, \cdot):\, \mathcal{C}_p(\pi)\le \delta\}.$$
We now show that \(\overline \Pi_\delta(\mu_{x_{1:t-1}}, \cdot)\) depends continuously on \((x_{1:t-1}, y_{1:t-1})\).

\begin{proposition}\label{prop:correcpondence_lhs}
Let \(\mu \in \mathcal{P}_p(\mathbb{R}^N)\) be a successively $\mathcal{W}_p$--continuous probability measure. Then the correspondence \(\mathbb{R}^{t-1}  \ni x_{1:t-1} \twoheadrightarrow \overline\Pi_\delta(\mu_{x_{1:t-1}}, \cdot) \subseteq (\mathcal{P}_p(\mathbb{R}^2), \mathcal{W}_p)\) is lower hemicontinuous. Consequently, this correspondence is lower hemicontinuous with respect to $(\mathcal{P}_{p - \varepsilon}(\mathbb{R}^2), \mathcal{W}_{p - \varepsilon})$ for any $\varepsilon > 0$.
\end{proposition}
\begin{proof}[Proof.]
Take any sequence $(x_{1:t-1}^{(n)})_{n\in \N}$ converging to some $x_{1:t-1}\in \R^{t-1}$, and a probability measure \(\pi \in \overline \Pi_\delta(\mu_{x_{1:t-1}}, \cdot)\). Consider the optimal transport plan \(\eta_n \in \Pi(\mu_{x_{1:t-1}^{(n)}}, \mu_{x_{1:t-1}})\) for \(\mathcal{W}_p\left(\mu_{x_{1:t-1}^{(n)}}, \mu_{x_{1:t-1}}\right)\), and define
\[
\pi_n := (x_t^{(n)}, y_t + \lambda_n(x_t^{(n)} - y_t))_\# (\eta_n \dot{\oplus} \pi) (dx_t^{(n)}, dx_t, dy_t),
\]
where \(\dot{\oplus}\) represents the gluing operation and \(1 - \lambda_n := \frac{\delta}{\delta + \mathcal{C}_p(\eta_n)}\). To prove lower hemicontinuity, it suffices to show that \(\pi_n \in \overline \Pi_\delta(\mu_{x_{1:t-1}^{(n)}}, \cdot)\) and \(\pi_n \to \pi\) in \((\mathcal{P}_p(\mathbb{R}^2), \mathcal{W}_p)\):
\begin{itemize}
\item  \(\pi_n \in \overline \Pi_\delta(\mu_{x_{1:t-1}^{(n)}}, \cdot)\): By \cite[Gluing lemma, p.12] {villani2009optimal} we have \(\pi_n \in \Pi(\mu_{x_{1:t-1}^{(n)}}, \cdot)\). Moreover,
    \begin{align*}
    \mathcal{C}_p(\pi_n) &= \left(\int |x_t^{(n)} - (y_t + \lambda_n (x_t^{(n)} - y_t))|^p \,(\eta_n \dot{\oplus} \pi) (dx_t^{(n)}, dx_t, dy_t)\right)^\frac{1}{p}\\
    &= (1 - \lambda_n) \left(\int |x_t^{(n)} - y_t|^p (\eta_n \dot{\oplus} \pi)(dx_t^{(n)}, dx_t, dy_t)\right)^\frac{1}{p}\\
    &\leq (1 - \lambda_n) (\mathcal{C}_p(\eta_n) + \mathcal{C}_p(\pi)) \leq \delta,
    \end{align*}
    where we have used Minkowski's inequality for \(L^p(\eta_n \dot{\oplus} \pi)\) and the definition of \(\lambda_n\). This confirms \(\pi_n \in \overline \Pi_\delta(\mu_{x_{1:t-1}^{(n)}}, \cdot)\).
    \item \(\pi_n \to \pi\) in \((\mathcal{P}_p(\mathbb{R}^2), \mathcal{W}_p)\): we bound \(\mathcal{W}_p(\pi_n, \pi)\) from above using the transport plan
    \[
    (x_t^{(n)}, y_t + \lambda_n(x_t^{(n)} - y_t), x_t, y_t)_\# (\eta_n \dot{\oplus} \pi) (dx_t^{(n)}, dx_t, dy_t) \in \Pi(\pi_n, \pi),
    \]
    yielding
    \begin{align*}
    &\mathcal{W}_p(\pi_n, \pi) \leq \left(\int \left(|x_t^{(n)} - x_t|^p + |\lambda_n (x_t^{(n)} - y_t)|^p\right) (\eta_n \dot{\oplus} \pi)(dx_t^{(n)}, dx_t, dy_t)\right)^\frac{1}{p}\\
    &\leq \mathcal{C}_p(\eta_n) + \lambda_n \left(\int |x_t^{(n)} - y_t|^p (\eta_n \dot{\oplus} \pi)(dx_t^{(n)}, dx_t, dy_t)\right)^\frac{1}{p}\\
    &\leq (1 + \lambda_n) \mathcal{C}_p(\eta_n) + \lambda_n \mathcal{C}_p(\pi) \to 0 \quad \text{for }n\to \infty;
    \end{align*}
    the last statement follows from  $\mathcal{C}_p(\eta_n) = \mathcal{W}_p(\mu_{x_{1:t-1}^{(n)}}, \mu_{x_{1:t-1}}) \to 0$ by successive weak continuity of $\mu$, which in turn implies $\lambda_n=\frac{\mathcal{C}_p(\eta_n)}{\delta+\mathcal{C}_p(\eta_n)} \to 0$. 
\end{itemize}
\end{proof}

\begin{proposition}\label{hemicontinuity}
Let \(\mu \in \mathcal{P}_p(\mathbb{R}^N)\) be a successively $\mathcal{W}_p$--continuous probability measure. Then for any $\varepsilon > 0$ and all $x_{1:t}\in \R^{t-1}$ the set $\overline\Pi_\delta(\mu_{x_{1:t-1}}, \cdot)$ is compact in $(\mathcal{P}_{p - \varepsilon}(\mathbb{R}^2),\mathcal{W}_{p - \varepsilon})$. Furthermore the correspondence \(\mathbb{R}^{t-1}  \ni x_{1:t-1} \twoheadrightarrow \overline\Pi_\delta(\mu_{x_{1:t-1}}, \cdot) \subseteq (\mathcal{P}_{p - \varepsilon}(\mathbb{R}^2),\mathcal{W}_{p - \varepsilon})\) is continuous, i.e., lower and upper hemicontinuous.
\end{proposition}

\begin{proof}[Proof.]
Lower hemicontinuity follows from Proposition \ref{prop:correcpondence_lhs}. It remains to prove upper hemicontinuity. Take any sequence $(x_{1:t-1}^{(n)})_{n\in \N}$ converging to $x_{1:t-1}\in \R^{t-1}$ as \(n \to \infty\), and probability measures \(\pi_n \in \overline \Pi_\delta(\mu_{x_{1:t-1}^{(n)}}, \cdot)\). Take any transport plan $\pi^* \in \overline \Pi_\delta(\mu_{x_{1:t-1}^{(1)}}, \cdot)$, and denote $\sup_n \mathcal{W}_p(\mu_{x_{1:t-1}^{(n)}}, \mu_{x_{1:t-1}^{(1)}})$ by $D$. We claim that
\begin{equation}\label{eqn:prop:hemicontinuity:subset}
\overline \Pi_\delta := \bigcup_{n \in \mathbb{N}} \overline \Pi_\delta(\mu_{x_{1:t-1}^{(n)}}, \cdot) \subseteq B^{\mathcal{W}_p}_{2 (\delta + D)}(\pi^*),
\end{equation}
where we recall that $B^{\mathcal{W}_p}_{\delta}(\pi^*)$ denotes the $\mathcal{W}_p$-ball of radius $\delta$ around $\pi^*.$
Indeed, take any other plan $\pi' \in \overline\Pi_\delta$ and let $\widehat{\pi} := [(y, x)_\# \pi'] \dot{\oplus} \eta \dot{\oplus} \pi^* \in \Pi(\cdot, \mu_{x_{1:t-1}^{(n)}}, \mu_{x_{1:t-1}^{(1)}}, \cdot)$, where $\eta$ is the optimal transport plan between $(\operatorname{proj}^1)_\# \pi'$ and $\mu_{x_{1:t-1}^{(1)}}$. Then $( x', y',  x, y)_\# \widehat{\pi}(dy', dx', dx, dy)$ is a transport plan between $\pi'$ and $\pi^*$. This yields the upper bound
\begin{align*}
&\mathcal{W}_p(\pi^*, \pi')\\
&\leq \left(\int [|x - x'|^p + |y - y'|^p ]\,\widehat{\pi}(dy', dx', dx, dy)\right)^\frac{1}{p}\\
&\leq 2 \left(\int |x - x'|^p \,d\widehat{\pi}\right)^\frac{1}{p} + \left(\int |x - y|^p \,d\widehat{\pi}\right)^\frac{1}{p} + \left(\int |x' - y'|^p \,d\widehat{\pi}\right)^\frac{1}{p}\\
&\leq 4 D + \mathcal{C}_p(\pi^*) + \mathcal{C}_p(\pi') \leq 2 (D+\delta ),
\end{align*}
where we have used the inequality $|y-y'|\le |y-x|+|x-x'|+|x'-y'|$ for the second inequality. In consequence, \eqref{eqn:prop:hemicontinuity:subset} follows. Since $B^{\mathcal{W}_p}_{2 (\delta + D)}(\pi^*)$ is compact in $(\mathcal{P}_{p - \varepsilon}(\mathbb{R}^2), \mathcal{W}_{p - \varepsilon})$ by \cite[Lemma 24]{bartl2021sensitivity}, then $\overline \Pi_\delta$ is precompact and $\pi_n \to \pi \in \mathcal{P}_{p - \varepsilon}(\mathbb{R}^2)$ after taking a subsequence if necessary. Recalling that $\mu_{x_{1:t-1}^{(n)}} \to \mu_{x_{1:t-1}}$ in $\mathcal{W}_p$ as $\mu$ is successively $\mathcal{W}_p$--continuous, we have \(\pi \in \Pi(\mu_{x_{1:t-1}}, \cdot)\). Moreover, by \cite[Lemma 4.3]{villani2009optimal}, 
$$
\mathcal{C}_p(\pi) \le \liminf_{n\to \infty} \mathcal{C}_p(\pi_n) \le \delta \Rightarrow \pi \in \overline \Pi_\delta(\mu_{x_{1:t-1}}, \cdot).
$$
Using the above argument with the constant sequence $x_{1:t}^{(n)} = x_{1:t}$ in particular shows that $\overline{\Pi}_\delta(\mu_{x_{1:t-1}}, \cdot)$ is compact. Thus we conclude from Definition \ref{def:uhc}  that $x_{1:t-1} \twoheadrightarrow \overline\Pi_\delta(\mu_{x_{1:t-1}}, \cdot)$ is upper hemicontinuous.
\end{proof}

To show measurability of the graph of $x_{1:t} \twoheadrightarrow \Pi^{\operatorname{T}}_\delta(\mu_{x_{1:t}}, \cdot)$ we recall the following definition from \cite{eder2019compactness}:

\begin{definition}
Let $\pi \in \mathcal{P}_p(\R^2)$ and $\delta>0$. Then the modulus of continuity $\omega_p(\pi, 
\delta)$ is defined as
$$
\omega_p(\pi, \delta) = \sup_{\substack{\gamma \in \Pi(\pi, \pi)\\ \mathcal{C}_p((\operatorname{proj}^{1, 3})_\# \gamma) \leq \delta}} \mathcal{C}_p((\operatorname{proj}^{2, 4})_\# \gamma).
$$
\end{definition}

\begin{lemma} \label{lem:meas}
For any $\mu \in \mathcal{P}_p(\mathbb{R}^N)$ the graph of $x_{1:t} \twoheadrightarrow \Pi_\delta (\mu_{x_{1:t}}, \cdot)$ is Borel measurable.
\end{lemma}

\begin{proof}
First, we check that the set $\{(\mu, \gamma)\,:\, \gamma \in \Pi(\mu, \cdot)\}$ is closed with respect to $\mathcal{W}_p$. We have
$$
\mu_n \to \mu, \; \Pi(\mu_n, \cdot) \ni \gamma_n \to \gamma \;\; \text{in} \; \mathcal{W}_p \Rightarrow \gamma^1 = \lim_{n \to \infty} \gamma_n^1 = \lim_{n \to \infty} \mu_n = \mu \Rightarrow \gamma \in \Pi(\mu, \cdot).
$$
The mapping $\Psi: (x_{1:t}, \gamma) \mapsto (\mu_{x_{1:t}}, \gamma)$ is Borel measurable, and hence
$$
\Psi^{-1}\left(\left\{(\mu, \gamma)\,:\, \gamma \in \Pi(\mu, \cdot)\right\}\right) = \operatorname{graph}(x_{1:t} \twoheadrightarrow \Pi(\mu_{x_{1:t}}, \cdot))
$$
is Borel. Moreover, the set $\{\gamma \in \mathcal{P}(\mathbb{R}^2)\,:\, \mathcal{C}_p(\gamma) < \delta\}$ is open and in particular Borel, because its complement is closed:
$$
\gamma_n \to \gamma \; \text{and} \; \mathcal{C}_p(\gamma_n) \geq \delta \Rightarrow \mathcal{C}_p(\gamma) \geq \delta
$$
by the characterisation of convergence in $\mathcal{W}_p$ from \cite[Theorem 6.9]{villani2009optimal}. Finally,
\begin{align*}
&\operatorname{graph}(x_{1:t} \twoheadrightarrow \Pi_\delta(\mu_{x_{1:t}}, \cdot))\\
&= \operatorname{graph}(x_{1:t} \twoheadrightarrow \Pi(\mu_{x_{1:t}}, \cdot)) \cap \left[\mathbb{R}^t \times \{\gamma \in \mathcal{P}(\mathbb{R}^2)\,:\, \mathcal{C}_p(\gamma) < \delta\}\right]
\end{align*}
is Borel, since it is an intersection of Borel sets.
\end{proof}

\begin{lemma}\label{lem:analytic_graph}
For any $\mu \in \mathcal{P}_p(\mathbb{R}^N)$ the graph of $x_{1:t} \twoheadrightarrow \Pi^{\operatorname{T}}_\delta(\mu_{x_{1:t}}, \cdot)$ is Borel measurable.
\end{lemma}
\begin{proof}
First, we the mapping $(\pi, \delta) \mapsto \omega_p(\pi, \delta)$ is Borel as an envelope of a continuous functional over continuous correspondence. Moreover, for any $\pi \in \mathcal{P}_p(\R^2)$ we have
$$
\lim_{\delta \downarrow 0} \omega_p(\pi, \delta) = 0 \quad \Leftrightarrow\quad \pi = (x, T(x))_\# \mu
$$
for some $\mu \in \mathcal{P}_p(\R)$ and Borel mapping $T: \R \to \R$, see \cite[Lemma 2.7]{eder2019compactness}. Since the mapping $\pi \mapsto \lim_{\delta \downarrow 0} \omega_p(\pi, \delta)$ is Borel measurable as a pointwise limit of Borel functions $(\pi, \delta) \mapsto \omega_p(\pi, \delta)$, we conclude that
\begin{align*}
\{\pi \in \mathcal{P}_p(\R^2): \lim_{\delta \downarrow 0} \omega_p(\pi, \delta)=0\}
\end{align*}
is Borel as the preimage of a Borel measurable function. Combining these two items with Lemma \ref{lem:meas},
\begin{align*}
&\operatorname{graph}(x_{1:t} \twoheadrightarrow \Pi^{\operatorname{T}}_\delta(\mu_{x_{1:t}}, \cdot))\\
&= \operatorname{graph}(x_{1:t} \twoheadrightarrow \Pi_\delta(\mu_{x_{1:t}}, \cdot)) \cap \left[\R^t \times \{\pi \in \mathcal{P}_p(\R^2): \lim_{\delta \downarrow 0} \omega_p(\pi, \delta)=0\}\right]
\end{align*}
is Borel.
\end{proof}

\subsection{Proof of Theorem \ref{equiv:basic}}
We now give the proof of Theorem \ref{equiv:basic}. Recall that we separate the control-free case from the controlled case stated in Theorem \ref{equiv:control_open}, since the former holds without any regularity assumption on the reference measure \(\mu \in \mathcal{P}_p(\mathbb{R}^N)\). The proof uses an approximation argument via Lusin's Theorem.

\begin{proof}[Proof of Theorem \ref{equiv:basic}.]
Let us first assume that $f$ is bounded and Lipschitz. We aim to prove that
\begin{equation}\label{proof_equiv:basic_cost_to_go}
V_t^\delta(x_{1:t}, y_{1:t}) = \sup_{\nu \in B_\delta(\bar{\mu}_{x_{1:t}})} \int f(y) \,\nu(dy_{t+1:N})
\end{equation}
holds for all $t=0, \dots, N$. We proceed via backward induction. The statement holds for $t=N$ as \(V_N^\delta(x, y) = f(y)\) by definition. Suppose now that \eqref{proof_equiv:basic_cost_to_go} holds for some \(t \leq N\), and fix $(x_{1:t-1}, y_{1:t-1}) \in \operatorname{spt}(\mu) \times \mathbb{R}^{t-1}$ throughout the rest of the proof.

We start with the proof of the ``$\leq$"--inequality of \eqref{proof_equiv:basic_cost_to_go}. The proof has two steps:
\begin{enumerate}
    \item First, we prove that
    \begin{align}\label{eqn:thm:basic:map}
    V^\delta_{t-1}(x_{1:t-1}, y_{1:t-1}) &= \sup_{\gamma^t \in \Pi_\delta(\mu_{x_{1:t-1}}, \cdot)} \int V^\delta_t(x_{1:t}, y_{1:t}) \, \gamma^t(dx_t, dy_t)\nonumber\\
    &= \sup_{\gamma^t \in \Pi^{\operatorname{T}}_\delta(\mu_{x_{1:t-1}}, \cdot)} \int V^\delta_t(x_{1:t}, y_{1:t}) \, \gamma^t(dx_t, dy_t).
    \end{align}
    \item The second step is to argue that
    \begin{equation}\label{eqh:thm:basic:less_or_equal}
    \sup_{\gamma^t \in \Pi^{\operatorname{T}}_\delta(\mu_{x_{1:t-1}}, \cdot)} \int V^\delta_t(x_{1:t}, y_{1:t}) \, \gamma^t(dx_t, dy_t) \leq \sup_{\nu \in B_\delta(\bar{\mu}_{x_{1:t-1}})} \int f(y) \,\nu(dy_{t:N})
    \end{equation}
    using the induction hypothesis.
\end{enumerate}

\begin{proof}[Proof of step (1).]
Fix $\varepsilon > 0$. By Lusin's Theorem \cite[Theorem 7.14.25]{bogachev2007measure} and Dugundji's Extension Theorem \cite[Theorem 4.1]{dugundji1951extension} we obtain a function $\bar{\mu}^\epsilon:\R \to \mathcal{P}_p(\mathbb{R}^{N - t})$ which satisfies
\begin{align}\label{eq:lusin}
(|\cdot|, \mathbb{R}) \ni x_{t} \mapsto \bar{\mu}^\varepsilon_{x_{1:t}} \in (\mathcal{AW}^\infty_p, \mathcal{P}_p(\mathbb{R}^{N - t})) \;\; \text{is continuous and}\;\;\mu_{x_{1:t-1}} \{x_{t}:\; \bar{\mu}_{x_{1:t}} \neq \bar{\mu}^\varepsilon_{x_{1:t}}\} < \varepsilon. 
\end{align}
This allows to define
\[
V^{\delta, \varepsilon}_t(x_{1:t}, y_{1:t}) := \sup_{\nu \in B_\delta(\bar{\mu}^\varepsilon_{x_{1:t}})} \int f(y)\, \nu(dy_{t+1:N}),
\]
which is a lower semicontinuous function by the maximum theorem \cite[Lemma 17.29]{guide2006infinite}, since $x_t \mapsto \bar{\mu}^\epsilon_{x_{1:t}}$ is continuous by construction (and thus lower hemi-continuity of $x_{1:t}\twoheadrightarrow B_\delta(\bar{\mu}^\epsilon_{x_{1:t}})$ follows, see Definition \ref{def:lower_hemicontinuous}) and $f$ is lower semicontinuous. Moreover, for any transport plan $\gamma^t \in \Pi(\mu_{x_{1:t-1}}, \cdot)$ we have
\begin{align*}
\gamma^t\{(x_{t}, y_{t}):\;V_{t}^\delta(x_{1:t}, y_{1:t}) \neq V^{\delta, \varepsilon}_t(x_{1:t}, y_{1:t})\} &\leq \gamma^t\{(x_{t}, y_{t}):\; \bar{\mu}_{x_{1:t}} \neq \bar{\mu}^\varepsilon_{x_{1:t}}\}\\
&= \mu_{x_{1:t-1}}\{x_{t}:\;\bar{\mu}_{x_{1:t}} \neq \bar{\mu}^\varepsilon_{x_{1:t}}\} < \varepsilon.
\end{align*}
Hence, using $|\sup A - \sup B| \leq \sup |A - B|$ and boundedness of $f$ we conclude
\begin{align*}
&\left|\sup_{\gamma^t \in \Pi_\delta(\mu_{x_{1:t-1}}, \cdot)} \int V_t^\delta \,d\gamma^t - \sup_{\gamma^t \in \Pi_\delta(\mu_{x_{1:t-1}}, \cdot)} \int V_t^{\delta, \varepsilon}\,d\gamma^t\right| \leq 2 \max_{y \in \mathbb{R}^N} |f(y)| \cdot \varepsilon,\\
&\left|\sup_{\gamma^t \in \Pi^{\operatorname{T}}_\delta(\mu_{x_{1:t-1}}, \cdot)} \int V_t^\delta \,d\gamma^t - \sup_{\gamma^t \in \Pi^{\operatorname{T}}_\delta(\mu_{x_{1:t-1}}, \cdot)} \int V_t^{\delta, \varepsilon}\,d\gamma^t\right| \leq 2 \max_{y \in \mathbb{R}^N} |f(y)| \cdot \varepsilon.
\end{align*}
Crucially, we have $\sup_{\gamma^t \in \Pi_\delta(\mu_{x_{1:t-1}}, \cdot)} \int V_t^{\delta, \varepsilon}\,d\gamma^t = \sup_{\gamma^t \in \Pi^{\operatorname{T}}_\delta(\mu_{x_{1:t-1}}, \cdot)} \int V_t^{\delta, \varepsilon}\,d\gamma^t$ by Lemma \ref{lem:supremum_closed}, because $V^{\delta, \varepsilon}_t$ is lower semicontinuous and bounded below. Hence, by the triangle inequality we conclude
\begin{equation*}
\left|\sup_{\gamma^t \in \Pi_\delta(\mu_{x_{1:t-1}}, \cdot)} \int V_t^\delta \,d\gamma^t - \sup_{\gamma^t \in \Pi^{\operatorname{T}}_\delta(\mu_{x_{1:t-1}}, \cdot)} \int V_t^\delta \,d\gamma^t\right| \leq 4 \max_{y \in \mathbb{R}^N} |f(y)| \cdot \varepsilon.
\end{equation*}
As $\varepsilon > 0$ was arbitrary, we conclude that \eqref{eqn:thm:basic:map} holds.
\end{proof}
\begin{proof}[Proof of step (2).]
Recall that according to the induction hypothesis,
$$
V^\delta_t(x_{1:t}, y_{1:t}) = \sup_{\nu \in B_\delta(\bar{\mu}_{x_{1:t}})} \int f(y) \, \nu(dy_{t+1:N}),
$$
and our goal is to establish
\begin{align*}
\sup_{\gamma^t \in \Pi_\delta^{\operatorname{T}}(\mu_{x_{1:t-1}}, \cdot)} \int V^\delta_t(x_{1:t}, y_{1:t}) \, \gamma^t (dx_t, dy_t)=V^\delta_{t-1}(x_{1:t-1}, y_{1:t-1}) \leq \sup_{\nu \in B_\delta(\bar{\mu}_{x_{1:t-1}})} \int f(y) \,\nu(dy_{t:N}).
\end{align*}
Fix $\varepsilon \in (0,\delta)$. According to Lemma \ref{lem:cost_to_go:lipschitz}, we have
\begin{equation}\label{eqn:dpp_basic:lipschitz}
V^\delta_t(x_{1:t}, y_{1:t}) \leq V^{\delta - \varepsilon}_t(x_{1:t}, y_{1:t}) + 2^{N - t} L \varepsilon,
\end{equation}
where $L > 0$ is a Lipschitz constant for $f$. Now we aim to choose universally measurable $\varepsilon$-selectors for $V^{\delta - \varepsilon}_t(x_{1:t}, y_{1:t})$. We use a standard measurable selection result for this, so we shall be short: note that the graph of the correspondence $x_{1:t-1} \twoheadrightarrow \Pi_{\delta - \varepsilon}^{\operatorname{T}}(\mu_{x_{1:t-1}}, \cdot)$ is analytic by Lemma \ref{lem:analytic_graph}. Then by the representation
\begin{align*}
V_t^{\delta-\epsilon}(x_{1:t},y_{1:t})=\sup_{\gamma^t \in \Pi_\delta^{\operatorname{T}}(\mu_{x_{1:t-1}}, \cdot)} \int V^{\delta-\epsilon}_t(x_{1:t}, y_{1:t}) \, \gamma^t (dx_t, dy_t)    
\end{align*}
established in Step 1 and
a  backward induction argument, $V_{t-1}^{\delta - \varepsilon}$ is lower semianalytic (see \cite[Proposition 7.47, Proposition 7.48]{bertsekas1996stochastic}). Hence, by \cite[Proposition 7.50.(b)]{bertsekas1996stochastic} we can choose
\begin{equation}\label{eqn:thm:basic:less_or_equal:near_optimal_hypothesis}
\Pi_{\delta - \varepsilon}^{\operatorname{T}}(\mu_{x_{1:t-1}}, \cdot) \ni \gamma_{x_{1:t-1}, y_{1:t-1}} \in \varepsilon-\operatorname{argmax} \int V_{t}^{\delta - \varepsilon}(x_{1:t}, y_{1:t}) \, \gamma_{x_{1:t-1}, y_{1:t-1}}(dx_{t}, dy_{t}),
\end{equation}
where $(x_{1:t-1}, y_{1:t-1}) \mapsto \gamma_{x_{1:t-1}, y_{1:t-1}}$ is universally measurable. 
By \cite[Lemma 7.28.(c)]{bertsekas1996stochastic}, the map $(x_{1:t-1}, y_{1:t-1}) \mapsto \gamma_{x_{1:t-1}, y_{1:t-1}}$ can actually be chosen to be Borel measurable. By a simple backward induction argument we have
\begin{align*}
\sup_{\gamma^t \in \Pi_\delta^{\operatorname{T}}(\mu_{x_{1:t-1}}, \cdot)} \int V_{t}^{\delta - \varepsilon}(x_{1:t}, y_{1:t}) \, \gamma(dx_t, dy_t) \leq \int f(y) \, \bar{\gamma}_{x_{1:t-1}, y_{1:t-1}}(dy_{t:N}) + (N - t) \varepsilon,
\end{align*}
where $\bar{\gamma}_{x_{1:t-1}, y_{1:t-1}} := \gamma_{x_{1:t-1}, y_{1:t-1}} \otimes \ldots \otimes \gamma_{x_{1:N-1}, y_{1:N-1}}$ and $\gamma_{x_{1:t}, y_{1:t}}, \dots, \gamma_{x_{1:N-1}, y_{1:N-1}}$ have been constructed in previous steps of the backward induction. By construction, $\bar{\gamma}_{x_{1:t-1}, y_{1:t-1}}(dy_{t:N}) \in B_\delta(\bar{\mu}_{x_{1:t-1}})$, and accounting for the Lipschitz continuity \eqref{eqn:dpp_basic:lipschitz} we obtain
\begin{align*}
V_{t-1}^\delta(x_{1:t-1}, y_{1:t-1}) &= \sup_{\gamma^t \in \Pi_\delta^{\operatorname{T}}(\mu_{x_{1:t-1}}, \cdot)} \int V_t^\delta(x_{1:t}, y_{1:t}) \, \gamma^t(dx_t, dy_t)\\
&\leq \sup_{\gamma^t \in \Pi_\delta^{\operatorname{T}}(\mu_{x_{1:t-1}}, \cdot)} \int V_t^{\delta - \varepsilon}(x_{1:t}, y_{1:t}) \, \gamma^t(dx_t, dy_t) + 2^{N - t} L \varepsilon\\
&\leq \int f(y) \, \bar{\gamma}_{x_{1:t-1}, y_{1:t-1}}(dy_{t:N}) + (N - t + 2^{N - t} L) \varepsilon\\
&\leq \sup_{\nu \in B_\delta(\bar{\mu}_{x_{1:t-1}})} \int f(y) \, \nu(dy_{t:N}) + (N - t + 2^{N - t} L) \varepsilon.
\end{align*}
This completes the proof of Step 2.
\end{proof}

Combining the two steps, we obtain
\begin{align*}
V^\delta_{t-1}(x_{1:t-1}, y_{1:t-1}) &= \sup_{\gamma^t \in \Pi_\delta(\mu_{x_{1:t-1}}, \cdot)} \int V^\delta_t(x_{1:t}, y_{1:t}) \, \gamma^t(dx_t, dy_t)\nonumber\\
&= \sup_{\gamma^t \in \Pi^{\operatorname{T}}_\delta(\mu_{x_{1:t-1}}, \cdot)} \int V^\delta_t(x_{1:t}, y_{1:t}) \, \gamma^t(dx_t, dy_t)\\
&\leq \sup_{\nu \in B_\delta(\bar{\mu}_{x_{1:t-1}})} \int f(y) \,\nu(dy_{t:N}),
\end{align*}
which proves ``$\leq$"--inequality of \eqref{proof_equiv:basic_cost_to_go}.

\bigskip

\noindent To prove the ``$\geq$"--inequality we again fix $(x_{1:t-1}, y_{1:t-1})\in \R^{t-1}\times\R^{t-1}$, an arbitrary \(\varepsilon > 0\) and  a measure \(\eta \in B_\delta(\bar{\mu}_{x_{1:t-1}})\) that satisfies
\begin{equation}\label{sup_def}
\sup_{ \nu \in B_\delta(\bar{\mu}_{x_{1:t-1}})} \int f(y)\,  \nu(dy_{t:N}) <  \int f(y) \,\eta(dy_{t:N})+ \varepsilon.
\end{equation}
Now we disintegrate the measure \(\eta(dy_{t:N}) = \eta(dy_t) \otimes \bar{\eta}_{y_t}(dy_{t+1:N})\) and use Lemma \ref{lem:adap_new} to find a transport plan \(\gamma^t \in \Pi(\mu_{x_{1:t-1}}, \eta^t )\) that satisfies
\[
\mathcal{C}_p(\gamma^t) \vee \|\mathcal{AW}^\infty_p(\bar{\mu}_{x_{1:t}}, \bar{\eta}_{y_t})\|_{L^\infty(\gamma^t)} < \delta.
\]
Then we rewrite the right-hand side of \eqref{sup_def} as follows:
\begin{align*}
\int f(y) \,\eta(dy_{t:N}) &= \int f(y) \,\bar{\eta}_{y_t} (dy_{t+1:N})\, \gamma^t(dx_t, dy_t)\\
&\leq \int  \sup_{\nu \in B_\delta(\bar{\mu}_{x_{1:t}})} \int f(y) \,\nu(dy_{t+1:N}) \,\gamma^t(dx_t, dy_t)\\
&\stackrel{\eqref{proof_equiv:basic_cost_to_go}}{=} \int V_t^\delta(x_{1:t}, y_{1:t}) \,\gamma^t(dx_t, dy_t)\\
&\leq V^\delta_{t-1}(x_{1:t-1}, y_{1:t-1}).
\end{align*}
Combining this with \eqref{sup_def} finishes the proof of the ``$\ge$"--inequality.

We now generalize the statement from bounded Lipschitz continuous function to the lower semicontinuous \(f\), which is bounded from below, by taking an approximating sequence
\[
f_n \uparrow f, \;\; f_n \;\text{bounded and Lipschitz continuous};
\]
see \cite[Lemma 7.14.(a)]{bertsekas1996stochastic}.
Indeed, in this case we have
\begin{align}\label{eqn:dpp_basic:mon}
V_t^\delta(x_{1:t}, y_{1:t}) &= \sup_{\gamma^{t+1} \in \Pi_\delta(\mu_{x_{1:t}}, \cdot)} \int \ldots \sup_{\gamma^N \in \Pi_\delta(\mu_{x_{1:N-1}}, \cdot)} \int f(y) \,d\gamma^N d\gamma^{N-1} \ldots d\gamma^{t+1}\nonumber\\
&= \sup_{n\in \N} \sup_{\gamma^{t+1} \in \Pi_\delta(\mu_{x_{1:t}}, \cdot)} \int \ldots \sup_{\gamma^N \in \Pi_\delta(\mu_{x_{1:N-1}}, \cdot)} \int f_n(y) \,d\gamma^N d\gamma^{N-1} \ldots d\gamma^{t+1}\nonumber\\
&= \sup_{n\in \N} \sup_{\nu \in B_\delta(\bar{\mu}_{x_{1:t}})} \int f_n(y) \,\nu(dy_{t+1:N})\nonumber\\
&= \sup_{\nu \in B_\delta(\bar{\mu}_{x_{1:t}})} \int f(y) \,\nu(dy_{t+1:N})
\end{align}
by successive applications of the Monotone Convergence Theorem.

To finish the proof of the theorem, it remains to generalize the statement to lower semicontinuous $f$, which satisfies $|f(x)| \geq -C(1 + \|x\|^{p - \varepsilon})$ for some $\varepsilon > 0$ and constant $C > 0$. To that end, assume that $f$ is bounded from above at first, and let
$$
f_n(x) := f(x) \vee (-n).
$$
By definition $f_n$ is lower semicontinuous as a maximum of two lower semicontinuous functions, and bounded from below. Hence, the dynamic programming principle holds for $f_n$. It remains to prove that
\begin{align}\label{eqn:dpp_basic:extension_dro}
\lim_{n \to \infty} \sup_{\nu \in B_\delta(\mu)} \int f_n(y) \, \nu(dy) = \sup_{\nu \in B_\delta(\mu)} \int f(y) \, \nu(dy),
\end{align}
and
\begin{align}\label{eqn:dpp_basic:extension_dpp}
&\lim_{n\to \infty} \sup_{\gamma^{1} \in \Pi_\delta(\mu^1, \cdot)} \int \ldots \sup_{\gamma^N \in \Pi_\delta(\mu_{x_{1:N-1}}, \cdot)} \int f_n(y) \,d\gamma^N \ldots d\gamma^{1}\nonumber\\
&= \sup_{\gamma^{1} \in \Pi_\delta(\mu^1, \cdot)} \int \ldots \sup_{\gamma^N \in \Pi_\delta(\mu_{x_{1:N-1}}, \cdot)} \int f(y) \,d\gamma^N \ldots d\gamma^{1}.
\end{align}
\end{proof}

\begin{proof}[Proof of \eqref{eqn:dpp_basic:extension_dro}]
By $|\sup A - \sup B| \leq \sup |A - B|$ we conclude
$$
\left|\sup_{\nu \in B_\delta(\mu)} \int f_n(y) \, \nu(dy) - \sup_{\nu \in B_\delta(\mu)} \int f(y) \, \nu(dy)\right| \le \sup_{\nu \in B_\delta(\mu)} \left|\int f_n(y) - f(y)\,\nu(dy)\right|.$$
Take now any $\nu \in B_\delta(\mu)$. By Lemma \ref{lem:easy} we have
$$
\mathcal{W}_p(\mu, \nu) \leq \mathcal{AW}_p(\mu, \nu) \leq N^{1/p} \mathcal{AW}_p^\infty(\mu, \nu) < N^{1/p} \delta,
$$
hence there exists a transport plan $\gamma \in \Pi(\mu, \nu)$, which satisfies
$$
\left(\int \|x - y\|^p \, \gamma(dx, dy)\right)^{1/p} < N^{1/p} \delta.
$$
Therefore,
\begin{align}\label{eqn:dpp_basic:extension_dro_first}
\int |f_n(y) - f(y)|\, \nu(dy) &= \int_{\{f < -n\}} |f(y) + n|\, \nu(dy)\nonumber\\
&\lesssim \int_{\{f < -n\}} (1 + \|y\|^{p - \varepsilon}) \, \nu(dy)\nonumber\\
&\lesssim \|1 + \|y\|^{p - \varepsilon}\|_{L^{\frac{p}{p - \varepsilon}}(\gamma)} \cdot \left(\nu\{y\,:\, f(y) < -n\}\right)^\frac{\varepsilon}{p}
\end{align}
by H\"older's inequality. The set $\nu\{y\,:\,f(y) < -n\}$ can be uniformly bounded from above by
\begin{align}\label{eqn:dpp_basic:extension_dro_measure_bound}
\nu\{y\,:\,f(y) < -n\} &\leq \nu\{y\,:\,C(1 + \|y\|^{p - \varepsilon}) > n\}\nonumber\\
&\lesssim n^{-\frac{p}{p - \varepsilon}} \left(\int (1 + \|y\|^{p - \varepsilon})\, \nu(dy)\right)^\frac{p}{p - \varepsilon}\nonumber\\
&\lesssim n^{-\frac{p}{p - \varepsilon}} \int (1 + \|y\|^p)\, \gamma(dx, dy)\nonumber\\
&\lesssim n^{-\frac{p}{p - \varepsilon}} \left(1 + \int (\|x\|^p + \|x - y\|^p)\, \gamma(dx, dy)\right)\nonumber\\
&\lesssim n^{-\frac{p}{p - \varepsilon}} \left(1 + \int \|x\|^p \, \mu(dx) + N \delta^p\right)\nonumber\\
&\lesssim n^{-\frac{p}{p - \varepsilon}},
\end{align}
where the second inequality follows from Markov's inequality, third one is the application of Jensen's inequality and the next ones follow from $(|a| + |b|)^p \leq 2^{p-1}(|a|^p + |b|^p)$ and the definition of $\gamma$. Similarly,
\begin{align}\label{eqn:dpp_basic:extension_dro_norm_bound}
\|1 + \|y\|^{p - \varepsilon}\|_{L^{\frac{p}{p - \varepsilon}}(\gamma)} \lesssim \left(\int (1 + \|y\|^p) \, \gamma(dx, dy)\right)^\frac{p - \varepsilon}{p} \lesssim 1.
\end{align}
Combining \eqref{eqn:dpp_basic:extension_dro_first}, \eqref{eqn:dpp_basic:extension_dro_measure_bound} and \eqref{eqn:dpp_basic:extension_dro_norm_bound} we obtain
\begin{align}\label{eqn:dpp_basic:extension_dro_final_estimate}
\int |f_n(y) - f(y)|\, \nu(dy) \lesssim n^{-\frac{\varepsilon}{p - \varepsilon}}.
\end{align}
This completes the proof.

\begin{proof}[Proof of \eqref{eqn:dpp_basic:extension_dpp}]
Applying the inequalities $|\sup A - \sup B| \leq \sup |A - B|$ and $|\int f| \leq \int |f|$ repeatedly, we obtain
\begin{align}\label{eqn:dpp_basic:extension_dpp_objective}
&\left|\sup_{\gamma^{1} \in \Pi_\delta(\mu^1, \cdot)} \int \ldots \sup_{\gamma^N \in \Pi_\delta(\mu_{x_{1:N-1}}, \cdot)} \int f_n(y) \,d\gamma^N \ldots d\gamma^{1}\right.\nonumber\\
&\left.- \sup_{\gamma^{1} \in \Pi_\delta(\mu^1, \cdot)} \int \ldots \sup_{\gamma^N \in \Pi_\delta(\mu_{x_{1:N-1}}, \cdot)} \int f(y) \,d\gamma^N \ldots d\gamma^{1}\right|\nonumber\\
&\leq \sup_{\gamma^{1} \in \Pi_\delta(\mu^1, \cdot)} \int \ldots \sup_{\gamma^N \in \Pi_\delta(\mu_{x_{1:N-1}}, \cdot)} \int |f_n(y) - f(y)| \, d\gamma^N \ldots d\gamma^1.
\end{align}
Throughout the rest of the proof we shall estimate \eqref{eqn:dpp_basic:extension_dpp_objective} from above. Applying Proposition \ref{prop:epsilon_optimizers} with $g = -|f_n - f|$ and recalling Step 1 of the proof, we get
\begin{align*}
\sup_{\gamma \in \Pi_\delta(\mu, \cdot)} \int |f_n(y) - f(y)| \, \gamma(dx, dy) = \sup_{\gamma^{1} \in \Pi_\delta(\mu^1, \cdot)} \int \ldots \sup_{\gamma^N \in \Pi_\delta(\mu_{x_{1:N-1}}, \cdot)} \int |f_n(y) - f(y)| \, d\gamma^N \ldots d\gamma^1.
\end{align*}
By definition of $\Pi_\delta(\mu, \cdot)$, for any transport plan $\gamma \in \Pi_\delta(\mu, \cdot)$ we have
$$
\int \|x - y\|^p \, \gamma(dx, dy) \leq N \delta^p,
$$
and using the same estimate as in \eqref{eqn:dpp_basic:extension_dro_final_estimate} we obtain
$$
\sup_{\gamma \in \Pi_\delta(\mu, \cdot)} \int |f_n(y) - f(y)| \, \gamma(dx, dy) \lesssim n^{-\frac{\varepsilon}{p - \varepsilon}} \to 0, \;\; n \to +\infty,
$$
which completes the proof.
\end{proof}

Using \eqref{eqn:dpp_basic:extension_dro} and \eqref{eqn:dpp_basic:extension_dpp} and the previous result for lower semicontinuous functions bounded from below, we conclude
\begin{align*}
&\sup_{\gamma^{1} \in \Pi_\delta(\mu^1, \cdot)} \int \ldots \sup_{\gamma^N \in \Pi_\delta(\mu_{x_{1:N-1}}, \cdot)} \int f(y) \,d\gamma^N \ldots d\gamma^{1}\\
&= \lim_{n \to +\infty} \sup_{\gamma^{1} \in \Pi_\delta(\mu^1, \cdot)} \int \ldots \sup_{\gamma^N \in \Pi_\delta(\mu_{x_{1:N-1}}, \cdot)} \int f_n(y) \,d\gamma^N \ldots d\gamma^{1}\\
&= \lim_{n \to +\infty} \sup_{\nu \in B_\delta(\mu)} \int f_n(y) \, \nu(dy)\\
&= \sup_{\nu \in B_\delta(\mu)} \int f(y) \, \nu(dy).
\end{align*}

To relax the boundedness from above we follow \eqref{eqn:dpp_basic:mon} line by line by considering $f_n(x) := f(x) \wedge n$, hence the proof of the first part of the theorem is complete.
\end{proof}

\subsection{Proofs of Lemma \ref{lem:cost_to_go_regularity}, Theorem  \ref{equiv:control_open}, Corollary \ref{cor:linear} and Theorem \ref{minimax}}

\begin{proof}[Proof of Lemma \ref{lem:cost_to_go_regularity}.]
We proceed by induction. Evidently, the claim is true for \(V^{\delta}_N(x, y, \alpha) = f(y, \alpha)\) by assumption. Suppose that the statement holds for \(V^{\delta}_{t+1}\), where \(2\le t + 1 \leq N\). Recall that by \eqref{dpp:control} we have
\begin{align*}
V_t^\delta(x_{1:t}, y_{1:t}, \alpha_{1:t}) &= \inf_{\alpha_{t+1} \in K} \sup_{\gamma^{t+1} \in \Pi_\delta(\mu_{x_{1:t}}, \cdot)} \int V^\delta_{t+1}(x_{1:t+1}, y_{1:t+1}, \alpha_{1:t+1})\, \gamma^{t+1}(dx_{t+1}, dy_{t+1})\\
&= \inf_{\alpha_{t+1} \in K} \sup_{\gamma^{t+1} \in \overline \Pi_\delta(\mu_{x_{1:t}}, \cdot)} \int V^\delta_{t+1}(x_{1:t+1}, y_{1:t+1}, \alpha_{1:t+1})\, \gamma^{t+1}(dx_{t+1}, dy_{t+1}),
\end{align*}
where the final equality follows from Lemma \ref{lem:supremum_closed}. The plan is to prove the statement for the mapping \((x_{1:t}, y_{1:t}, \alpha_{1:t+1}) \mapsto \sup_{\gamma^{t+1} \in \Pi_\delta(\mu_{x_{1:t}}, \cdot)} \int V^\delta_{t+1} d\gamma^{t+1}\) and then use compactness of $K$ to extend the regularity to \((x_{1:t}, y_{1:t}, \alpha_{1:t}) \mapsto V_t^\delta(x_{1:t}, y_{1:t}, \alpha_{1:t})\). We now consider two cases:

\bigskip

Suppose that $V^\delta_{t+1}$ is lower semicontinuous and bounded from below. Then by \cite[Lemma 4.3]{villani2009optimal} applied with $c = V^\delta_{t+1}$ and $h= 0$, and Proposition \ref{prop:correcpondence_lhs} we have
\begin{align*}
    &(x_{1:t}, y_{1:t}, \alpha_{1:t+1}, \gamma) \mapsto \int V^\delta_{t+1}(x_{1:t+1}, y_{1:t+1}, \alpha_{1:t+1}) \,\gamma(dx_{t+1}, dy_{t+1}) \;\; \text{is lsc. and bounded below,}\\
    &(x_{1:t}, y_{1:t}, \alpha_{1:t+1}) \twoheadrightarrow \overline{\Pi}_\delta(\mu_{x_{1:t}}, \cdot) \;\; \text{is lower hemicontinuous in} \; (\mathcal{W}_{p}, \mathcal{P}_{p}(\mathbb{R}^2)).
\end{align*}
Then the mapping
$$
(x_{1:t}, y_{1:t}, \alpha_{1:t+1}) \mapsto \sup_{\gamma^{t+1} \in \overline{\Pi}_\delta(\mu_{x_{1:t}}, \cdot)} \int V^\delta_{t+1}(x_{1:t+1}, y_{1:t+1}, \alpha_{1:t+1}) \,\gamma^{t+1}(dx_{t+1}, dy_{t+1})
$$
is lower semicontinuous and bounded from below by the maximum theorem \cite[Lemma 17.29]{guide2006infinite}. Applying \cite[Proposition 7.32.(a)]{bertsekas1996stochastic} we conclude that $V^\delta_t$ is also lower semicontinuous and bounded from below as an envelope over compact set.

\bigskip

Suppose now that $V^\delta_{t+1}$ is continuous and
$$
|V^\delta_{t+1}(x_{1:t+1}, y_{1:t+1}, \alpha_{1:t+1})| \lesssim 1 + \|x_{1:t+1}\|^{p - \varepsilon} + \|y_{1:t+1}\|^{p - \varepsilon}.
$$
Then for any $\gamma^{t+1} \in \overline \Pi_\delta(\mu_{x_{1:t}}, \cdot)$ and $\alpha_{t+1} \in K$ we have
\begin{align*}
&\int |V_{t+1}^{\delta}(x_{1:t+1}, y_{1:t+1}, \alpha_{1:t+1})| \,\gamma^{t+1}(dx_{t+1}, dy_{t+1})\\
&\lesssim 1 + \int \|x_{1:t+1}\|^{p - \varepsilon} + \|y_{1:t+1}\|^{p - \varepsilon}\, \gamma^{t+1}(dx_{t+1}, dy_{t+1})\\
&\lesssim 1 + \|x_{1:t}\|^{p - \varepsilon} + \|y_{1:t}\|^{p - \varepsilon} + \int |x_{t+1}|^{p - \varepsilon} + |y_{t+1}|^{p - \varepsilon} \,\gamma^{t+1}(dx_{t+1}, dy_{t+1})\\
&\lesssim 1 + \|x_{1:t}\|^{p - \varepsilon} + \|y_{1:t}\|^{p - \varepsilon} + \int |x_{t+1}|^{p - \varepsilon} \mu_{x_{1:t}}(dx_{t+1}) + \mathcal{C}_{p - \varepsilon}(\gamma^{t+1})^{p - \varepsilon}\\
&\lesssim 1 + \|x_{1:t}\|^{p - \varepsilon} + \|y_{1:t}\|^{p - \varepsilon},
\end{align*}
because of the definition of $\overline \Pi_\delta(\mu_{x_{1:t}}, \cdot)$ and the growth assumption on $\mu$. Therefore, the claimed growth for $V_{t}^{\delta}$ is proven, and it remains to show continuity. For this we conclude from Lemma \ref{sliced_continuity} and Proposition \ref{hemicontinuity}
\begin{align*}
    &(x_{1:t}, y_{1:t}, \alpha_{1:t+1}, \gamma) \mapsto \int V^\delta_{t+1}(x_{1:t+1}, y_{1:t+1}, \alpha_{1:t+1}) \,\gamma(dx_{t+1}, dy_{t+1}) \;\; \text{is continuous,}\\
    &(x_{1:t}, y_{1:t}, \alpha_{1:t+1}) \twoheadrightarrow \overline{\Pi}_\delta(\mu_{x_{1:t}}, \cdot) \;\; \text{is continuous in} \; (\mathcal{P}_{p - \varepsilon}(\mathbb{R}^2),\mathcal{W}_{p-\epsilon}),
\end{align*}
hence by Berge's maximum theorem \cite[Theorem 17.31]{guide2006infinite} the mapping
$$
(x_{1:t}, y_{1:t}, \alpha_{1:t+1}) \mapsto \sup_{\gamma^{t+1} \in \overline{\Pi}_\delta(\mu_{x_{1:t}}, \cdot)} \int V^\delta_{t+1}(x_{1:t+1}, y_{1:t+1}, \alpha_{1:t+1}) \,\gamma^{t+1}(dx_{t+1}, dy_{t+1})
$$
is continuous. Hence, $V^\delta_t$ is continuous as an envelope of a continuous function over a compact set (see \cite[Theorem 7.30]{guide2006infinite}). The proof is complete.
\end{proof}

The proof of Theorem \ref{equiv:control_open} is a consequence of topological properties of \(\mathcal{AW}^\infty_p\) and Lemma \ref{lem:cost_to_go_regularity}. In particular, we use the nested construction pointed out in the Introduction. Before proceeding with the proof, we state the following measurable selection argument:

\begin{proposition}\label{prop:epsilon_optimizers}
Let $\mu \in \mathcal{P}_p(\mathbb{R}^N)$ and $g: \mathbb{R}^N \times \mathbb{R}^N \to \mathbb{R}$ be a lower semianalytic function. Then for any $\varepsilon > 0$ the optimization problem
$$
\inf_{\gamma^1 \in \Pi_\delta^{\operatorname{T}}(\mu^1, \cdot)} \int \ldots \inf_{\gamma^N \in \Pi_\delta^{\operatorname{T}}(\mu_{x_{1:N-1}}, \cdot)} \int g(x, y) \, d\gamma^N \ldots d\gamma^1
$$
admits a selection of universally measurable $\varepsilon$-optimizers $(x_{1:t}, y_{1:t}) \mapsto \gamma_{x_{1:t}, y_{1:t}} \in \Pi^{\operatorname{T}}_\delta(\mu_{x_{1:t}}, \cdot)$, meaning that
\begin{align*}
&\int g(x, y) \, \gamma^{N}_{x_{1:N-1}, y_{1:N-1}}(dx_N, dy_N) \ldots \gamma^1(dx_1, dy_1)\\
&\leq \inf_{\gamma^1 \in \Pi_\delta^{\operatorname{T}}(\mu^1, \cdot)} \int \ldots \inf_{\gamma^N \in \Pi_\delta^{\operatorname{T}}(\mu_{x_{1:N-1}}, \cdot)} \int g(x, y) \, d\gamma^N \ldots d\gamma^1 + \varepsilon.
\end{align*}
Moreover, $(x_{1:t}, y_{1:t}) \mapsto \gamma_{x_{1:t}, y_{1:t}}$ can be chosen to be Borel measurable. The same statement holds if one replaces $\Pi^{\operatorname{T}}_\delta(\mu_{x_{1:t}}, \cdot)$ with $\Pi_\delta(\mu_{x_{1:t}}, \cdot)$.
\end{proposition}
\begin{proof}[Proof.]
Define
\begin{align*}
A_N(x, y) &:= g(x, y)\\
A_t(x_{1:t}, y_{1:t}) &:= \inf_{\gamma^{t+1} \in \Pi_\delta^{\operatorname{T}}(\mu_{x_{1:t}}, \cdot)} \int A_{t+1}(x_{1:t+1}, y_{1:t+1}) \, \gamma^{t+1}(dx_{t+1}, dy_{t+1}).
\end{align*}

We start by showing that $A_t$ is lower semianalytic by backward induction. Indeed, this clearly holds for $A_N = g$. Now suppose that $A_{t+1}$ is lower semianalytic. Then the mapping $(x_{1:t}, y_{1:t}, \gamma^{t+1}) \mapsto \int A_{t+1} \, d\gamma^{t+1}$ is lower semianalytic by \cite[Proposition 7.48]{bertsekas1996stochastic} applied with $f = A_{t+1}$ and $q = \gamma^{t+1}$. Hence, $A_t$ is lower semianalytic by \cite[Proposition 7.47]{bertsekas1996stochastic} applied with $f = \int A_{t+1} \, d\gamma^{t+1}$ and $D= \text{graph}(x_{1:t} \twoheadrightarrow \Pi_\delta^{\operatorname{T}}(\mu_{x_{1:t}}, \cdot))$, which is analytic by Lemma \ref{lem:analytic_graph}. 

Take $\varepsilon > 0$. Since $A_t$ is lower semianalytic, \cite[Proposition 7.50]{bertsekas1996stochastic} guarantees existence of  a universally measurable $\varepsilon$-optimizer $\gamma_{x_{1:t}, y_{1:t}} \in \Pi_\delta^{\operatorname{T}}(\mu_{x_{1:t}}, \cdot)$ for $A_t$, i.e.,
\begin{equation}\label{eqn:prop:epsilon_optimizers:epsilon_optimizer_step}
\int A_{t+1}(x_{1:t+1}, y_{1:t+1}) \, \gamma_{x_{1:t}, y_{1:t}}(dx_{t+1}, dy_{t+1}) \leq A_t(x_{1:t}, y_{1:t}) + \varepsilon.
\end{equation}
Moreover, iteratively applying \cite[Lemma 7.28.(c)]{bertsekas1996stochastic} with $p = \gamma^1 \otimes \ldots \otimes \gamma_{x_{1:t-1}, y_{1:t-1}}$ and $q = \gamma_{x_{1:t}, y_{1:t}}$ for $t=1, \dots, N-1,$ we obtain Borel measurable versions of $(x_{1:t}, y_{1:t}) \mapsto \gamma_{x_{1:t}, y_{1:t}}$. With this modification, \eqref{eqn:prop:epsilon_optimizers:epsilon_optimizer_step} holds $(\gamma^1 \otimes \ldots \otimes \gamma_{x_{1:t-1}, y_{1:t-1}})$-almost surely. Thus
\begin{equation}\label{eqn:prop:epsilon_optimizers:epsilon_optimizer_total}
\int g(x, y) \, \gamma_{x_{1:N-1}, y_{1:N-1}}(dx_N, dy_N) \ldots \gamma^1(dx_1, dy_1) \leq A_0 + N \varepsilon,
\end{equation}
which completes the proof of the claim for $\Pi_\delta^{\operatorname{T}}(\mu_{x_{1:t}}, \cdot)$. The proof for $\Pi_\delta(\mu_{x_{1:t}}, \cdot)$ follows the same arguments, except that analyticity of $\text{graph}(x_{1:t} \twoheadrightarrow\Pi_\delta(\mu_{x_{1:t}}, \cdot))$ is given by Lemma \ref{lem:meas}.
\end{proof}

\begin{proposition}\label{prop:measurable_selection}
Let $\mu \in \mathcal{P}_p(\mathbb{R}^N)$ and $g: \mathbb{R}^N \times \mathbb{R}^N \to \mathbb{R}$ be a lower semianalytic function. Then
\begin{align*}
\inf_{\gamma \in \Pi_\delta^{\operatorname{T}}(\mu, \cdot)} \int g(x, y) \, \gamma(dx, dy) = \inf_{\gamma^1 \in \Pi_\delta^{\operatorname{T}}(\mu^1, \cdot)} \int \ldots \inf_{\gamma^N \in \Pi_\delta^{\operatorname{T}}(\mu_{x_{1:N-1}}, \cdot)} \int g(x, y) \, d\gamma^N \ldots d\gamma^1
\end{align*}
Similarly,
\begin{align*}
\inf_{\gamma \in \Pi_\delta(\mu, \cdot)} \int g(x, y) \, \gamma(dx, dy) = \inf_{\gamma^1 \in \Pi_\delta(\mu^1, \cdot)} \int \ldots \inf_{\gamma^N \in \Pi_\delta(\mu_{x_{1:N-1}}, \cdot)} \int g(x, y) \, d\gamma^N \ldots d\gamma^1.
\end{align*}
\end{proposition}
\begin{proof}[Proof]
We start with ``$\geq$"--inequality. Take any transport plan $\gamma \in \Pi^{\operatorname{T}}_\delta(\mu, \cdot)$ and consider its disintegration:
$$
\gamma = \gamma^1 \otimes \ldots \otimes \gamma_{x_{1:N-1}, y_{1:N-1}}, \qquad  \gamma_{x_{1:t}, y_{1:t}} \in \Pi_\delta^{\operatorname{T}}(\mu_{x_{1:t}}, \cdot).
$$
By the disintegration theorem we have
\begin{align*}
\int g(x, y) \, \gamma(dx, dy) &= \int g(x, y) \, \gamma_{x_{1:N-1}, y_{1:N-1}}(dx_N, dy_N) \ldots \gamma^1(dx_1, dy_1)\\
&\geq \inf_{\gamma^1 \in \Pi_\delta^{\operatorname{T}}(\mu^1, \cdot)} \int \ldots \inf_{\gamma^N \in \Pi_\delta^{\operatorname{T}}(\mu_{x_{1:N-1}}, \cdot)} \int g(x, y) \, d\gamma^N \ldots d\gamma^1.
\end{align*}
Taking the supremum over $\gamma \in \Pi_\delta^{\operatorname{T}}(\mu, \cdot)$ on the left-hand side, we conclude ``$\geq$"--inequality. To prove the opposite inequality, we take $\varepsilon > 0$ and use Proposition \ref{prop:epsilon_optimizers} to obtain Borel measurable $\varepsilon$-optimizers $(x_{1:t}, y_{1:t}) \mapsto \gamma_{x_{1:t}, y_{1:t}}$ satisfying
\begin{align*}
\int g(x, y) \, \gamma(dx, dy) &= \int g(x, y) \, \gamma_{x_{1:N-1}, y_{1:N-1}}(dx_N, dy_N) \ldots \gamma^1(dx_1, dy_1)\\
&\leq \inf_{\gamma^1 \in \Pi_\delta^{\operatorname{T}}(\mu^1, \cdot)} \int \ldots \inf_{\gamma^N \in \Pi_\delta^{\operatorname{T}}(\mu_{x_{1:N-1}}, \cdot)} \int g(x, y) \, d\gamma^N \ldots d\gamma^1 + \varepsilon,
\end{align*}
where we set $\gamma := \gamma^1 \otimes \ldots \otimes \gamma_{x_{1:N-1}, y_{1:N-1}}$. By definition, $\gamma \in \Pi_\delta^{\operatorname{T}}(\mu, \cdot)$, and taking the infimum over $\Pi_\delta^{\operatorname{T}}(\mu, \cdot)$ on the left-hand side we obtain
$$
\inf_{\gamma \in \Pi_\delta^{\operatorname{T}}(\mu, \cdot)} \int g(x, y) \, \gamma(dx, dy) \leq A_0 + N \varepsilon = \inf_{\gamma^1 \in \Pi_\delta^{\operatorname{T}}(\mu^1, \cdot)} \int \ldots \inf_{\gamma^N \in \Pi_\delta^{\operatorname{T}}(\mu_{x_{1:N-1}}, \cdot)} \int g(x, y) \, d\gamma^N \ldots d\gamma^1 + N\varepsilon.
$$
As $\varepsilon > 0$ was arbitrary, this shows the claim for $\Pi_\delta^{\operatorname{T}}(\mu, \cdot)$. The proof for $\Pi_\delta(\mu, \cdot)$ follows the same arguments line by line.
\end{proof}

\begin{proof}[Proof of Theorem \ref{equiv:control_open} for $f$ bounded from below]
By the lower semi-continuity of $$ (x_{1:t}, y_{1:t}, \alpha_{1:t+1}) \mapsto \sup_{\gamma^{t+1} \in \overline\Pi_\delta(\mu_{x_{1:t}}, \cdot)} \int V^\delta_{t+1}(x_{1:t+1}, y_{1:t+1}, \alpha_{1:t+1}) \,\gamma^{t+1}(dx_{t+1}, dy_{t+1})$$ established in the proof of Lemma \ref{lem:cost_to_go_regularity} and recalling that $K$ is compact, \cite[Proposition 7.33]{bertsekas1996stochastic} yields a Borel measurable selector \((x_{1:t}, y_{1:t}, \alpha_{1:t}) \mapsto \alpha^\delta_{t+1}(x_{1:t}, y_{1:t}, \alpha_{1:t})\) for \(V^\delta_{t}\) , i.e.,
\begin{align*}
V_t^\delta(x_{1:t}, y_{1:t}, \alpha_{1:t}) &\stackrel{\text{Lemma } \ref{lem:supremum_closed}}{=} \sup_{\gamma^{t+1} \in \overline \Pi_\delta(\mu_{x_{1:t}}, \cdot)} \int V^\delta_{t+1}(x_{1:t+1}, y_{1:t+1}, \alpha^\delta_{1:t+1}(x_{1:t}, y_{1:t}, \alpha_{1:t})) \,\gamma^{t+1}(dx_{t+1}, dy_{t+1})\\
&\stackrel{\text{Lemma } \ref{lem:supremum_closed}}{=} \sup_{\gamma^{t+1} \in \Pi^{\operatorname{T}}_\delta(\mu_{x_{1:t}}, \cdot)} \int V^\delta_{t+1}(x_{1:t+1}, y_{1:t+1}, \alpha^\delta_{1:t+1}(x_{1:t}, y_{1:t}, \alpha_{1:t})) \,\gamma^{t+1}(dx_{t+1}, dy_{t+1}).
\end{align*}
By backward induction we obtain a predictable control $\alpha^\delta := (\alpha^\delta_1, \alpha^\delta_2, \ldots, \alpha^\delta_N) \in \mathcal{A}$. Thus
\begin{align}\label{eqn:thm:control_open:leq}
V^\delta_0 &= \sup_{\gamma^1 \in \Pi^{\operatorname{T}}_\delta(\mu^1, \cdot)} \int \sup_{\gamma^2 \in \Pi^{\operatorname{T}}_\delta(\mu_{x_1}, \cdot)} \int \ldots \sup_{\gamma^N \in \Pi^{\operatorname{T}}_\delta(\mu_{x_{1:N-1}}, \cdot)} \int f(y, \alpha^\delta) \,d\gamma^N \ldots d\gamma^1\nonumber \\
&\geq \inf_{\alpha \in \mathcal{A}} \sup_{\gamma^1 \in \Pi^{\operatorname{T}}_\delta(\mu^1, \cdot)} \int \sup_{\gamma^2 \in \Pi^{\operatorname{T}}_\delta(\mu_{x_1}, \cdot)} \int \ldots \sup_{\gamma^N \in \Pi^{\operatorname{T}}_\delta(\mu_{x_{1:N-1}}, \cdot)} \int f(y, \alpha) \,d\gamma^N \ldots d\gamma^1.
\end{align}
On the other hand, for any predictable control $\alpha = (\alpha_1, \alpha_2(x_1, y_1), \ldots, \alpha_N(x_{1:N-1}, y_{1:N-1})) \in \mathcal{A}$ we have
$$
V_t^\delta(x_{1:t}, y_{1:t}, \alpha_{1:t}) \leq \sup_{\gamma^{t+1} \in \Pi^{\operatorname{T}}_\delta(\mu_{x_{1:t}}, \cdot)} \int V^\delta_{t+1}(x_{1:t+1}, y_{1:t+1}, \alpha_{1:t}, \alpha_{t+1}(x_{1:t}, y_{1:t})) \,\gamma^{t+1}(dx_{t+1}, dy_{t+1})
$$
by definition of $V^\delta_t$. Iterating over the previous inequality yields
$$
V^\delta_0 \leq \sup_{\gamma^1 \in \Pi^{\operatorname{T}}_\delta(\mu^1, \cdot)} \int \sup_{\gamma^2 \in \Pi^{\operatorname{T}}_\delta(\mu_{x_1}, \cdot)} \int \ldots \sup_{\gamma^N \in \Pi^{\operatorname{T}}_\delta(\mu_{x_{1:N-1}}, \cdot)} \int f(y, \alpha) \,d\gamma^N \ldots d\gamma^1.
$$
Taking the infimum over $\alpha \in \mathcal{A}$, we obtain
\begin{align}\label{eqn:thm:control_open:geq}
V^\delta_0 \leq \inf_{\alpha \in \mathcal{A}} \sup_{\gamma^1 \in \Pi^{\operatorname{T}}_\delta(\mu^1, \cdot)} \int \sup_{\gamma^2 \in \Pi^{\operatorname{T}}_\delta(\mu_{x_1}, \cdot)} \int \ldots \sup_{\gamma^N \in \Pi^{\operatorname{T}}_\delta(\mu_{x_{1:N-1}}, \cdot)} \int f(y, \alpha) \,d\gamma^N \ldots d\gamma^1.
\end{align}
Combining \eqref{eqn:thm:control_open:leq} and \eqref{eqn:thm:control_open:geq},
\begin{align*}
V^\delta_0 &= \inf_{\alpha \in \mathcal{A}} \sup_{\gamma^1 \in \Pi^{\operatorname{T}}_\delta(\mu^1, \cdot)} \int \sup_{\gamma^2 \in \Pi^{\operatorname{T}}_\delta(\mu_{x_1}, \cdot)} \int \ldots \sup_{\gamma^{N} \in \Pi^{\operatorname{T}}_\delta(\mu_{x_{1:N-1}}, \cdot)} \int f(y, \alpha(x, y))\, d\gamma^N \ldots d\gamma^1\\
&= \inf_{\alpha \in \mathcal{A}} \sup_{\gamma \in \Pi^{\operatorname{T}}_{\delta}(\mu, \cdot)} \int f(y, \alpha(x, y))\, \gamma(dx, dy)\\
&\leq \inf_{\alpha \in \mathcal{A}} \sup_{\gamma \in \Pi_{\operatorname{bc}, \delta}(\mu, \cdot)} \int f(y, \alpha(x, y))\, \gamma(dx, dy)\\
&= V(\delta),
\end{align*}
where the second equality holds by Proposition \ref{prop:measurable_selection} applied with $g = -f$
and the inequality holds since \(\Pi^{\operatorname{T}}_{\delta}(\mu, \cdot) \subseteq \Pi_{\operatorname{bc}, \delta}(\mu, \cdot)\). On the other hand,
\[
\Pi_{\operatorname{bc}, \delta}(\mu, \cdot) \subseteq \{  \gamma^1 \otimes \gamma_{x_1, y_1} \otimes \ldots \otimes \gamma_{x_{1:N-1}, y_{1:N-1}}:\ \gamma_{x_{1:t-1},y_{1:t-1}} \in \Pi_\delta(\mu_{x_{1:t-1}}, \cdot), \;t =1,\dots, N\},
\]
where $(x_{1:t}, y_{1:t})\mapsto \gamma_{x_{1:t}, y_{1:t}}$ are Borel measurable functions, and hence we have the opposite bound
\begin{align*}
V(\delta) &= \inf_{\alpha \in \mathcal{A}} \sup_{\gamma \in \Pi_{\operatorname{bc}, \delta}(\mu, \cdot)} \int f(y, \alpha(x, y))\, \gamma(dx, dy)\\
&\leq \inf_{\alpha \in \mathcal{A}} \sup_{\gamma^1 \in \Pi_\delta(\mu^1, \cdot)} \int \sup_{\gamma^2 \in \Pi_\delta(\mu_{x_1}, \cdot)} \int \ldots \sup_{\gamma^N \in \Pi_\delta(\mu_{x_{1:N-1}}, \cdot)} \int f(y, \alpha(x, y))\, d\gamma^N \ldots d\gamma^1\\
&= V^\delta_0.
\end{align*}
This concludes the proof.
\end{proof}

\begin{proof}[Proof of Theorem \ref{equiv:control_open} for $|f(x, \alpha)| \lesssim 1 + \|x\|^{p - \varepsilon}$]
Consider
$$
f_n(x) = f(x) \vee (-n), \;\; n \in \mathbb{N}.
$$
This function is lower semicontinuous as a maximum of two lower semicontinuous functions, and is bounded from below by $-n$. Hence,
\begin{align*}
&\inf_{\alpha \in \mathcal{A}} \sup_{\gamma \in \Pi_{\operatorname{bc}, \delta}(\mu, \cdot)} \int f_n(y, \alpha(x, y)) \, \gamma(dx, dy)\\
&= \inf_{\alpha_1 \in K} \sup_{\gamma^1 \in \Pi_\delta(\mu^1, \cdot)} \int \ldots \inf_{\alpha_N \in K} \sup_{\gamma^N \in \Pi_\delta(\mu_{x_{1:N-1}}, \cdot)} \int f_n(y, \alpha(x, y)) \, d\gamma^N \ldots d\gamma^1
\end{align*}
by [Theorem \ref{equiv:control_open} for $f$ bounded from below]. It remains to show that
\begin{align}\label{eqn:thm:equiv_control:dro_convergence}
\lim_{n\to \infty}\inf_{\alpha \in \mathcal{A}} \sup_{\gamma \in \Pi_{\operatorname{bc}, \delta}(\mu, \cdot)} \int f_n(y, \alpha(x, y)) \, \gamma(dx, dy) = \inf_{\alpha \in \mathcal{A}} \sup_{\gamma \in \Pi_{\operatorname{bc}, \delta}(\mu, \cdot)} \int f(y, \alpha(x, y)) \, \gamma(dx, dy)
\end{align}
and
\begin{align}\label{eqn:thm:equiv_control:dpp_convergence}
\begin{split}
&\lim_{n\to \infty}\inf_{\alpha_1 \in K} \sup_{\gamma^1 \in \Pi_\delta(\mu^1, \cdot)} \int \ldots \inf_{\alpha_N \in K} \sup_{\gamma^N \in \Pi_\delta(\mu_{x_{1:N-1}}, \cdot)} \int f_n(y, \alpha(x, y)) \, d\gamma^N \ldots d\gamma^1\\
&= \inf_{\alpha_1 \in K} \sup_{\gamma^1 \in \Pi_\delta(\mu^1, \cdot)} \int \ldots \inf_{\alpha_N \in K} \sup_{\gamma^N \in \Pi_\delta(\mu_{x_{1:N-1}}, \cdot)} \int f(y, \alpha(x, y)) \, d\gamma^N \ldots d\gamma^1.
\end{split}
\end{align}
By applying repeatedly the inequalities $|\sup A - \sup B| \leq \sup |A - B|$, $|\inf A - \inf B| \leq \sup |A - B|$ and $|\int f| \leq \int |f|$ to \eqref{eqn:thm:equiv_control:dro_convergence} and \eqref{eqn:thm:equiv_control:dpp_convergence}, we obtain
\begin{align}\label{eqn:thm:equiv_control:dro_convergence_bound}
&\left|\inf_{\alpha \in \mathcal{A}} \sup_{\gamma \in \Pi_{\operatorname{bc}, \delta}(\mu, \cdot)} \int f_n(y, \alpha(x, y)) \, \gamma(dx, dy) - \inf_{\alpha \in \mathcal{A}} \sup_{\gamma \in \Pi_{\operatorname{bc}, \delta}(\mu, \cdot)} \int f(y, \alpha(x, y)) \, \gamma(dx, dy)\right|\nonumber\\
&\leq \sup_{\alpha \in \mathcal{A}} \sup_{\gamma \in \Pi_{\operatorname{bc}, \delta}(\mu, \cdot)} \int |f_n(y, \alpha(x, y)) - f(y, \alpha(x, y))| \, \gamma(dx, dy),
\end{align}
and
\begin{align}\label{eqn:thm:equiv_control:dpp_convergence_bound}
&\left|\inf_{\alpha_1 \in K} \sup_{\gamma^1 \in \Pi_\delta(\mu^1, \cdot)} \int \ldots \inf_{\alpha_N \in K} \sup_{\gamma^N \in \Pi_\delta(\mu_{x_{1:N-1}}, \cdot)} \int f_n(y, \alpha(x, y)) \, d\gamma^N \ldots d\gamma^1\right.\nonumber\\
&\left.- \inf_{\alpha_1 \in K} \sup_{\gamma^1 \in \Pi_\delta(\mu^1, \cdot)} \int \ldots \inf_{\alpha_N \in K} \sup_{\gamma^N \in \Pi_\delta(\mu_{x_{1:N-1}}, \cdot)} \int f(y, \alpha(x, y)) \, d\gamma^N \ldots d\gamma^1\right|\nonumber\\
&\leq \sup_{\alpha_1 \in K} \sup_{\gamma^1 \in \Pi_\delta(\mu^1, \cdot)} \int \ldots \sup_{\alpha_N \in K} \sup_{\gamma^N \in \Pi_\delta(\mu_{x_{1:N-1}}, \cdot)} \int |f_n(y, \alpha(x, y)) - f(y, \alpha(x, y))| \, d\gamma^N \ldots d\gamma^1.
\end{align}
We first show that both \eqref{eqn:thm:equiv_control:dro_convergence_bound} and \eqref{eqn:thm:equiv_control:dpp_convergence_bound} are bounded by
\begin{equation}\label{eqn:thm:equiv_control:common_bound}
\sup_{\alpha \in \mathcal{A}} \sup_{\gamma \in \Pi_\delta(\mu, \cdot)} \int |f_n(y, \alpha(x, y)) - f(y, \alpha(x, y))| \, \gamma(dx, dy).
\end{equation}
Indeed, \eqref{eqn:thm:equiv_control:dro_convergence_bound} is bounded by \eqref{eqn:thm:equiv_control:common_bound}, as $\Pi_{\operatorname{bc}, \delta}(\mu, \cdot) \subseteq \Pi_\delta(\mu, \cdot)$. For \eqref{eqn:thm:equiv_control:dpp_convergence_bound} we proceed with a measurable selection argument similar to the proof of Proposition \eqref{prop:epsilon_optimizers}: we define
\begin{align*}
&A_N(x, y, \alpha) := |f_n(y, \alpha) - f(y, \alpha)|\\
&A_t(x_{1:t}, y_{1:t}, \alpha_{1:t}) := \sup_{(\alpha_{t+1}, \gamma^{t+1}) \in K \times \Pi_\delta(\mu_{x_{1:t}}, \cdot)} \int A_{t+1}(x_{1:t+1}, y_{1:t+1}, \alpha_{1:t+1}) \, \gamma^{t+1}(dx_{t+1}, dy_{t+1}).
\end{align*}
Now we note that the set \text{graph}($x_{1:t} \twoheadrightarrow K \times \Pi_\delta(\mu_{x_{1:t}}, \cdot))$ is analytic by \cite[Proposition 7.40]{bertsekas1996stochastic} applied to $f: K \times \text{graph}(x_{1:t} \twoheadrightarrow\Pi_\delta(\mu_{x_{1:t}}, \cdot)) \ni (x, y) \mapsto y$ and $B = \text{graph}(x_{1:t} \twoheadrightarrow \Pi_\delta(\mu_{x_{1:t}}, \cdot))$, which is analytic by Proposition \ref{lem:meas}. We then argue by backward induction that $A_t$ is lower semianalytic: indeed, this holds for $A_N = |f_n - f|$. Suppose now that $A_{t+1}$ is lower semianalytic. Then $(x_{1:t}, y_{1:t}, \alpha_{1:t+1}, \gamma^{t+1}) \mapsto \int A_{t+1} \, d\gamma^{t+1}$ is lower semianalytic by \cite[Proposition 7.48]{bertsekas1996stochastic} applied with $f = A_{t+1}$ and $q = \gamma^{t+1}$. Next, $A_t$ is lower semianalytic by \cite[Proposition 7.47]{bertsekas1996stochastic}. Hence, taking $\varepsilon > 0$ and using \cite[Proposition 7.50.(b)]{bertsekas1996stochastic} we obtain universally measurable $\varepsilon$-optimizers $(x_{1:t}, y_{1:t}, \alpha_{1:t}) \mapsto (\alpha^\delta_{t+1}(x_{1:t}, y_{1:t}, \alpha_{1:t}), \gamma_{x_{1:t}, y_{1:t}, \alpha_{1:t}})$ of $A_t$, i.e.,
\begin{equation}\label{eqn:thm:equiv_control:dpp_convergence_bound_step}
A_t(x_{1:t}, y_{1:t}, \alpha_{1:t}) \leq \int A_{t+1}(x_{1:t+1}, y_{1:t+1}, \alpha_{1:t}, \alpha^\delta_{t+1}) \, \gamma_{x_{1:t}, y_{1:t}, \alpha_{1:t}}(dx_{t+1}, dy_{t+1}) + \varepsilon.
\end{equation}
In particular, the mappings $(x_{1:t}, y_{1:t}, \alpha_{1:t}) \mapsto \alpha^\delta_{t+1}(x_{1:t}, y_{1:t}, \alpha_{1:t})$ and $(x_{1:t}, y_{1:t}, \alpha_{1:t}) \mapsto \gamma_{x_{1:t}, y_{1:t}, \alpha_{1:t}}$ are universally measurable by \cite[Proposition 7.44]{bertsekas1996stochastic} applied to $f(x, y) = x$ and $f(x, y) = y$. Moreover, applying this proposition with
$$
f(x_{1:t}, y_{1:t}) = (x_{1:t}, y_{1:t}, \alpha^\delta_{1:t}(x_{1:t-1}, y_{1:t-1})) \;\; \text{and} \;\; g(x_{1:t}, y_{1:t}, \alpha_{1:t}) = \alpha^\delta(x_{1:t}, y_{1:t}, \alpha_{1:t})
$$
for $t=1,\dots,N$, we obtain universally measurable mappings $(x_{1:t}, y_{1:t}) \mapsto \alpha^\delta_{t+1}(x_{1:t}, y_{1:t})$. As the composition of universally measurable functions is again universally measurable (see \cite[Proposition 7.44]{bertsekas1996stochastic}), $(x_{1:t}, y_{1:t}) \mapsto \gamma_{x_{1:t}, y_{1:t}} := \gamma_{x_{1:t}, y_{1:t}, \alpha^\delta_{1:t}(x_{1:t-1}, y_{1:t-1})}$ is a universally measurable kernel. We now apply \cite[Lemma 7.28.(c)]{bertsekas1996stochastic} with $p = (\gamma^1 \otimes \ldots \otimes \gamma_{x_{1:t-1}, y_{1:t-1}})$ and $q = \gamma_{x_{1:t}, y_{1:t}}$ to obtain Borel measurable versions of $(x_{1:t}, y_{1:t}) \mapsto \gamma_{x_{1:t}, y_{1:t}}$. We thus have
\begin{equation}\label{eqn:thm:equiv_control:dpp_convergence_bound_total}
A_0 \leq \int |f_n(y, \alpha^\delta(x, y)) - f(y, \alpha^\delta(x, y))| \, \gamma(dx, dy) + N \varepsilon
\end{equation}
by a backward induction argument, where we set $\gamma := \gamma^1 \otimes \ldots \otimes \gamma_{x_{1:N-1}, y_{1:N-1}} \in \Pi_\delta(\mu, \cdot)$. As $(x, y) \mapsto |f_n(y, \alpha^\delta(x, y)) - f(y, \alpha^\delta(x, y))|$ is universally measurable by \cite[Proposition 7.44]{bertsekas1996stochastic} applied to $f(x, y) = (y, \alpha^\delta(x, y))$ and $g(y, \alpha) = |f_n(y, \alpha) - f(y, \alpha)|$, one can use \cite[Lemma 7.27]{bertsekas1996stochastic} with $p = \gamma$ and $f = \alpha^\delta$ to obtain a Borel measurable version of $(x, y) \mapsto \alpha^\delta(x, y)$. Hence, taking a  supremum over $\mathcal{A} \times \Pi_\delta(\mu, \cdot)$ in \eqref{eqn:thm:equiv_control:dpp_convergence_bound_total}, we obtain
\begin{align*}
&\sup_{\alpha_1 \in K} \sup_{\gamma^1 \in \Pi_\delta(\mu^1, \cdot)} \int \ldots \sup_{\alpha_N \in K} \sup_{\gamma^N \in \Pi_\delta(\mu_{x_{1:N-1}}, \cdot)} \int |f_n(y, \alpha(x, y)) - f(y, \alpha(x, y))| \, d\gamma^N \ldots d\gamma^1\\
&\leq \sup_{\alpha \in \mathcal{A}} \sup_{\gamma \in \Pi_\delta(\mu, \cdot)} \int |f_n(y, \alpha(x, y)) - f(y, \alpha(x, y))| \, \gamma(dx, dy) + N \varepsilon.
\end{align*}
As $\varepsilon > 0$ is arbitrary, this concludes the proof of \eqref{eqn:thm:equiv_control:common_bound}.

Lastly we prove that
\begin{equation}\label{eqn:thm:equiv_control:final_bound}
\lim_{n\to \infty}\sup_{\alpha \in \mathcal{A}} \sup_{\gamma \in \Pi_\delta(\mu, \cdot)} \int |f_n(y, \alpha(x, y)) - f(y, \alpha(x, y))| \, \gamma(dx, dy) = 0,
\end{equation}
which implies \eqref{eqn:thm:equiv_control:dro_convergence} and \eqref{eqn:thm:equiv_control:dpp_convergence}. For this, take any $\alpha \in \mathcal{A}$ and $\gamma \in \Pi_\delta(\mu, \cdot)$. As $\gamma \in \Pi_\delta(\mu, \cdot)$,
$$
\int \|x - y\|^p \, \gamma(dx, dy) \leq N \delta^p.
$$
Recalling that $|f(y, \alpha)| \lesssim 1 + \|y\|^{p - \varepsilon}$,
\begin{align*}
&\int |f_n(y, \alpha(x, y)) - f(y, \alpha(x, y))| \, \gamma(dx, dy)\\
&= \int_{\{f(y, \alpha(x, y)) \leq -n\}} |f(y, \alpha(x, y)) + n| \, \gamma(dx, dy)\\
&\leq \int_{\{n \lesssim 1 + \|y\|^{p - \varepsilon}\}} |f(y, \alpha(x, y)) + n| \, \gamma(dx, dy)\\
&\leq \int_{\{n \lesssim 1 + \|y\|^{p - \varepsilon}\}} |f(y, \alpha(x, y))| \, \gamma(dx, dy) + n \cdot \gamma\{n \lesssim 1 + \|y\|^{p - \varepsilon}\}\\
&\leq \|f\|_{L^\frac{p}{p - \varepsilon}(\gamma)} \cdot \gamma \{n \lesssim 1 + \|y\|^{p - \varepsilon}\}^\frac{\varepsilon}{p} + n \cdot \gamma\{n \lesssim 1 + \|y\|^{p - \varepsilon}\}\\
&\lesssim n^{-\frac{\varepsilon}{p - \varepsilon}} \to 0, \;\; n \to +\infty,
\end{align*}
where the first inequality follows from the growth condition on $f$, the second inequality follows from the triangle inequality, the third inequality follows from H\"older's inequality and the final inequality is a consequence of $\{n \lesssim 1 + \|y\|^{p - \varepsilon}\} = \{n^\frac{p}{p - \varepsilon} \lesssim 1 + \|y\|^p\}$, $p$-integrability of $\mu$ and $(|a| + |b|)^p \leq 2^{p - 1}(|a|^p + |b|^p)$. This concludes the proof.
\end{proof}

We now prove that the bicausal and causal optimization problems have the same value. 

\begin{proof}[Proof of Corollary \ref{cor:causal_bicausal}]
In light of Theorem \ref{equiv:control_open} it suffices to prove that
\begin{equation*}
V(\delta) = V^\delta_0 \geq \inf_{\alpha \in \mathcal{A}} \sup_{\gamma \in \Pi_{\delta}(\mu, \cdot)} \int f(y, \alpha(x, y)) \,\gamma(dx, dy),
\end{equation*}
since the opposite inequality follows from $\Pi_{\operatorname{bc}, \delta}(\mu, \cdot) \subseteq \Pi_{\delta}(\mu, \cdot)$. Similarly to the proof of Theorem \ref{equiv:control_open}, we choose Borel measurable selectors $(x_{1:t}, y_{1:t}, \alpha_{1:t}) \mapsto \alpha^\delta_{t+1}(x_{1:t}, y_{1:t}, \alpha_{1:t})$ for $V^\delta_t$ using \cite[Proposition 7.33]{bertsekas1996stochastic}, which gives rise to a Borel measurable control $\alpha^\delta = (\alpha^\delta_1, \ldots, \alpha^\delta_N) \in \mathcal{A}$. In particular 
\begin{align}\label{eqn:cor:bicausal_causal:optimal_control}
V^\delta_0 &= \sup_{\gamma^1 \in \Pi_\delta(\mu^1, \cdot)} \int \sup_{\gamma^2 \in \Pi_\delta(\mu_{x_1}, \cdot)} \ldots \sup_{\gamma^N \in \Pi_\delta(\mu_{x_{1:N-1}}, \cdot)} \int f(y, \alpha^\delta(x, y)) \,\gamma(dx, dy).
\end{align}
By definition, any plan $\gamma \in \Pi_{\delta}(\mu, \cdot)$ disintegrates into
\[
\gamma = \gamma^1 \otimes \gamma_{x_1, y_1} \otimes \ldots \otimes \gamma_{x_{1:N-1}, y_{1:N-1}}
\]
with $\gamma_{x_{1:t-1}, y_{1:t-1}} \in \Pi_\delta(\mu_{x_{1:t-1}}, \cdot)$ for $t =1,\dots ,N$. We obtain 
\[
V^\delta_0 \geq \sup_{\gamma \in \Pi_{\delta}(\mu, \cdot)} \int f(y, \alpha^\delta(x, y)) \, \gamma(dx, dy)
\]
from \eqref{eqn:cor:bicausal_causal:optimal_control}.
Taking an infimum over $\alpha \in \mathcal{A}$, we arrive at the desired conclusion.
\end{proof}

Before proving Corollary \ref{cor:linear}, we state the following lemma:

\begin{lemma}\label{density_of_maps_martingale}
Take any martingale probability measure \(\mu \in \mathcal{P}_p(\mathbb{R})\) and recall the set of martingale measures
\begin{align*}
\Pi^\mathcal{M}_{\delta}(\mu_{x_{1:t}}, \cdot) = \Big\{\pi \in \Pi(\mu_{x_{1:t}}, \cdot)\;:\; \mathcal{C}_p(\pi) < \delta, \int (x-y)\,\pi(dx,dy) =0  \Big\}.    
\end{align*}
Then the set $$\Pi_\delta^{\mathcal{M}, \operatorname{T}}(\mu_{x_{1:t}}, \cdot) := \Pi_\delta^\mathcal{M}(\mu_{x_{1:t}}, \cdot) \cap \Pi_\delta^{\operatorname{T}}(\mu_{x_{1:t}}, \cdot)$$ is dense in \(\Pi_\delta^{\mathcal{M}}(\mu_{x_{1:t}}, \cdot)\) with respect to \(\mathcal{W}_p\).
\end{lemma}

\begin{proof}
The proof is exactly the same as the proof of Lemma \ref{density_of_maps}: indeed, as in this proof we define the smoothed coupling $\pi_\sigma\in \Pi_\delta(\mu_{x_{1:t}}, \cdot)$ via \eqref{eq:smooth} for any $\pi\in \Pi_\delta^\mathcal{M}(\mu_{x_{1:t}},\cdot)$. By definition
\begin{align}
\int (y-x)\,\pi_\sigma(dx,dy) = \int (z-x) \varphi_\sigma(z-y)\,dz\, \pi(dx,dy) =\int (y-x)\,\pi(dx, dy)=0.
\end{align}
In conclusion, smoothing does not affect the linear constraint $\int (y-x)\,\pi(dx,dy)=0$. Applying Lemma \ref{lem:gangbo} preserves the constraint $\int (y-x)\,\pi(dx,dy)=0$ as well, because it only involves the first moments of the marginals of $\pi$, which are fixed in Lemma \ref{lem:gangbo}. 
\end{proof}

\begin{proof}[Proof of Corollary \ref{cor:linear}]
By Lemma \ref{density_of_maps_martingale} we conclude
\[
\sup_{\gamma^{t+1} \in \Pi^{\mathcal{M}}_\delta(\mu_{x_{1:t}}, \cdot)} \int V^{\mathcal{M}}_{t+1} \,d\gamma^{t+1} = \sup_{\gamma^{t+1} \in \Pi^{\mathcal{M},\operatorname{T}}_\delta(\mu_{x_{1:t}}, \cdot)} \int V^\mathcal{M}_{t+1}\, d\gamma^{t+1}.
\]
The remainder of the proof follows the proof of Theorem \ref{equiv:control_open} line by line.
\end{proof}

We now give the proof of Theorem \ref{minimax}, which is based on an application of the minimax theorem and the semi-separability assumption on $f.$

\begin{proof}[Proof of Theorem \ref{minimax}]
By Lemma \ref{lem:cost_to_go_regularity} we have enough regularity of $V_{t+1}^\delta$ to apply Lemmas \ref{lem:supremum_closed}, \ref{lem:supremum_closed} and conclude
\begin{align*}
V_t^\delta(x_{1:t}, y_{1:t}, \alpha_{1:t}) &= \inf_{\alpha_{t+1} \in K} \sup_{\gamma^{t+1} \in \overline\Pi_\delta(\mu_{x_{1:t}}, \cdot)} \int V^\delta_{t+1}(x_{1:t+1}, y_{1:t+1}, \alpha_{1:t+1}) \,\gamma^{t+1}(dx_{t+1}\, dy_{t+1})\\
&= \inf_{\alpha_{t+1} \in K} \sup_{\gamma^{t+1} \in \Pi^{\operatorname{T}}_\delta(\mu_{x_{1:t}}, \cdot)} \int V^\delta_{t+1}(x_{1:t+1}, y_{1:t+1}, \alpha_{1:t+1}) \,\gamma^{t+1}(dx_{t+1}\, dy_{t+1}).
\end{align*}
We now want to apply the minimax theorem \cite[Corollary 2]{terkelsen1972some} to interchange the order of the supremum and infimum. For this we note:
\begin{enumerate}
    \item The mapping $\overline\Pi_\delta(\mu_{x_{1:t}}, \cdot) \ni \gamma^{t+1} \mapsto \int V^\delta_{t+1}(x_{1:t+1}, y_{1:t+1}, \alpha_{1:t+1})\, \gamma^{t+1}(dx_{t+1}, dy_{t+1})$ is concave (even linear) for every $\alpha_{t+1} \in K$, and $\overline\Pi_\delta(\mu_{x_{1:t}}, \cdot)$ is convex.
    \item The mapping $K \ni \alpha_{t+1} \mapsto \int V^\delta_{t+1}(x_{1:t+1}, y_{1:t+1}, \alpha_{1:t+1})\, \gamma^{t+1}(dx_{t+1}, dy_{t+1})$ is convex by convexity of $\alpha \mapsto f(x, \alpha)$ and a backward induction argument. It is lower semicontinuous for every $\gamma^{t+1} \in \overline\Pi_\delta(\mu_{x_{1:t}}, \cdot)$ by Fatou's Lemma, lower semicontinuity of $V_{t+1}^\delta$ and the growth condition or boundedness from below. Moreover, $K$ is compact.
\end{enumerate}
Hence, applying the minimax theorem we obtain
\begin{align*}
V_t^\delta(x_{1:t}, y_{1:t}, \alpha_{1:t}) &= 
\inf_{\alpha_{t+1} \in K} \sup_{\gamma^{t+1} \in \overline \Pi_\delta(\mu_{x_{1:t}}, \cdot)} \int V^\delta_{t+1}(x_{1:t+1}, y_{1:t+1}, \alpha_{1:t+1}) \,\gamma^{t+1}(dx_{t+1}, dy_{t+1})\\
&= \sup_{\gamma^{t+1} \in \overline \Pi_\delta(\mu_{x_{1:t}}, \cdot)} \inf_{\alpha_{t+1} \in K} \int V^\delta_{t+1}(x_{1:t+1}, y_{1:t+1}, \alpha_{1:t+1}) \,\gamma^{t+1}(dx_{t+1}, dy_{t+1}).
\end{align*}
Furthermore, the above is equal to 
\begin{align*}
 \sup_{\gamma^{t+1} \in \Pi^{\operatorname{T}}_\delta(\mu_{x_{1:t}}, \cdot)} \inf_{\alpha_{t+1} \in K} \int V^\delta_{t+1}(x_{1:t+1}, y_{1:t+1}, \alpha_{1:t+1}) \,\gamma^{t+1}(dx_{t+1}, dy_{t+1}),
\end{align*}
since $(x_{1:t}, y_{1:t}, \alpha_{1:t+1}) \mapsto \int V^{t+1}_\delta \, d\gamma^{t+1}$ is lower semicontinuous by Fatou's Lemma, and an envelope over compact set $K$ of lower semicontinuous functions is again lower semicontinuous by \cite[Proposition 7.32.(a)]{bertsekas1996stochastic}. Hence, by weak density of $\Pi^{\operatorname{T}}_\delta(\mu_{x_{1:t}}, \cdot)$ in $\overline \Pi_\delta(\mu_{x_{1:t}}, \cdot)$ established in Lemma \ref{density_of_maps} and Proposition \ref{prop:open_dense_in_closed}, the conclusion follows. Next, we rewrite the dynamic programming principle for \(f(y, \alpha) = \sum_{t = 1}^N f_t(y_{1:t}, \alpha_t)\) as 
\begin{align}\label{minimax_proof:decomposition}
\begin{split}
V^\delta_0 &= \sup_{\gamma^1 \in \Pi^{\operatorname{T}}_\delta(\mu^1, \cdot)} \inf_{\alpha_1 \in K} \int \sup_{\gamma^2 \in \Pi^{\operatorname{T}}_\delta(\mu_{x_1}, \cdot)} \inf_{\alpha_2 \in K} \int \ldots \sup_{\gamma^N \in \Pi^{\operatorname{T}}_\delta(\mu_{x_{1:N-1}}, \cdot)} \inf_{\alpha_N \in K} \int f(y, \alpha) \,\gamma(dx, dy)\\
&= \sup_{\gamma^1 \in \Pi^{\operatorname{T}}_\delta(\mu^1, \cdot)} \inf_{\alpha_1 \in K} \int \dots \,
\left[\sum_{s=1}^{t} f_s(y_{1:s}, \alpha_s) + 
\sup_{\gamma^{t+1} \in \Pi^{\operatorname{T}}_\delta(\mu_{x_{1:t}}, \cdot)} \inf_{\alpha_{t+1} \in K} \int \dots\right.\\
&\quad \left.\dots \sup_{\gamma^N \in \Pi^{\operatorname{T}}_\delta(\mu_{x_{1:N-1}}, \cdot)} \inf_{\alpha_N \in K} \int \sum_{s=t}^Nf(y_{1:s}, \alpha_s) \,\gamma^N(dx_N, dy_N) \ldots \gamma^{t+1}(dx_{t+1}, dy_{t+1})\right] \,\gamma(dx_{1:t}, dy_{1:t}),
\end{split}
\end{align}
where $\gamma := \gamma^1 \otimes \ldots \otimes \gamma^N$. Consider the mapping
\begin{align}\label{eqn:minimax}
(x_{1:t}, y_{1:t}, \alpha_{1:t}, \gamma^{t+1}) \mapsto &\inf_{\alpha_{t+1} \in K} \int \ldots \sup_{\gamma^N \in \Pi_\delta^{\operatorname{T}}(\mu_{x_{1:N-1}}, \cdot)} \inf_{\alpha_N \in K} \int \sum_{s = t+1}^N f_s(y_{1:s}, \alpha_s) \, d\gamma^N \ldots d\gamma^{t+1}\nonumber\\
&= \inf_{\alpha_{t+1} \in K} \int V_{t+1}^\delta(x_{1:t+1}, y_{1:t+1}, \alpha_{1:t+1}) \, \gamma^{t+1}(dx_{t+1}, dy_{t+1}) - \sum_{s = 1}^{t} f_s(y_{1:s}, \alpha_s).
\end{align}
Since $V^\delta_{t+1}$ and $f$ are continuous and satisfy the growth condition, the mapping $$(x_{1:t}, y_{1:t}, \alpha_{1:t+1}, \gamma^{t+1}) \mapsto \int V^\delta_{t+1} \, d\gamma^{t+1} - \sum_{s = 1}^{t} f_s$$ is continuous  by Lemma \ref{sliced_continuity}. Consequently, \eqref{eqn:minimax} is upper semicontinuous as a negative of a supremum of lower semicontinuous mappings. By \cite[Proposition 7.50.(b)]{bertsekas1996stochastic} there exist universally measurable $\varepsilon$-optimizers $(x_{1:t}, y_{1:t}, \alpha_{1:t}) \mapsto \gamma_{x_{1:t}, y_{1:t}, \alpha_{1:t}} \in \Pi_{\delta}^{\operatorname{T}}(\mu_{x_{1:t}}, \cdot)$, i.e.,
\begin{equation}\label{eqn:minimax:near_optimizers_step}
\gamma_{x_{1:t}, y_{1:t}, \alpha_{1:t}} \in \varepsilon - \operatorname{argmax} \left(\inf_{\alpha_{t+1} \in K} \int V_{t+1}^\delta(x_{1:t+1}, y_{1:t+1}, \alpha_{1:t+1}) \, \gamma(dx_{t+1}, dy_{t+1}) - \sum_{s = 1}^{t} f_s(y_{1:s}, \alpha_s)\right).
\end{equation}
From \eqref{eqn:minimax} we note that $\gamma_{x_{1:t}, y_{1:t}, \alpha_{1:t}}$ does not depend on $\alpha_{1:t}$ due to semi-separability of $f$. Hence we denote it by $\gamma_{x_{1:t}, y_{1:t}}$, omitting the third argument. We then argue by \cite[Lemma 7.28.(c)]{bertsekas1996stochastic} applied with $p = \gamma^1 \otimes \ldots \otimes \gamma_{x_{1:t-1}, y_{1:t-1}}$ and $q = \gamma_{x_{1:t}, y_{1:t}}$ for $t =1, \dots, N-1,$ that $(x_{1:t}, y_{1:t}) \mapsto \gamma_{x_{1:t}, y_{1:t}}$ can be chosen to be Borel measurable, while \eqref{eqn:minimax:near_optimizers_step} still holds $(\gamma^1 \otimes \ldots \otimes \gamma_{x_{1:t-1}, y_{1:t-1}})$--almost surely. Setting $\gamma := \gamma^1 \otimes \ldots \otimes \gamma_{x_{1:N-1}, y_{1:N-1}} \in \Pi_{\operatorname{bc}, \delta}(\mu, \cdot)$ and applying \eqref{eqn:minimax:near_optimizers_step} we obtain
\begin{align*}
V^\delta_0 &\le  \inf_{\alpha_1 \in K} \int \inf_{\alpha_2 \in K} \int \dots \inf_{\alpha_N \in K} \int f(y, \alpha) \,\gamma(dx, dy) +N \varepsilon \\
&\leq \sup_{\gamma \in \Pi_{\operatorname{bc}, \delta}(\mu, \cdot)} \inf_{\alpha_1 \in K} \int \inf_{\alpha_2 \in K} \int \dots \inf_{\alpha_N \in K} \int f(y, \alpha) \,\gamma(dx, dy) + N \varepsilon\\
&\leq \sup_{\gamma \in \Pi_{\operatorname{bc}, \delta}(\mu, \cdot)} \inf_{\alpha \in \mathcal{A}} \int f(y, \alpha)\, \gamma(dx, dy) + N \varepsilon,
\end{align*}
where the final inequality holds since every predictable process $\alpha \in \mathcal{A}$ yields a selector for the objective $\inf_{\alpha_1 \in K} \int \inf_{\alpha_2 \in K} \int \dots \inf_{\alpha_N \in K} \int f \,d\gamma$. Now taking $\varepsilon \to 0$ we obtain
$$
V^\delta_0 \leq \sup_{\gamma \in \Pi_{\operatorname{bc}, \delta}(\mu, \cdot)} \inf_{\alpha \in \mathcal{A}} \int f(y, \alpha)\, \gamma(dx, dy),
$$
which confirms $\inf_{\alpha \in \mathcal{A}} \sup_{\gamma \in \Pi_{\operatorname{bc}, \delta}(\mu, \cdot)} \int f(y, \alpha) \, \gamma(dx, dy) \leq \sup_{\gamma \in \Pi_{\operatorname{bc}, \delta}(\mu, \cdot)} \inf_{\alpha \in \mathcal{A}} \int f(y, \alpha)\, \gamma(dx, dy)$, as $V^\delta_0 = \inf_{\alpha \in \mathcal{A}} \sup_{\gamma \in \Pi_{\operatorname{bc}, \delta}(\mu, \cdot)} \int f(y, \alpha) \, \gamma(dx, dy)$ by Theorem \ref{equiv:control_open}, as $f$ is lower semicontinuous and satisfies desired growth condition. The opposite inequality is trivial, and hence the proof is complete.
\end{proof}

\subsection{Proof of Theorem \ref{dro:control_sensitivity} and Corollary \ref{dro:control_martingale_sensitivity}}

 We now proceed with the proof of Theorem \ref{dro:control_sensitivity}. We start with the easier control free case.

\begin{proof}[Proof of Theorem \ref{dro:control_sensitivity}, control free, upper bound.]
We first note that 
\begin{align}
f(x)-f(0)= \int_0^1  \langle \nabla_x f(\lambda x),x\rangle \,d\lambda \lesssim \int_0^1 (1+\|x\|)^{p-1-\epsilon} \|x\| \,d\lambda \lesssim 1 + \|x\|^{p-\epsilon}
\end{align}
for any $x \in \R^N,$ so that $f$ satisfies the growth condition of Theorem \ref{equiv:basic}.
By Theorem \ref{equiv:basic} and Lemma \ref{lem:supremum_closed} we have
\[
V(\delta) - V(0) = \sup_{\gamma^1 \in \overline\Pi_\delta(\mu^1, \cdot)} \int \sup_{\gamma^2 \in \overline\Pi_\delta(\mu_{x_{1}}, \cdot)} \int \ldots \sup_{\gamma^N \in \overline\Pi_\delta(\mu_{x_{1:N-1}}, \cdot)} \int [f(y) - f(x)] \,\gamma(dx, dy).
\]
By Lemma \ref{lem:cost_to_go_regularity} the function $V_t^\delta$ is continuous and satisfies $|V_t^\delta(x_{1:t}, y_{1:t})| \lesssim 1 + \|x_{1:t}\|^{p - \varepsilon} + \|y_{1:t}\|^{p - \varepsilon}$, hence $(x_{1:t}, y_{1:t}, \gamma^{t+1}) \mapsto \int V_{t+1}^\delta \, d\gamma^{t+1}$ is continuous by Lemma \ref{sliced_continuity}. Moreover, $\overline \Pi_\delta(\mu_{x_{1:t}}, \cdot)$ is compact in $(\mathcal{P}_{p - \varepsilon}(\mathbb{R}^2), \mathcal{W}_{p - \varepsilon})$ by Proposition \ref{hemicontinuity}. Hence, \cite[Proposition 7.33]{bertsekas1996stochastic} applied with $f = -\int V_{t+1}^\delta$ yields Borel measurable optimizers \((x_{1:t}, y_{1:t}) \mapsto \gamma_{x_{1:t}, y_{1:t}} \in \overline\Pi_\delta(\mu_{x_{1:t}}, \cdot)\) for  \(V_t^\delta(x_{1:t}, y_{1:t})\).

We now set \(\gamma: = \gamma^1\otimes \gamma_{x_1,y_1}\otimes\dots \gamma_{x_{1:N-1}, y_{N-1}}\). Clearly $\gamma$ depends on $\delta$; we omit this dependence for now to shorten notation, and will refer to $\gamma=\gamma^\delta$ in the second part of the proof. Using a telescoping argument,
\begin{align}\label{eq:telescope}
\begin{split}
V(\delta) - V(0) &= \int \sum_{t = 1}^N [f(y_{1:t}, x_{t+1:N}) - f(y_{1:t-1}, x_{t:N})]\, \gamma(dx, dy)\\
&= \sum_{t = 1}^N \int [f(y_{1:t}, x_{t+1:N}) - f(y_{1:t-1}, x_{t:N})]\, \gamma(dx, dy)\\
&= \sum_{t = 1}^N S_t(\gamma), \;\; \text{where} \;\; S_t(\gamma) := \int [f(y_{1:t}, x_{t+1:N}) - f(y_{1:t-1}, x_{t:N})] \,\gamma (dx, dy).
\end{split}
\end{align}
We now deal with each term \(S_t(\gamma)\) separately. Using differentiability of \(f\) and Fubini's theorem, we obtain
\begin{align}\label{eq:S}
\begin{split}
S_t(\gamma) &= \int [f(y_{1:t}, x_{t+1:N}) - f(y_{1:t-1}, x_{t:N})]\, \gamma(dx, dy)\\
&= \int \left[\int_0^1 \partial_t f(y_{1:t-1}, x_t + \lambda (y_t - x_t), x_{t+1:N}) (y_t - x_t) \,d\lambda\right]\, \gamma(dx, dy)\\
&= \int_0^1 \left[\int \partial_t f(y_{1:t-1}, x_t + \lambda (y_t - x_t), x_{t+1:N}) (y_t - x_t) \bar{\mu}_{x_{1:t}}(dx_{t+1:N})\,\gamma(dx_{1:t}, dy_{1:t})\right] \,d\lambda.
\end{split}
\end{align}
We now disintegrate the measure \(\gamma(dx_{1:t}, dy_{1:t}) = \gamma_{x_{1:t-1}, y_{1:t-1}}(dx_t, dy_t) \otimes \gamma(dx_{1:t-1}, dy_{1:t-1})\) and apply H\"older's inequality for the probability measure  \(\gamma_{x_{1:t-1}, y_{1:t-1}}(dx_t, dy_t)\). As \(\mathcal{C}_p(\gamma_{x_{1:t-1}, y_{1:t-1}}) < \delta\) we obtain for $q=p/(p-1)$
\begin{align*}
&\int \partial_t f(y_{1:t-1}, x_t + \lambda (y_t - x_t), x_{t+1:N}) (y_t - x_t) \,\bar{\mu}_{x_{1:t}}(dx_{t+1:N})\,\gamma(dx_{1:t}, dy_{1:t})\\
&\le \delta \int \left\|\int \partial_t f(y_{1:t-1}, x_t + \lambda (y_t - x_t), x_{t+1:N}) \,\bar{\mu}_{x_{1:t}}(dx_{t+1:N})\right\|_{L^q(\gamma_{x_{1:t-1}, y_{1:t-1}})} \,\gamma(dx_{1:t-1}, dy_{1:t-1})\\
&\le \delta \int F_\delta(\lambda, x_{1:t-1}, y_{1:t-1})\, \gamma(dx_{1:t-1}, dy_{1:t-1}),
\end{align*}
where 
\[
F_\delta(\lambda,x_{1:t-1}, y_{1:t-1}) := \sup_{\gamma \in \overline\Pi_\delta(\mu_{x_{1:t-1}}, \cdot)}\left\|\int \partial_t f(y_{1:t-1}, x_t + \lambda (y_t - x_t), x_{t+1:N}) \,\bar{\mu}_{x_{1:t}}(dx_{t+1:N})\right\|_{L^q(\gamma)}.
\]
Our next aim is to apply Lemma \ref{uniform_convergence_dini} to $F_\delta$ with $K_\delta(x) =\overline\Pi_\delta(\mu_{x_{1:t-1}}, \cdot)$. For this we collect a number of observations:
\begin{enumerate}
\item Recalling the growth condition $\|\nabla_x f(x, \alpha)\| \leq C(1 + \|x\|^{p - 1 - \varepsilon})$ we have
    \begin{align*}
    &\left|\int \partial_t f(y_{1:t-1}, x_t + \lambda (y_t - x_t), x_{t+1:N}) \,\bar{\mu}_{x_{1:t}}(dx_{t+1:N})\right|\\
    &\lesssim 1 + \|y_{1:t-1}\|^{p - 1 - \varepsilon} + \|x_t + \lambda (y_t - x_t)\|^{p - 1 - \varepsilon} + \int \|x_{t+1:N}\|^{p - 1 - \varepsilon}\, \bar{\mu}_{x_{1:t}}(dx_{t+1:N})\\
    &\lesssim 1 + \|y_{1:t}\|^{p - 1 - \varepsilon} + \|x_{1:t}\|^{p - 1 - \varepsilon}
    \end{align*}
    for all $x_{1:t}, y_{1:t}\in \R^{t-1}$ and $\lambda\in [0,1]$,
    where the last inequality follows from 
    \begin{align}\label{eq:useful}
    \int \|x_{t+1:N}\|^{p - 1 - \varepsilon} \bar{\mu}_{x_{1:t}}(dx_{t+1:N}) \leq \left(\int \|x_{t+1:N}\|^p \bar{\mu}_{x_{1:t}}(dx_{t+1:N})\right)^{\frac{p - 1 - \varepsilon}{p}} \lesssim 1+\|x_{1:t}\|^{p - 1 - \varepsilon}
    \end{align}
    and the assumption \(\int \|x_{t+1:N}\|^p \,\bar{\mu}_{x_{1:t}}(dx_{t+1:N}) \lesssim \|x_{1:t}\|^p\).
    Recalling that $q=p/(p-1)$ and using the inequality $(|a| + |b|)^p \leq 2^{p-1} (|a|^p + |b|^p)$  we conclude that 
    \[
    \left|\int \partial_t f(y_{1:t-1}, x_t + \lambda (y_t - x_t), x_{t+1:N}) \,\bar{\mu}_{x_{1:t}}(dx_{t+1:N})\right|^q \lesssim 1 + \|y_{1:t}\|^\frac{p(p - 1 - \varepsilon)}{p-1} + \|x_{1:t}\|^\frac{p(p - 1 - \varepsilon)}{p-1}.
    \]
    We also note that $$ \frac{p(p-1-\epsilon)}{p-1} \le p-1-\epsilon +1 =p-\epsilon<p-\epsilon/2.$$
    \item The correspondence \((x_{1:t-1}, y_{1:t-1}) \twoheadrightarrow \overline{\Pi}_\delta(\mu_{x_{1:t-1}}, \cdot)\) is continuous in $(\mathcal{P}_{p - \varepsilon/2}(\mathbb{R}^2), \mathcal{W}_{p - \varepsilon/2})$ by Proposition \ref{hemicontinuity}, and $\overline \Pi_\delta(\mu_{x_{1:t-1}}, \cdot)$ is compact. Moreover, \(\overline{\Pi}_\delta(\mu_{x_{1:t-1}}, \cdot) \downarrow \{(\operatorname{Id}, \operatorname{Id})_\# \mu_{x_{1:t-1}}\}\), as \(\delta \to 0\).
\end{enumerate}
This allows to conclude that $$\lim_{\delta\to 0}F_\delta(\lambda, x_{1:t-1}, y_{1:t-1}) = F_0:=\Big\|\int \partial_t f(y_{1:t-1}, x_{t:N}) \,\bar{\mu}_{x_{1:t}}(dx_{t+1:N})\Big\|_{L^q(\mu_{x_{1:t-1}})}$$
uniformly on any compact set \(K \subset [0,1]
\times \mathbb{R}^{t - 1} \times \mathbb{R}^{t - 1}\) by Lemma \ref{uniform_convergence_dini} (with $K_\delta(x) =\overline\Pi_\delta(\mu_{x_{1:t-1}}, \cdot)$ and $\mathcal{W}_{p-\epsilon/2}$). Recalling that $\gamma=\gamma^\delta$ we furthermore have \(\gamma^{\delta} \to (\operatorname{Id}, \operatorname{Id})_\# \mu\) in \(\mathcal{W}_p\). Thus we can use Prokhorov's Theorem to find compact sets $K_\epsilon$ for each $\epsilon>0$, such that
\begin{align}\label{eq:prokhorov}
\lim_{\varepsilon \to 0} \sup_{\delta \in [0, 1]} \gamma^{\delta}((K_\epsilon)^\text{c}) = 0.
\end{align}
Defining \(\Delta F_\delta(\lambda,\delta, x_{1:t-1}, y_{1:t-1}) := F_\delta(\lambda,\delta, x_{1:t-1}, y_{1:t-1}) - F_0(\lambda, x_{1:t-1}, y_{1:t-1})\) we conclude
\begin{align*}
&\limsup_{\delta \to 0} \int \Delta F_\delta(\lambda,\delta, x_{1:t-1}, y_{1:t-1}) \,\gamma^{\delta}(dx_{1:t-1}, dy_{1:t-1})\\
&\qquad\leq \limsup_{\delta \to 0} \int_{K_\varepsilon} \Delta F_\delta(\lambda,\delta, x_{1:t-1}, y_{1:t-1}) \,\gamma^{\delta}(dx_{1:t-1}, dy_{1:t-1})\\
&\qquad\qquad+ \limsup_{\delta \to 0} \int_{(K_\epsilon)^\text{c}} \Delta F_\delta(\lambda,\delta, x_{1:t-1}, y_{1:t-1}) \,\gamma^{ \delta}(dx_{1:t-1}, dy_{1:t-1}) \to 0
\end{align*}
as \(\varepsilon \to 0\): the first term vanishes due to the uniform convergence of $F_\delta\to F_0$ for $\delta\to 0$ on $K_\epsilon$, and the second term vanishes by H\"older's inequality, \eqref{eq:useful} and \eqref{eq:prokhorov}. Recalling \eqref{eq:S} and \eqref{eq:useful} we now use the dominated convergence theorem for the $d\lambda$-integral to conclude 
\begin{align*}
\begin{split}
\limsup_{\delta \to 0} \frac{S_t(\gamma^\delta)}{\delta} &\leq \int_0^1 \limsup_{\delta \to 0} \int F_\delta(\lambda, x_{1:t-1}, y_{1:t-1}) \,\gamma^{\delta}(dx_{1:t-1}, dy_{1:t-1}) \,d\lambda\\
&\leq \int_0^1 \limsup_{\delta \to 0} \int \Delta F_\delta(\lambda, x_{1:t-1}, y_{1:t-1}) \,\gamma^{\delta}(dx_{1:t-1}, dy_{1:t-1}) \,d\lambda\\
&\quad + \limsup_{\delta \to 0} \int\left\|\int \partial_t f(y_{1:t-1}, x_{t:N}) \,\bar{\mu}_{x_{1:t}}(dx_{t+1:N})\right\|_{L^q(\mu_{x_{1:t-1}})} \gamma^{\delta}(dx_{1:t-1}, dy_{1:t-1})\\
&= \int \left\|\int \partial_t f(x) \,\bar{\mu}_{x_{1:t}}(dx_{t+1:N})\right\|_{L^q(\mu_{x_{1:t-1}})} \,\mu(dx_{1:t-1}),
\end{split}
\end{align*}
noting that \(\gamma^{\delta} \to (\operatorname{Id}, \operatorname{Id})_\# \mu\) in \(\mathcal{W}_p(\mathbb{R}^{N} \times \mathbb{R}^{N})\) for the last equality. Combining the estimates for \(S_t(\gamma^\delta)\) we obtain
\begin{align*}
&\limsup_{\delta \to 0} \frac{V(\delta) - V(0)}{\delta} \leq \sum_{t = 1}^N \limsup_{\delta \to 0} \frac{S_t(\gamma^\delta)}{\delta}& \\
&\leq \left\|\int \partial_1 f(x) \,\bar{\mu}_{x_{1}}(dx_{2:N})\right\|_{L^q(\mu^1)} + \sum_{t = 2}^N \int \left\|\int \partial_t f(x) \,\bar{\mu}_{x_{1:t}}(dx_{t+1:N})\right\|_{L^q(\mu_{x_{1:t-1}})} \,\mu(dx_{1:t-1}),
\end{align*}
as claimed.
\end{proof}

\begin{proof}[Proof of Theorem \ref{dro:control_sensitivity}, control free, lower bound]
By duality between \(L^p(\mu_{x_{1:t-1}})\) and \(L^q(\mu_{x_{1:t-1}})\) there exists Borel measurable functions $T_t:\R^t\to \R$ with $\|T_t(x_{1:t})\|_{L^p(\mu_{x_{1:t-1}})} = 1$, that satisfy
\begin{align}\label{eq:dual}
\begin{split}
\left\|\int \partial_t f(x)\, \bar{\mu}_{x_{1:t}}(dx_{t+1:N})\right\|_{L^q(\mu_{x_{1:t-1}})} &= 
\int \Big(\int \partial_t f(x) \, \bar{\mu}_{x_{1:t}}(dx_{t+1:N}) \Big) T_t(x_{1:t}) \,\mu_{x_{1:t-1}}(dx_{t})\\
&=\int \partial_t f(x) T_t(x_{1:t}) \,\bar{\mu}_{x_{1:t-1}}(dx_{t:N}).
\end{split}
\end{align}
We now define 
$$
\gamma^\delta_{x_{1:t-1}, y_{1:t-1}} := (x_t, x_t + \bar \delta T_t(x_{1:t}))_\# \mu_{x_{1:t-1}} \;\; \text{for} \;\; \bar{\delta} = \frac{\delta}{1 + \delta},
$$
and set $\gamma^\delta := \gamma^{1, \delta} \otimes \gamma^\delta_{x_1, y_1} \otimes \ldots \otimes \gamma^\delta_{x_{1:N-1}, y_{1:N-1}}$. This transport plan is causal by \cite[Proposition 2.4, 2]{backhoff2017causal}, and $\gamma^\delta \in \Pi_{\delta}(\mu, \cdot)$ holds as $\|T_t(x_{1:t})\|_{L^p(\mu_{x_{1:t-1}})} = 1$ and $\bar \delta < \delta$. Recalling the definition of $S_t$ from \eqref{eq:telescope} and using the fundamental theorem of calculus together with Fubini's Theorem we have
\begin{align*}
S_t(\gamma^\delta) &= \int [f(y_{1:t}, x_{t+1:N}) - f(y_{1:t-1}, x_{t:N})]\, \gamma^{\delta}(dx, dy)\\
&= \int \left[\int_0^1 \partial_t f(y_{1:t-1}, x_t + \lambda (y_t - x_t), x_{t+1:N}) (y_t - x_t) \,d\lambda \right]\, \gamma^{\delta}(dx, dy)\\
&= \bar \delta \int_0^1 \left[\int \partial_t f(x_{1:t-1} + \bar \delta T_{1:t-1}(x_{1:t-1}), x_t + \lambda \bar \delta T_t(x_{1:t}), x_{t+1:N}) T_t(x_{1:t})\, \mu(dx)\right]\, d\lambda.
\end{align*}
We note that $\partial_t f \cdot T_t$ is bounded in $L^1(\mu)$ uniformly in $\lambda$ by H\"older's inequality due to the growth condition on $\nabla f$ and $\|T_t\|_{L^p(\mu)} = 1$. Hence, applying the dominated convergence theorem we obtain 
\begin{align*}
&\liminf_{\delta \to 0} \frac{S_t(\gamma^\delta)}{\delta}\\
&\geq \liminf_{\delta \to 0} \frac{1}{1 + \delta} \int_0^1 \left[\int \partial_t f(x_{1:t-1} + \bar \delta T_{1:t-1}(x_{1:t-1}), x_t + \lambda \bar \delta T_t(x_{1:t}), x_{t+1:N}) T_t(x_{1:t})\, \mu(dx)\right]\, d\lambda\\
&= \int \partial_t f(x) T_t(x_{1:t}) \,\mu(dx)\stackrel{\eqref{eq:dual}} {=} \int \left\|\int \partial_t f(x) \,\bar{\mu}_{x_{1:t}}(dx_{t+1:N})\right\|_{L^q(\mu_{x_{1:t-1}})} \mu(dx_{1:t-1}).
\end{align*}

Finally, using Corollary \ref{cor:causal_bicausal} we have
\begin{align}
\frac{V(\delta) - V(0)}{\delta} \ge \sum_{t = 1}^N  \frac{S_t(\gamma^\delta)}{\delta}.
\end{align}
Taking the limit inferior on both sides and sending $\varepsilon \to 0$, we arrive to the desired conclusion.
\end{proof}

To extend the results to the controlled case, we start with the auxiliary regularity result for the cost-to-go functions $V_t^\delta$ for the case $\delta = 0$.

\begin{lemma}\label{lem:control_certain:strong_convexity}
Consider the dynamic programming principle for the case $\delta = 0$, i.e., without model uncertainty:
\begin{align}
\begin{split}\label{eq:V0}
V^0_N(x, \alpha) &= f(x, \alpha)\\
V^0_t(x_{1:t}, \alpha_{1:t}) &= \inf_{\alpha_{t+1} \in K} \int V^0_{t+1}(x_{1:t+1}, \alpha_{1:t+1}) \,\mu_{x_{1:t}}(dx_{t+1}).
\end{split}
\end{align}
In the setting of Lemma \ref{lem:cost_to_go_regularity} assume that $\alpha \mapsto f(x, \alpha)$ is $\varepsilon(x)$-strongly convex. Then the cost-to-go functions $\alpha_{1:t} \mapsto V_t^0(x_{1:t}, \alpha_{1:t})$ are $\varepsilon_t(x_{1:t})$-strongly convex, where $\epsilon_N(x):=\epsilon(x)$ and
$$
\varepsilon_t(x_{1:t}) := \int \varepsilon(x) \,\bar{\mu}_{x_{1:t}}(dx_{t+1:N}).
$$
\end{lemma}
\begin{proof}
We show the claim by backward induction. The statement clearly holds for $t = N$. Suppose that $t < N$, fix $x_{1:t} \in \R^t$ and assume that the statement holds for $V_{t+1}^0$. We first note that
\begin{align}\label{eq:epsilon}
\int \epsilon_{t+1}\,\mu_{x_{1:t}}(dx_{t+1}) =\int \int \epsilon(x)\, \bar{\mu}_{x_{1:t+1}}(dx_{t+2:N})\,\mu_{x_{1:t}}(dx_{t+1}) = \int \epsilon(x) \, \bar{\mu}_{x_{1:t}}(dx_{t+1:N}) =\epsilon_t(x_{1:t}).
\end{align}
Pick any $\lambda \in (0, 1)$ and $\alpha_{1:t}^{(1)}, \alpha_{1:t}^{(2)} \in K^t$, and let $\alpha_{t+1}^{(i)} \in K$ be an optimizer for $V^0_t(x_{1:t}, \alpha_{1:t}^{(i)})$, where $i = 1, 2$ (note that existence of optimizers is guaranteed in the setting of Lemma \ref{lem:cost_to_go_regularity}). Then we have
\begin{align*}
&\lambda V_t^0(x_{1:t}, \alpha_{1:t}^{(1)}) + (1 - \lambda) V^0_t(x_{1:t}, \alpha_{1:t}^{(2)})\\
&= \int \lambda V_t^0(x_{1:t+1}, \alpha_{1:t+1}^{(1)}) + (1 - \lambda) V^0_t(x_{1:t+1}, \alpha_{1:t+1}^{(2)})\,\mu_{x_{1:t}}(dx_{t+1})\\
&\stackrel{\text{(IH)}}{\ge} \int \Big[ V_t^0(x_{1:t+1}, \lambda \alpha_{1:t+1}^{(1)} + (1 - \lambda) \alpha_{1:t+1}^{(2)})) + \frac{1}{2}\epsilon_{t+1}(x_{1:t+1}) \lambda(1-\lambda) \sum_{s=1}^{t+1} |\alpha_s^{(1)}-\alpha_s^{(2)}|^2 \Big]\,\mu_{x_{1:t}}(dx_{t+1})\\
&\stackrel{\eqref{eq:epsilon}}{=} 
\int [ V_t^0(x_{1:t+1}, \lambda \alpha_{1:t+1}^{(1)} + (1 - \lambda) \alpha_{1:t+1}^{(2)}))]\,\mu_{x_{1:t}}(dx_{t+1}) + \frac{1}{2}\epsilon_{t}(x_{1:t}) \lambda(1-\lambda) \sum_{s=1}^{t} |\alpha_s^{(1)}-\alpha_s^{(2)}|^2 \\
&\ge V_t^0(x_{1:t}, \lambda \alpha_{1:t}^{(1)} + (1 - \lambda) \alpha_{1:t}^{(2)})+ \frac{1}{2}\epsilon_{t}(x_{1:t}) \lambda(1-\lambda) \sum_{s=1}^{t} |\alpha_{s}^{(1)}-\alpha_{s}^{(2)}|^2.
\end{align*}
This shows the claim.
\end{proof}

Lemma \ref{lem:control_certain:strong_convexity} guarantees continuity of the optimal control. In consequence, the proof of the upper bound in Theorem \ref{dro:control_sensitivity} is straightforward. 

\begin{proof}[Proof of Theorem \ref{dro:control_sensitivity} with control, upper bound]
As in the uncontrolled case, it is straightforward to check that $V_t^\delta$ satisfies the growth assumptions of Lemma \ref{lem:cost_to_go_regularity}.(2). 
Since $f$ is also continuous and $\varepsilon(x)$--strongly convex by assumption, Lemma \ref{lem:cost_to_go_regularity}.(2) shows that each cost-to-go function $V^0_t$ is continuous and Lemma \ref{lem:control_certain:strong_convexity} shows that $\alpha_{1:t}\mapsto V_t^0(x_{1:t}, \alpha_{1:t})$ is $\varepsilon_t(x_{1:t})$--strongly convex.
Hence, the optimal control $\alpha^\star(x) \in \operatorname{argmin}(V(0))$ is unique and continuous (by Berge's maximum theorem \cite[Theorem 17.31]{guide2006infinite}, noting that a single-valued correspondence is continuous if it is upper hemicontinuous). Thus
\[
V(\delta) - V(0) \leq \sup_{\gamma\in \Pi_{\operatorname{bc}, \delta}(\mu,\cdot)} \int [f(y, \alpha^\star(x)) - f(x, \alpha^\star(x))] \,\gamma(dx, dy).
\]
We now proceed in the same way as in [Proof of Theorem \ref{dro:control_martingale_sensitivity}, control free, upper bound].
\end{proof}

\begin{proof}[Proof of Theorem \ref{dro:control_sensitivity} with control, lower bound]
We first note that
\begin{align*}
V(\delta) &= \inf_{\alpha \in \mathcal{A}} \sup_{\gamma \in \Pi_{\operatorname{bc}, \delta}(\mu, \cdot)} \int f(y, \alpha(x, y)) \,\gamma(dx, dy)\\
&= \inf_{\alpha \in \mathcal{A}} \sup_{\gamma \in \Pi_\delta(\mu, \cdot)} \int f(y, \alpha(x, y)) \,\gamma(dx, dy)\\
&\geq \sup_{\gamma \in \Pi_\delta(\mu, \cdot)} \inf_{\alpha \in \mathcal{A}} \int f(y, \alpha(x, y))\, \gamma(dx, dy)\\
& =: \widetilde{V}(\delta),
\end{align*}
where the second equality follows from Corollary \ref{cor:causal_bicausal}, and the inequality is valid since $\inf \sup \geq \sup \inf$. As $\widetilde{V}(0) = V(0)$ we have
$$
V(\delta) - V(0) \geq \widetilde{V}(\delta) - \widetilde{V}(0).
$$
Throughout the rest of the proof we estimate $\widetilde{V}(\delta) - \widetilde{V}(0)$. The remainder of the proof is very similar to 
[Proof of Theorem \ref{dro:control_sensitivity}, control free, lower bound]. For completeness, we state it in full detail: 
denote by $\alpha^\star$ the optimal control for $V(0)$. Similarly to the control-free case, we use duality between $L^p(\mu_{x_{1:t-1}})$ and $L^q(\mu_{x_{1:t-1}})$ to find Borel measurable functions $T_t:\R^t\to \R$, which satisfy the identities $\|T_t(x_{1:t-1}, \cdot)\|_{L^p(\mu_{x_{1:t-1}})} = 1$ and
\begin{align*}
\left\|\int \partial_t f(x, \alpha^\star(x)) \,\bar{\mu}_{x_{1:t}}(dx_{t+1:N})\right\|_{L^q(\mu_{x_{1:t-1}})} &= 
\int \Big(\int \partial_t f(x, \alpha^\star(x)) \, \bar{\mu}_{x_{1:t}}(dx_{t+1:N}) \Big) T_t(x_{1:t}) \,\mu_{x_{1:t-1}}(dx_{t})\\
&=\int \partial_t f(x, \alpha^\star(x)) T_t(x_{1:t}) \,\bar{\mu}_{x_{1:t-1}}(dx_{t:N}).
\end{align*}
We now define 
$$\gamma^\delta_{x_{1:t-1}, y_{1:t-1}} := (x_t, x_t + \bar \delta T_t(x_{1:t}))_\# \mu_{x_{1:t-1}} \;\; \text{for} \;\; \bar \delta = \frac{\delta}{1 + \delta},$$
and $\gamma^\delta := \gamma^{1, \delta} \otimes \gamma^\delta_{x_1, y_1} \otimes \ldots \otimes \gamma^\delta_{x_{1:N-1}, y_{1:N-1}}$, which is causal by \cite[Proposition 2.4, 2]{backhoff2017causal}. We now take controls $(x,y)\mapsto \alpha^\delta(x, y)$, which minimize $\inf_{\alpha \in \mathcal{A}} \int f(y, \alpha(x, y)) \,\gamma^\delta(dx, dy)$ so that
\begin{align}\label{eqn:thm:lower_bound_control:selected}
\widetilde{V}(\delta) - \widetilde{V}(0) &\geq \int f(y, \alpha^\delta(x, y))-f(x, \alpha^\delta(x,y)) \,\gamma^\delta(dx, dy) \nonumber\\
&= \int [f(x + \bar \delta T(x), \alpha^\delta(x)) - f(x, \alpha^\delta(x))]\,\mu(dx),
\end{align}
where we set $T(x):=(T_1(x_1), \dots, T_N(x_{1:N}))$ and $\alpha^\delta(x) := \alpha^\delta(x, x 
+ \bar \delta T(x))$. For now assume that $\alpha^\delta \to \alpha^\star$ in $\mu$-measure, and hence the convergence holds $\mu$-almost everywhere along a subsequence. We now use the same telescoping and fundamental theorem of calculus argument for \eqref{eqn:thm:lower_bound_control:selected} as in [Proof of Theorem \ref{dro:control_sensitivity}, control free, lower bound], i.e.,
\begin{align*}
&\int [f(x + \bar \delta T(x), \alpha^\delta(x)) - f(x, \alpha^\delta(x))]\,\mu(dx) = \sum_{t = 0}^N S_t(\gamma^\delta, \alpha^\delta), \;\; \text{where}\\
&S_t(\gamma^\delta, \alpha^\delta) := \bar \delta \int_0^1 \left[\int \partial_t f(x_{1:t-1} + \bar \delta T_{1:t-1}(x_{1:t-1}), x_t + \lambda \bar \delta T_t(x_{1:t}), x_{t+1:N}, \alpha^\delta(x)) T_t(x_{1:t})\, \mu(dx)\right]\, d\lambda.
\end{align*}
Note that $\partial f \cdot T_t$ is uniformly bounded in $L^1(\mu)$ along a subsequence due to H\"older's inequality and the growth condition on $\nabla f$ together with boundedness of $\alpha^\delta$. We we can thus apply the dominated convergence theorem to obtain
\begin{align*}
&\liminf_{\delta \to 0} \frac{S_t(\gamma^\delta, \alpha^\delta)}{\delta}\\
&\geq \liminf_{\delta \to 0} \frac{1}{1 + \delta} \int_0^1 \left[\int \partial_t f(x_{1:t-1} + \bar \delta T_{1:t-1}(x_{1:t-1}), x_t + \lambda \bar \delta T_t(x_{1:t}), x_{t+1:N}, \alpha^\delta(x)) T_t(x_{1:t})\, \mu(dx)\right]\, d\lambda\\
&= \int \partial_t f(x, \alpha^\star(x)) T_t(x_{1:t}) \,\mu(dx) = \int \left\|\int \partial_t f(x, \alpha^\star(x)) \,\bar{\mu}_{x_{1:t}}(dx_{t+1:N})\right\|_{L^q(\mu_{x_{1:t-1}})} \mu(dx_{1:t-1}),
\end{align*}
where we used that $\alpha^\delta \to \alpha^\star$ along a subsequence $\mu$-almost everywhere. Finally, summing up the estimates for $t =1, \dots, N$ we conclude
\begin{align*}
\liminf_{\delta \to 0} \frac{V(\delta) - V(0)}{\delta} &\geq \liminf_{\delta \to 0} \sum_{t = 1}^N \frac{S_t(\gamma^\delta, \alpha^\delta)}{\delta} \geq \sum_{t = 1}^N \liminf_{\delta \to 0} \frac{S_t(\gamma^\delta, \alpha^\delta)}{\delta}\\
&\geq \left\|\int \partial_1 f(x, \alpha(x)) \,\bar{\mu}_{x_{1}}(dx_{2:N})\right\|_{L^q(\mu^1)}\\
&+ \sum_{t = 2}^N \int \left\|\int \partial_t f(x, \alpha(x)) \,\bar{\mu}_{x_{1:t}}(dx_{t+1:N})\right\|_{L^q(\mu_{x_{1:t-1}})} \mu(dx_{1:t-1}).
\end{align*}
It remains to show that $\alpha^\delta \to \alpha^\star$ in $\mu$-measure. To that end, note that for any transport plan $\gamma \in \Pi_\delta(\mu, \cdot)$ and control $\alpha \in \widetilde{\mathcal{A}}$ we have 
\begin{align*}
\int [f(y, \alpha(x)) - f(x, \alpha(x))]\,\gamma(dx, dy) &= \int \int_0^1 \langle y - x, \nabla_x f(x + \lambda (y - x), \alpha(x)) \rangle\,d\lambda\,\gamma(dx, dy)\\
&\lesssim \mathcal{C}_p(\gamma) \cdot \left(1+\|\|x\|^{p-1} +\|y-x\|^{p-1} \|_{L^q(\gamma)}\right) \leq C \delta
\end{align*}
using again H\"older's inequality for $q=p/(p-1)$ and the growth assumption on $\nabla_x f$. Note that $C$ depends only on $\mu$ and $\delta$, but not on $\alpha$. Hence, by $\varepsilon(x)$-strong convexity of $\alpha \mapsto f(x, \alpha)$ we have
\begin{align*}
\int f(y, \alpha^\delta(x))\, \gamma(dx, dy) - \int f(x, \alpha^\star(x))\, \mu(dx)
&\geq \int f(x, \alpha^\delta(x))\, \mu(dx) - \int f(x, \alpha^\star(x))\,\mu(dx) + O(\delta)\\
&\geq \int \langle\nabla_\alpha f(x, \alpha^\star(x)),\alpha^\delta(x) - \alpha^\star(x))(x)\rangle\, \mu(dx) \\
&\quad + \frac{1}{2} \int \varepsilon(x) \|\alpha^\delta(x) - \alpha^\star(x)\|_2^2\, \mu(dx) + O(\delta), \;\; \delta \to 0.
\end{align*}
The first term on the right hand side is non-negative by optimality of $\alpha^\star$. Recall that $\mu(\varepsilon(x) > 0)=1$ by assumption. Since $V(\delta)\to V(0)$ we thus conclude \(\alpha^\delta \to \alpha^\star\) in \(\mu\)-measure by Markov's inequality, which completes the argument and hence the proof.
\end{proof}

To extend the result to the martingale case, we first prove several regularity results, including an extension of Proposition \ref{hemicontinuity} to the correspondence $$x_{1:t-1} \twoheadrightarrow \overline \Pi_\delta^{\mathcal{M}}(\mu_{x_{1:t-1}}, \cdot) := \Big\{\gamma \in \overline \Pi_\delta(\mu_{x_{1:t-1}}, \cdot)\,:\, \int (x - y) \, \gamma(dx, dy) = 0\Big\}.$$ 
These are straightforward and we state them for completeness only.

\begin{proposition}\label{prop:martingale_hemicontinuity}
Let \(\mu \in \mathcal{P}_p(\mathbb{R}^N)\) be a successively $\mathcal{W}_p$--continuous probability measure. Then the correspondence $x_{1:t-1} \twoheadrightarrow \overline \Pi_\delta^{\mathcal{M}}(\mu_{x_{1:t-1}}, \cdot)$ is continuous in $(\mathcal{P}_{p - \varepsilon}(\mathbb{R}^2), \mathcal{W}_{p - \varepsilon})$ for any $\varepsilon > 0$.
\end{proposition}
\begin{proof}[Proof.]
The proof of upper hemicontinuity follows the proof of Proposition \ref{hemicontinuity} line by line. The proof lower hemicontinuity is also similar to the proof of Proposition \ref{prop:correcpondence_lhs}: take any sequence $(x_{1:t-1}^{(n)})_{n\in \mathbb{N}}$ converging to some $x_{1:t-1}\in \mathbb{R}^{t-1}$, and a probability measure \(\pi \in \overline \Pi_\delta^{\mathcal{M}}(\mu_{x_{1:t-1}}, \cdot)\). Consider the optimal transport plan \(\eta_n \in \Pi(\mu_{x_{1:t-1}^{(n)}}, \mu_{x_{1:t-1}})\) for \(\mathcal{W}_p(\mu_{x_{1:t-1}^{(n)}}, \mu_{x_{1:t-1}})\), and define
\[
\pi_n := (x_t^{(n)}, y_t + D_n + \lambda_n(x_t^{(n)} - y_t - D_n))_\# (\eta_n \dot{\oplus} \pi) (dx_t^{(n)}, dx_t, dy_t),
\]
where $D_n := \int (x_t^{(n)} - y_t) \,(\eta_n \dot{\oplus} \pi)(dx_t^{(n)}, dy_t)$ and $1 - \lambda_n := \frac{\mathcal{C}_p(\pi)}{\mathcal{C}_p(\eta_n) + \mathcal{C}_p(\pi) + |D_n|}$ (assuming without loss of generality that $\mathcal{C}_p(\pi)>0$). We aim to show that $\pi_n \in \overline \Pi^{\mathcal{M}}_\delta(\mu_{x_{1:t-1}^{(n)}}, \cdot)$ and $\pi_n \to \pi$ in $(\mathcal{P}_{p}(\mathbb{R}^2), \mathcal{W}_{p})$, which implies convergence in $(\mathcal{P}_{p - \varepsilon}(\mathbb{R}^2), \mathcal{W}_{p - \varepsilon})$:
\begin{itemize}
    \item $\pi_n \in \overline \Pi^{\mathcal{M}}_\delta(\mu_{x_{1:t-1}^{(n)}}, \cdot)$: by \cite[Gluing lemma, p.12] {villani2009optimal} we have $\pi_n \in \Pi(\mu_{x_{1:t-1}^{(n)}}, \cdot)$. The martingale constraint is satisfied, because
    \begin{align*}
    &\int [x_t^{(n)} - (y_t + D_n + \lambda_n (x_t^{(n)} - y_t - D_n))] \, (\eta_n \dot{\oplus} \pi) (dx_t^{(n)}, dy_t)\\
    &= (1 - \lambda_n) \left(\int [x_t^{(n)} - y_t] \, (\eta_n \dot{\oplus} \pi) (dx_t^{(n)}, dy_t) - D_n \right)\\
    &= 0
    \end{align*}
    by the definition of $D_n$. Moreover,
    \begin{align*}
    &\mathcal{C}_p(\pi_n)\\
    &=\left(\int \left|x_t^{(n)} - (y_t + D_n + \lambda_n (x_t^{(n)} - y_t - D_n))\right|^p \, (\eta_n \dot{\oplus} \pi) (dx_t^{(n)}, dy_t)\right)^\frac{1}{p}\\
    &= (1 - \lambda_n) \left(\int \left|x_t^{(n)} - y_t - D_n\right|^p \, (\eta_n \dot{\oplus} \pi) (dx_t^{(n)}, dy_t)\right)^\frac{1}{p}\\
    &\leq (1 - \lambda_n) (\mathcal{C}_p(\eta_n) + \mathcal{C}_p(\pi) + |D_n|)\\
    &\leq \mathcal{C}_p(\pi)\\
    &\leq \delta,
    \end{align*}
    where the first inequality follows from Minkowski's inequality for $L^p(\eta_n \dot{\oplus} \pi)$ and the second and third hold by the definition of $\lambda_n$ and $\pi$ respectively.
    \item To establish convergence, we estimate $\mathcal{W}_p(\pi_n, \pi)$ using the following transport plan:
    $$
    (x_t^{(n)}, y_t + D_n + \lambda_n(x_t^{(n)} - y_t - D_n), x_t, y_t)_\# (\eta_n \dot{\oplus} \pi) (dx_t^{(n)}, dx_t, dy_t) \in \Pi(\pi_n, \pi).
    $$
    The $p$-transportation cost for this plan can be bounded as follows:
    \begin{align*}
    &\mathcal{W}_p(\pi_n, \pi)\\
    &\leq \left(\int (|x_t^{(n)} - x_t|^p + |\lambda_n (x_t^{(n)} - y_t) + (1 - \lambda_n) D_n|)^p \, (\eta_n \dot{\oplus} \pi) (dx_t^{(n)}, dx_t, dy_t)\right)^\frac{1}{p}\\
    &\leq \mathcal{C}_p\left(\eta_n\right) + \lambda_n \left(\int |x_t^{(n)} - y_t|^p \, (\eta_n \dot{\oplus} \pi) (dx_t^{(n)}, dy_t)\right)^\frac{1}{p} + (1 - \lambda_n) |D_n|\\
    &\leq (1 + \lambda_n) \mathcal{C}_p\left(\eta_n\right) + \lambda_n \mathcal{C}_p(\pi) + (1 - \lambda_n) |D_n| \to 0, \;\; n \to +\infty,
    \end{align*}
    where  convergence to zero can be justified as follows: first, $\lambda_n \in (0, 1)$ and $$\mathcal{C}_p(\eta_n) = \mathcal{W}_p(\mu_{x_{1:t-1}^{(n)}}, \mu_{x_{1:t-1}}) \to 0$$ as $\mu$ is successively $\mathcal{W}_p$--continuous. For the second term,  $\lambda_n = \frac{\mathcal{C}_p(\eta_n) + |D_n|}{\mathcal{C}_p(\eta_n) + \mathcal{C}_p(\pi) + |D_n|} \to 0$ as $\mathcal{C}_p(\eta_n) \to 0$ and $|D_n| \to 0$ by Jensen's inequality, because $\mu_{x_{1:t-1}^{(n)}} \to \mu_{x_{1:t-1}}$ in $\mathcal{W}_p$. The third term converges to zero, as $|D_n| \to 0$ and $(1 - \lambda_n)_{n\in \N}$ is bounded.
\end{itemize}
\end{proof}

\begin{lemma}\label{lem:control_martingale}
Let \(p > 1\) and let \(\mu \in \mathcal{P}_p(\mathbb{R}^N)\) be a successively $\mathcal{W}_p$--continuous probability measure, which satisfies \(\int \|x_{t+1:N}\|^p \bar{\mu}_{x_{1:t}}(dx_{t+1:N}) \lesssim 1 + \|x_{1:t}\|^p\). Take a compact set $K \subset \mathbb{R}$, and let \(f: \mathbb{R}^N \times K^N \to \mathbb{R}\) be a continuous function, such that $|f(x, \alpha)| \lesssim 1 + \|x\|^{p - \varepsilon}$ for some $\varepsilon > 0$. Then $V^{\mathcal{M},\delta}_t$ is continuous, and satisfies 
\begin{align}\label{eq:growth_martingale}
|V^{\mathcal{M},\delta}_t(x_{1:t}, y_{1:t}, \alpha_{1:t})| \lesssim 1 + \|x_{1:t}\|^{p - \varepsilon} + \|y_{1:t}\|^{p - \varepsilon}.
\end{align}
\end{lemma}

\begin{proof}
First we note that $|V^{\mathcal{M},\delta}_t|\le |V_t^\delta|$, so that the growth bound \eqref{eq:growth_martingale} follows directly from Lemma \ref{lem:cost_to_go_regularity}.(2).
For continuity, we proceed by backward induction. Assume that $(x_{1:t}, y_{1:t}, \alpha_{1:t}) \mapsto V^{\mathcal{M},\delta}_t(x_{1:t}, y_{1:t}, \alpha_{1:t})$ is continuous. Then  $(x_{1:t-1}, y_{1:t-1}, \alpha_{1:t}, \gamma^t) \mapsto \int V^{\mathcal{M},\delta}_t \, d\gamma^t$ is continuous by \eqref{eq:growth_martingale} and Lemma \ref{sliced_continuity}. By Lemma 
\ref{lem:supremum_closed} we have
\begin{equation}\label{lem:control_martingale:regularity}
\sup_{\gamma^{t} \in \Pi^{\mathcal{M}}_\delta(\mu_{x_{1:t-1}}, \cdot)} \int V^{\mathcal{M},\delta}_t \, d\gamma^t = \sup_{\gamma^{t} \in \overline \Pi^{\mathcal{M}}_\delta(\mu_{x_{1:t-1}}, \cdot)} \int V^{\mathcal{M},\delta}_t \, d\gamma^t.
\end{equation}
Recall that the correspondence $x_{1:t-1} \twoheadrightarrow \overline \Pi_\delta^{\mathcal{M}}(\mu_{x_{1:t-1}}, \cdot)$ is continuous by Proposition \ref{prop:martingale_hemicontinuity}. Therefore, by Berge's maximum theorem \cite[Theorem 17.31]{guide2006infinite}, the mapping $$(x_{1:t-1}, y_{1:t-1}, \alpha_{1:t}) \mapsto \sup_{\gamma^{t} \in \overline \Pi^{\mathcal{M}}_\delta(\mu_{x_{1:t-1}}, \cdot)} \int V^{\mathcal{M},\delta}_t \, d\gamma^t$$ is continuous. Finally, an envelope over compact set of continuous function is continuous by another application of Berge's maximum theorem \cite[Theorem 17.31]{guide2006infinite}, hence $V^{\mathcal{M},\delta}_{t-1}$ is continuous. 
\end{proof}

\begin{proposition}\label{prop:martingale_sensitivity_lagrange_multiplier_continuity}
Let $\mu \in \mathcal{P}_p(\mathbb{R}^N)$ be a successively $\mathcal{W}_p$--continuous probability measure, such that $\int \|x_{t+1:N}\|^p \, \bar{\mu}_{x_{1:t}}(dx_{t+1:N}) \lesssim 1 + \|x_{1:t}\|^p$, and suppose that $f: \mathbb{R}^N \to \mathbb{R}$ is a continuous function, which satisfies $|\nabla f(x)| \lesssim 1 + \|x\|^{p - 1 - \varepsilon}$ for some $\varepsilon > 0$. Then
\begin{align}\label{eq:opt_lambda}
x_{1:t-1} \mapsto \inf_{\lambda_t \in \mathbb{R}} \left\|\int \partial_t f(x) \, \bar{\mu}_{x_{1:t}}(dx_{t+1:N}) + \lambda_t \right\|_{L^q(\mu_{x_{1:t-1}})}.
\end{align}
admits a continuous optimizer $x_{1:t-1} \mapsto \lambda^\star_t (x_{1:t-1})$.
\end{proposition}
\begin{proof}[Proof.]
First, we restrict the optimization problem to a compact subset of $\mathbb{R}$. Indeed, by Minkowski's inequality for $L^q(\mu_{x_{1:t-1}})$ we have
$$
|\lambda_t| - \left\|\int \partial_t f(x) \, \bar{\mu}_{x_{1:t}}(dx_{t+1:N})\right\|_{L^q(\mu_{x_{1:t-1}})} \leq \left\|\int \partial_t f(x) \, \bar{\mu}_{x_{1:t}}(dx_{t+1:N}) + \lambda_t \right\|_{L^q(\mu_{x_{1:t-1}})}
$$
for any $\lambda_t \in \mathbb{R}$, hence \eqref{eq:opt_lambda} is equivalent to
$$
x_{1:t-1}\mapsto \inf_{|\lambda_t| \leq C_{x_{1:t-1}}} \left\|\int \partial_t f(x) \, \bar{\mu}_{x_{1:t}}(dx_{t+1:N}) + \lambda_t \right\|_{L^q(\mu_{x_{1:t-1}})}
$$
for $$C_{x_{1:t-1}} := 2 \left\|\int \partial_t f(x) \, \bar{\mu}_{x_{1:t}}(dx_{t+1:N})\right\|_{L^q(\mu_{x_{1:t-1}})}.$$ We note that the mapping $$(x_{1:t-1}, \lambda_t) \mapsto \left\|\int \partial_t f(x) \, \bar{\mu}_{x_{1:t}}(dx_{t+1:N}) + \lambda_t\right\|_{L^q(\mu_{x_{1:t-1}})}$$ is continuous by Minkowski's inequality for $L^q(\mu_{x_{1:t-1}})$ and Lemma \ref{sliced_continuity} together with the growth condition on $\nabla_x f$. Using the same arguments, it can be checked that the correspondence $x_{1:t-1} \twoheadrightarrow [-C_{x_{1:t-1}}, C_{x_{1:t-1}}]$ is continuous. Hence, existence of minimizers follows from Berge's maximum theorem \cite[Theorem 17.31]{guide2006infinite}.
\end{proof}

\begin{proposition}\label{prop:martingale_sensitivity_duality}
Let $\mu \in \mathcal{P}_p(\mathbb{R}^N)$ be a successively $\mathcal{W}_p$-continuous probability measure, such that $\int \|x_{t+1:N}\|^p \, \bar{\mu}_{x_{1:t}}(dx_{t+1:N}) \lesssim 1 + \|x_{1:t}\|^p$, and suppose that $f: \mathbb{R}^N \to \mathbb{R}$ is a continuous function, which satisfies $\|\nabla f(x)\| \lesssim 1 + \|x\|^{p - 1 - \varepsilon}$ for some $\varepsilon > 0$. Then
\begin{align*}
&\sup_{\substack{\|T_t \|_{L^p(\mu_{x_{1:t-1}})} \leq 1\\\int T_t(x_{t}) \, \mu_{x_{1:t-1}}(dx_t) = 0}} \int \left[\int\partial_t f(x) \, \bar{\mu}_{x_{1:t}}(dx_{t+1:N}) \right] T_t(x_{t}) \, \mu_{x_{1:t-1}}(dx_t)\\
&= \inf_{\lambda_t \in \mathbb{R}} \left\|\int \partial_t f(x) \, \bar{\mu}_{x_{1:t}}(dx_{t+1:N}) + \lambda_t\right\|_{L^q(\mu_{x_{1:t-1}})}
\end{align*}
Moreover, the supremum is attained by some Borel measurable function $x_{1:t} \mapsto T^\star_t(x_{1:t})$.
\end{proposition}
\begin{proof}[Proof.]
First, we rewrite the constraint $\int T_t(x_{t}) \, \mu_{x_{1:t-1}}(dx_t) = 0$ by introducing a Lagrange multiplier: indeed,
\begin{align}\label{eqn:prop:martingale_sensitivity_duality}
&\sup_{\substack{\|T_t\|_{L^p(\mu_{x_{1:t-1}})} \leq 1\\\int T_t(x_{t}) \, \mu_{x_{1:t-1}}(dx_t) = 0}} \int \left[\int\partial_t f(x) \, \bar{\mu}_{x_{1:t}}(dx_{t+1:N}) \right] T_t(x_{t}) \, \mu_{x_{1:t-1}}(dx_t)\nonumber\\
&= \sup_{\|T_t\|_{L^p(\mu_{x_{1:t-1}})} \leq 1} \inf_{\lambda_t \in \mathbb{R}} \int \left[\int\partial_t f(x) \, \bar{\mu}_{x_{1:t}}(dx_{t+1:N}) + \lambda_t\right] T_t(x_{t}) \, \mu_{x_{1:t-1}}(dx_t).
\end{align}
Next we justify interchange of the order of $\inf$ and $\sup$:
\begin{enumerate}
    \item the set of functions $$\{T_t\in L^p(\mu_{x_{1:t-1}})\,:\,\|T_t\|_{L^p(\mu_{x_{1:t-1}})} \leq 1\}$$ is $L^p(\mu_{x_{1:t-1}})$--weakly compact by the Banach-Alaoglu Theorem. Furthermore, for fixed $\lambda_t\in \R$, the map
    $$
    T_t \mapsto \int \left[\int\partial_t f(x) \, \bar{\mu}_{x_{1:t}}(dx_{t+1:N}) +\lambda_t \right] T_t(x_{t}) \, \mu_{x_{1:t-1}}(dx_t)
    $$
    is $L^p(\mu_{x_{1:t-1}})$--weakly continuous by the definition of weak convergence in $L^p(\mu_{x_{1:t-1}})$ and the fact that the function $\int\partial_t f(x) \, \bar{\mu}_{x_{1:t}}(dx_{t+1:N}) \in L^q(\mu_{x_{1:t-1}})$ by the growth condition on $\nabla f(x)$.
    \item $\mathbb{R}$ is convex, and $\lambda_t \mapsto \int \left[\int\partial_t f(x) \, \bar{\mu}_{x_{1:t}}(dx_{t+1:N})+\lambda_t \right] T_t(x_{t}) \, \mu_{x_{1:t-1}}(dx_t)$ is linear.
\end{enumerate}
Hence, by the minimax theorem \cite[Corollary 2]{terkelsen1972some} applied with $X = \{\|T_t\|_{L^p(\mu_{x_{1:t-1}})} \leq 1\}\subseteq L^p(\mu_{x_{1:t-1}})$ and $Y = \mathbb{R}$ we obtain
\begin{align}\label{eq:mart_minmax}
\begin{split}
&\sup_{\|T_t\|_{L^p(\mu_{x_{1:t-1}})} \leq 1} \inf_{\lambda_t \in \mathbb{R}} \int \left[\int\partial_t f(x) \, \bar{\mu}_{x_{1:t}}(dx_{t+1:N}) + \lambda_t\right] T_t(x_{t}) \, \mu_{x_{1:t-1}}(dx_t)\\
&= \inf_{\lambda_t \in \mathbb{R}} \sup_{\|T_t\|_{L^p(\mu_{x_{1:t-1}})} \leq 1}  \int \left[\int\partial_t f(x) \, \bar{\mu}_{x_{1:t}}(dx_{t+1:N}) + \lambda_t\right] T_t(x_{t}) \, \mu_{x_{1:t-1}}(dx_t)\\
&= \inf_{\lambda_t \in \mathbb{R}} \left\|\int \partial_t f(x) \, \bar{\mu}_{x_{1:t}}(dx_{t+1:N}) + \lambda_t\right\|_{L^q(\mu_{x_{1:t-1}})},
\end{split}
\end{align}
where the last step follows from duality between $L^p(\mu_{x_{1:t-1}})$ and $L^q(\mu_{x_{1:t-1}})$.

It remains to show the existence of a Borel measurable optimizer $x_{1:t} \mapsto T^\star_t(x_{1:t})$ satisfying the constraint $\int T^\star_t(x_{1:t}) \, \mu_{x_{1:t-1}}(dx_t) = 0$. For this we first note that
\begin{align}\label{eqn:prop:martingale_sensitivity_duality_extension}
&\sup_{\substack{\|T_t\|_{L^p(\mu_{x_{1:t-1}})} \leq 1\\\int T_t(x_{t}) \, \mu_{x_{1:t-1}}(dx_t) = 0}} \int \left[\int\partial_t f(x) \, \bar{\mu}_{x_{1:t}}(dx_{t+1:N})\right] T_t(x_t) \, \mu_{x_{1:t-1}}(dx_t)\nonumber\\
&= \sup_{\gamma \in \overline \Pi^{\mathcal{M}}_1(\mu_{x_{1:t-1}}, \cdot)} \int \left[\int\partial_t f(x) \, \bar{\mu}_{x_{1:t}}(dx_{t+1:N})\right] (y_t - x_t) \, \gamma(dx_t, dy_t).
\end{align}
Indeed, the ``$\leq$"--inequality holds as
$$
\int \left[\int\partial_t f(x) \, \bar{\mu}_{x_{1:t}}(dx_{t+1:N})\right] T_t(x_t) \, \mu_{x_{1:t-1}} = \int \left[\int\partial_t f(x) \, \bar{\mu}_{x_{1:t}}(dx_{t+1:N})\right] (y_t - x_t) \, \gamma(dx_t, dy_t)
$$
for $\gamma := (x_t, x_t + T_t(x_t))_\# \mu_{x_{1:t-1}} \in \overline \Pi^\mathcal{M}_1(\mu_{x_{1:t-1}}, \cdot)$. The ``$\ge$"--inequality follows from H\"older's inequality: for any $\gamma \in \overline \Pi^\mathcal{M}_1(\mu_{x_{1:t-1}}, \cdot)$ and $\lambda_t \in \mathbb{R}$ we have
\begin{align*}
&\int \left[\int\partial_t f(x) \, \bar{\mu}_{x_{1:t}}(dx_{t+1:N})\right] (y_t - x_t) \, \gamma(dx_t, dy_t)\\
&= \int \left[\int\partial_t f(x) \, \bar{\mu}_{x_{1:t}}(dx_{t+1:N}) + \lambda_t\right] (y_t - x_t) \, \gamma(dx_t, dy_t)\\
&\leq \left\|\int \partial_t f(x) \, \bar{\mu}_{x_{1:t}}(dx_{t+1:N}) + \lambda_t\right\|_{L^q(\mu_{x_{1:t-1}})}.
\end{align*}
Taking infimum over $\lambda_t\in \R$ and recalling \eqref{eq:mart_minmax} concludes the proof of the ``$\geq$"--inequality.

The correspondence $x_{1:t-1} \twoheadrightarrow \overline \Pi^\mathcal{M}_1(\mu_{x_{1:t-1}}, \cdot)$ is continuous by Proposition \ref{prop:martingale_hemicontinuity}, and the mapping
$$
(x_{1:t-1}, \gamma) \mapsto \int \left[\int\partial_t f(x) \, \bar{\mu}_{x_{1:t}}(dx_{t+1:N})\right] (y_t - x_t) \, \gamma(dx_t, dy_t)
$$
is continuous in $\|\cdot\| + \mathcal{W}_{p - \varepsilon}(\cdot)$ by Lemma \ref{sliced_continuity}. Therefore, by Berge's maximum theorem \cite[Theorem 17.31]{guide2006infinite} the argmax correspondence 
\begin{align*}
x_{1:t-1} \twoheadrightarrow \arg \max_{\gamma \in \overline \Pi^{\mathcal{M}}_1(\mu_{x_{1:t-1}}, \cdot)} \int \left[\int\partial_t f(x) \, \bar{\mu}_{x_{1:t}}(dx_{t+1:N})\right] (y_t - x_t) \, \gamma(dx_t, dy_t)
\end{align*}
 is continuous. We also claim that the argmax is unique for all $x_{1:t-1}\in \R^{t-1}$ satisfying 
 \begin{align}\label{eq:non-degeneracy}
 \mu_{x_{1:t-1}}\left(\int\partial_t f(x) \, \bar{\mu}_{x_{1:t}}(dx_{t+1:N}) \neq 0\right) > 0.
 \end{align}
 Indeed, suppose that $\gamma_1, \gamma_2 \in \overline \Pi^\mathcal{M}_1(\mu_{x_{1:t-1}}, \cdot)$ attain the supremum and $\gamma_1 \neq \gamma_2$. Define
$$
\pi := \left[(y_t^{(1)}, x_t)_\# \gamma_1(dx_t, dy_t^{(1)})\right] \dot{\oplus} \,\gamma_2(dx_t, dy_t^{(2)}), \;\; \gamma := (x_t, \frac{1}{2}(y_t^{(1)} + y_t^{(2)}))_\# \pi.
$$
Clearly, $\gamma$ attains the supremum too, and
$$
\int (x_t - y_t) \, \gamma(dx_t, dy_t) = \frac{1}{2} \int (x_t - y_t^{(1)}) \, \gamma_1(dx_t, dy_t^{(1)}) + \frac{1}{2} \int (x_t - y_t^{(2)}) \, \gamma_2(dx_t, dy_t^{(2)}) = 0.
$$
Moreover, by strict convexity of $L^p(\pi)$ for $p>1$ and Minkowski's inequality we have
$$
\mathcal{C}_p(\gamma) < \frac{1}{2}(\mathcal{C}_p(\gamma_1) + \mathcal{C}_p(\gamma_2)) \leq 1, 
$$
so that $\gamma \in \Pi_1^{\mathcal{M}}(\mu_{x_{1:t-1}}, \cdot).$
Since $\gamma$ is the optimizer, we must have $\mathcal{C}_p(\gamma) > 0$ by \eqref{eq:non-degeneracy}, and
$$
\int \left[\int\partial_t f(x) \, \bar{\mu}_{x_{1:t}}(dx_{t+1:N})\right] (y_t - x_t) \, \gamma(dx_t, dy_t) > 0.
$$
Hence, the transport plan $\widehat{\gamma} := (x_t, x_t + \frac{1}{\mathcal{C}_p(\gamma)}(y_t - x_t))_\# \gamma$ is well-defined, belongs to $\overline \Pi^\mathcal{M}_1(\mu_{x_{1:t-1}}, \cdot)$, and
\begin{align*}
&\int \left[\int\partial_t f(x) \, \bar{\mu}_{x_{1:t}}(dx_{t+1:N})\right] (y_t - x_t) \, \widehat{\gamma}(dx_t, dy_t)\\
&= \frac{1}{\mathcal{C}_p(\gamma)} \int \left[\int\partial_t f(x) \, \bar{\mu}_{x_{1:t}}(dx_{t+1:N})\right] (y_t - x_t) \, \gamma(dx_t, dy_t)\\
&> \int \left[\int\partial_t f(x) \, \bar{\mu}_{x_{1:t}}(dx_{t+1:N})\right] (y_t - x_t) \, \gamma(dx_t, dy_t),
\end{align*}
which contradicts the optimality assumption. Therefore, the argmax is unique on the set
$$
U := \left\{x_{1:t-1} \in \R^{t-1}\,:\, \mu_{x_{1:t-1}}\left\{\int\partial_t f(x) \, \bar{\mu}_{x_{1:t}}(dx_{t+1:N}) \neq 0\right\} > 0\right\},
$$
and is equal to $\overline \Pi^\mathcal{M}_1(\mu_{x_{1:t-1}}, \cdot)$ on $\R^{t-1} \setminus U$, as
\begin{align*}
\int \left[\int\partial_t f(x) \, \bar{\mu}_{x_{1:t}}(dx_{t+1:N})\right] (y_t - x_t) \, \gamma(dx_t, dy_t)=0    
\end{align*}
for all $\gamma \in \overline \Pi^\mathcal{M}_1(\mu_{x_{1:t-1}}, \cdot)$
in this case. The set $U$ is Borel, since the mapping
$$
x_{1:t-1} \mapsto \mu_{x_{1:t-1}}\left\{\int\partial_t f(x) \, \bar{\mu}_{x_{1:t}}(dx_{t+1:N}) \neq 0\right\}
$$
is Borel by \cite[Corollary 7.26.1]{bertsekas1996stochastic}. The argmax on $U$ has the form $(x_t, x_t + T_{x_{1:t-1}}(x_t))_\# \mu_{x_{1:t-1}}$, because the supremum in \eqref{eqn:prop:martingale_sensitivity_duality} is attainable by some Borel function $x_t \mapsto T_{x_{1:t-1}}(x_t)$ (recall point (1) above). Hence, we define $T^\star_t$ as follows:
$$
T^\star_t(x_{1:t}) := \begin{cases}
    T_{x_{1:t-1}}(x_t), \;\; x_{1:t-1} \in U\\
    0, \;\; \text{otherwise.}
\end{cases}
$$
This concludes the proof.

\end{proof}

\begin{proof}[Proof of Corollary \ref{dro:control_martingale_sensitivity}.]
We only prove the case without controls. To extend the result to the controlled case, we proceed in the same way as in Theorem \ref{dro:control_sensitivity}. 

Recall that by Corollary \ref{cor:linear} we have
\begin{align}\label{eqn:dro:control_martingale_sensitivity:dpp}
V^{\mathcal{M}}(\delta) &= V^{\mathcal{M},\delta}_0\nonumber\\
&= \sup_{\gamma^1 \in \Pi^{\mathcal{M}}_\delta(\mu^1, \cdot)} \int \ldots \sup_{\gamma^N \in \Pi_\delta^{\mathcal{M}}(\mu_{x_{1:N-1}}, \cdot)} \int f(y) \, d\gamma^N \ldots d\gamma^1\nonumber\\
&= \sup_{\gamma^1 \in \Pi_\delta(\mu^1, \cdot)} \inf_{\lambda_1 \in \mathbb{R}} \int \ldots \sup_{\gamma^N \in \Pi_\delta(\mu_{x_{1:N-1}}, \cdot)} \inf_{\lambda_N \in \mathbb{R}} \int f(y) + \sum_{t = 1}^N \lambda_t (x_t - y_t) \, d\gamma^N \ldots d\gamma^1,
\end{align}
where we have used Lagrange multipliers to enforce the martingale constraint. To obtain the upper bound, we apply Proposition \ref{prop:martingale_sensitivity_lagrange_multiplier_continuity} to find continuous maps $x_{1:t-1} \mapsto \lambda^\star(x_{1:t-1})$. Using  \eqref{eqn:dro:control_martingale_sensitivity:dpp} we estimate
$$
V^{\mathcal{M}}(\delta) \leq \sup_{\gamma^1 \in \Pi_\delta(\mu^1, \cdot)} \int \ldots \sup_{\gamma^N \in \Pi_\delta(\mu_{x_{1:N-1}}, \cdot)} \int f(y) + \sum_{t = 1}^N \lambda^\star_t (x_t - y_t) \, d\gamma^N \ldots d\gamma^1.
$$
We now copy the proof of the upper bound in Theorem \ref{dro:control_sensitivity} line by line. Using Proposition \ref{prop:martingale_sensitivity_lagrange_multiplier_continuity} we find
\begin{align*}
\Upsilon^{\mathcal{M}} &\leq \inf_{\lambda_1 \in \mathbb{R}} \left\|\int \partial_{1} f(x) \,\bar{\mu}_{x_{1}}(dx_{2:N}) + \lambda_1 \right\|_{L^q(\mu^1)}\\
&+ \sum_{t = 2}^{N} \int \inf_{\lambda_t \in \mathbb{R}} \left\|\int \partial_{t} f(x) \,\bar{\mu}_{x_{1:t}}(dx_{t+1:N}) + \lambda_t \right\|_{L^q(\mu_{x_{1:t-1}})} \,\mu(dx_{1:t-1}).
\end{align*}
To get the lower bound, we make a specific choice of $(x_{1:t-1}, y_{1:t-1}) \mapsto \gamma_{x_{1:t-1}, y_{1:t-1}} \in \Pi_\delta^{\mathcal{M}}(\mu_{x_{1:t-1}}, \cdot)$ similarly to the Theorem \ref{dro:control_sensitivity}. In order to achieve this, we use Proposition \ref{prop:martingale_sensitivity_duality} to obtain Borel measurable mappings $x_{1:t} \mapsto T_t(x_{1:t})$ satisfying $\|T_t(x_{1:t-1}, \cdot)\|_{L^p(\mu_{x_{1:t-1}})}\le 1$ and $\int T_t(x_{1:t}) \, \mu_{x_{1:t-1}}(dx_t) = 0$, such that
\begin{align*}
\int \left[\int\partial_t f(x) \, \bar{\mu}_{x_{1:t}}(dx_{t+1:N}) \right] T_t(x_{1:t}) \, \mu_{x_{1:t-1}}(dx_t) = \inf_{\lambda_t \in \mathbb{R}} \left\|\int \partial_t f(x) \, \bar{\mu}_{x_{1:t}}(dx_{t+1:N}) + \lambda_t\right\|_{L^q(\mu_{x_{1:t-1}})}.
\end{align*}
Then we set $\gamma^{\delta}_{x_{1:t-1}, y_{1:t-1}} := (x_t, x_t + \bar \delta T_t(x_{1:t}))_\# \mu_{x_{1:t-1}} $ for $\bar \delta := \delta/(1 + \delta)$. By definition, $\gamma^\delta_{x_{1:t-1}, y_{1:t-1}} \in \Pi^\mathcal{M}_\delta(\mu_{x_{1:t-1}}, \cdot)$, hence we obtain the following lower bound:
\begin{align*}
V^{\mathcal{M}}(\delta) - V^{\mathcal{M}}(0) \geq \int f(y) - f(x) \, \gamma^\delta(dx, dy),
\end{align*}
where we set $\gamma^\delta := \gamma^{1, \delta} \otimes \ldots \otimes \gamma^{\delta}_{x_{1:N-1}, y_{1:N-1}}$. Using the dominated convergence theorem and the growth condition of $f$ we obtain 
\begin{align*}
\Upsilon^\mathcal{M} &\geq \inf_{\lambda_1 \in \mathbb{R}} \left\|\int \partial_{1} f(x) \,\bar{\mu}_{x_{1}}(dx_{2:N}) + \lambda_1 \right\|_{L^q(\mu^1)}\\
&+ \sum_{t = 2}^{N} \int \inf_{\lambda_t \in \mathbb{R}} \left\|\int \partial_{t} f(x) \,\bar{\mu}_{x_{1:t}}(dx_{t+1:N}) + \lambda_t \right\|_{L^q(\mu_{x_{1:t-1}})} \,\mu(dx_{1:t-1}).
\end{align*}
\end{proof}

\bibliographystyle{siam}
\bibliography{bib}

\end{document}